\documentclass[reqno, 12pt]{amsart} 

\usepackage{microtype} 
\usepackage{amsfonts,amsthm,amsmath,amssymb,amscd,mathrsfs} 

\usepackage{latexsym} 
\usepackage[colorlinks=true,linkcolor=blue,citecolor=red,urlcolor=black]{hyperref} 
\usepackage{graphics} 
\usepackage{indentfirst} 
\usepackage{cite} 
\usepackage[dvips]{epsfig} 
\usepackage{color} 
\usepackage{lmodern} 
\usepackage[numbers,sort&compress]{natbib}
\usepackage{soul} 

\theoremstyle{plain} 
\newtheorem{Thm}{Theorem}[section] 
\newtheorem{Lem}[Thm]{Lemma}     
\newtheorem{Prop}[Thm]{Proposition}
\newtheorem{Cor}[Thm]{Corollary}

\theoremstyle{definition}

\theoremstyle{remark}
\newtheorem{Rem}[Thm]{Remark}

\numberwithin{equation}{section} 

\newcommand{\beq}{\begin{equation}}
\newcommand{\eeq}{\end{equation}}
\newcommand{\ben}{\begin{eqnarray}}
\newcommand{\een}{\end{eqnarray}}
\newcommand{\beno}{\begin{eqnarray*}}
\newcommand{\eeno}{\end{eqnarray*}}

\newcommand{\lt}{\left}
\newcommand{\rt}{\right}

\newcommand{\px}{\partial_x}
\newcommand{\py}{\partial_y}
\newcommand{\pz}{\partial_z}
\newcommand{\pt}{\partial_t}
\newcommand{\abs}[1]{\lvert#1\rvert}  

\begin{document}
	
\title[\vspace{-0.2cm}Global bounded solutions of the 3D PKS-NS system]{Global bounded solutions \\of the 3D Patlak-Keller-Segel-Navier-Stokes system\\ via Couette flow and logistic source}

\author{Shikun~Cui}
\address[Shikun~Cui]{School of Mathematical Sciences, Dalian University of Technology, Dalian, 116024,  China}
\email{cskmath@163.com}

\author{Lili~Wang}
\address[Lili~Wang]{School of Mathematical Sciences, Dalian University of Technology, Dalian, 116024,  China}
\email{wllmath@163.com}

\author{Wendong~Wang}
\address[Wendong~Wang]{School of Mathematical Sciences, Dalian University of Technology, Dalian, 116024,  China}
\email{wendong@dlut.edu.cn}

\author{Juncheng~Wei}
\address[Juncheng~Wei]{Department of Mathematics, Chinese University of Hong Kong, Shatin, NT, Hong Kong}
\email{wei@math.cuhk.edu.hk}

\author{Guoxu~Yang}
\address[Guoxu~Yang]{School of Mathematical Sciences, Dalian University of Technology, Dalian, 116024,  China}
\email{guoxu\_dlut@outlook.com}

\maketitle
\begin{abstract}
	As is well-known, the solutions to the Patlak-Keller-Segel system in 3D may blow up in finite time regardless of any initial cell mass. In this paper, we investigate the existence of global bounded solutions of the 3D Patlak-Keller-Segel-Navier-Stokes system with a large initial cell mass via Couette flow or logistic source in a finite channel. On the one hand, it is proved that as long as the Couette flow is strong enough  and the initial velocity is small,  
the bounded solutions of the system are global in time without any limitation on the initial mass $M$ if the logistic source term exists. On the other hand, if the logistic source term vanishes, it is  proved that as long as the Couette flow is strong enough and the initial velocity is small,  
the solutions with a finite initial mass $M$, whose lower bound is $ \left(\frac{8\pi}{9}\right)^-$,   are global in time.
\end{abstract}


\section{Introduction}    
   Consider the following 3D parabolic-elliptic Patlak-Keller-Segel (PKS) system with logistic source,  coupled with the Navier-Stokes (NS) equations, in a finite channel $\mathbb{T} \times \mathbb{I} \times \mathbb{T}$ with $\mathbb{T}=[0,2 \pi)$ and $\mathbb{I}=[-1,1]$ :
\begin{equation}\label{ini}
\left\{\begin{array}{l}
	\partial_t n  +  v \cdot \nabla n= \Delta n-\nabla \cdot(n \nabla c)-\mu n^{2}, \\
	\Delta c+n-c=0, \\
	\partial_t v+v \cdot \nabla v+\nabla P= \Delta v+n \nabla \phi, \\ \nabla \cdot v=0,
\end{array}\right.
\end{equation}
along with initial conditions
$$(n,v)|_{t=0}=(n_{\rm in},v_{\rm in}).$$
Here, $n$ represents the cell density, $c$ the chemoattractant concentration, $v$ the fluid velocity, $\mu$ a non-negative constant related to the logistic source, $P$ the pressure, and $\phi$ the potential function influencing the fluid dynamics.

When $v=0$, $\phi=0$ and $\mu=0$, the system (\ref{ini}) is reduced to the classical 3D parabolic-elliptic PKS system. The classical PKS system serves as a classical mathematical framework for modeling aggregation behavior driven by both random motion and chemotaxis, which was jointly developed by Patlak \cite{Patlak1}, Keller and Segel \cite{Keller1}. Comprising two core equations, it characterizes the dynamics of individual density under the combined influence of diffusion and chemotaxis, thereby elucidating how chemical signals govern collective aggregation. 
For the one-dimensional PKS system, all its solutions are globally well-posed. 
In spatial dimensions exceeding one, solutions of the classical PKS system may exhibit finite-time blow-up. For the two-dimensional parabolic-parabolic system, a critical mass threshold emerges at $8\pi$: when the initial cell mass $M:=\|n_{\rm in}\|_{L^{1}}$ satisfies $M<8\pi$, the solutions remain global in time \cite{Calvez1}, whereas $M>8\pi$ induces finite-time blow-up \cite{BDDM2023,CGMN2022,Schweyer1}. This criticality extends to the parabolic-elliptic case, where Wei \cite{wei11} established global well-posedness precisely when $M\leq 8\pi$. When the spatial dimension is greater than two, regardless of parabolic-elliptic form or parabolic-parabolic form, the finite-time blow-up may occur for arbitrarily small initial mass (see \cite{Na2000, SW2019} for the parabolic-elliptic form, and see \cite{winkler1} for the parabolic-parabolic form). More results on this topic can be found in \cite{BCM2008,DDDMW,TW2017} and the related ones.

As said in \cite{zeng}:
{\it An interesting question is to consider whether the stabilizing effect of the moving fluid can suppress the finite time blow-up?} 

Recently, for the 3D PKS system with a time-dependent shear flow, He \cite{he24-2} demonstrated that when the total
mass of the cell density is below a specific threshold ($8\pi|\mathbb{T}|$), the solution remains globally regular in $\mathbb{T}^3$ as long as the
flow is sufficiently strong.  For the 3D PKS system coupled with the Navier-Stokes flow:
\begin{equation}\label{ini1}
	\left\{
	\begin{array}{lr}
		\partial_tn+Ay\partial_x n+u\cdot\nabla n-\triangle n=-\nabla\cdot(n\nabla c), \\
		\triangle c+n-\bar{n}=0, \\
		\partial_tu+Ay\partial_x u+\left(
		\begin{array}{c}
			Au_2 \\
			0 \\
			0 \\
		\end{array}
		\right)
		-\triangle u+u\cdot\nabla u+\nabla P=\left(
		\begin{array}{c}
			n \\
			0 \\
			0 \\
		\end{array}
		\right), \\
		\nabla \cdot u=0,
	\end{array}
	\right.
\end{equation}
Cui-Wang-Wang-Wei \cite{CWW20251} proved the following theorem.
\begin{Thm}[Cui-Wang-Wang-Wei, \cite{CWW20251}]\label{result0}
	Assume that $0<n_{\rm in}(x,y,z)\in H^{2}(\mathbb{T}^{3})$ and $u_{\rm in}(x,y,z)\in H^{2}(\mathbb{T}^{3}).$
	There exist a sufficiently small positive constant $\epsilon$ depending on
	$\|n_{\rm in}\|_{ H^{2}(\mathbb{T}^{3})}$ and $\|(u_{\rm in})_{\neq}\|_{H^{2}(\mathbb{T}^{3})},$
	and a positive constant $A_{1}$ depending on $\|n_{\rm in}\|_{ H^{2}(\mathbb{T}^{3})}$ and $\|u_{\rm in}\|_{H^{2}(\mathbb{T}^{3})}$,
	such that if $A\geq A_{1},$
	\begin{equation}\label{conditions:u20 u30}
		\|(u_{2,\rm in})_{0}\|_{H^{2}}+\|(u_{3,\rm in})_{0}\|_{H^{1}} \leq \epsilon
	\end{equation}
	 and \begin{equation}\label{condition: M}
		M=\int_{\mathbb{T}^{3}} n_{\rm in}dxdydz< 16\pi^{2}.
	\end{equation}
	Then the solution of \eqref{ini1} is global in time.
\end{Thm}
The index $16\pi^2$ in \eqref{condition: M} appears to be the sharp threshold for initial cell mass.
	Recall that the zero mode $ n_{0} $ satisfies 
	\begin{equation*}\label{eq:n000}
		\begin{aligned}
			\partial_{t}n_{0}=&\frac{1}{A}\triangle n_{0}-\frac{1}{A}\nabla\cdot(n_{0}\nabla c_{0})-\frac{1}{A}\nabla\cdot(n_{\neq}\nabla c_{\neq})_{0}-\frac{1}{A}(u_{0}\cdot\nabla n_{0})-\frac{1}{A}(u_{\neq}\cdot\nabla n_{\neq})_{0}.
		\end{aligned}
	\end{equation*}
	When the velocity $u$ vanishes and $n(t,x,y,z)=n(t,y,z)$, the critical mass is $8\pi$, which is similar to the 2D Keller-Segel system. With an additional advection term modeling ambient fluid flow, Kiselev-Xu \cite{Kiselev1} proved that for any initial data, there exist incompressible fluid flows such that the solution to the  2D or 3D Keller-Segel system stays globally regular. Bedrossian-He \cite{Bedro2} studied the suppression of blow-up for the parabolic-elliptic PKS system by shear flows in $\mathbb{T}^3$ and $\mathbb{T}\times\mathbb{R}^2$. He \cite{he0} investigated the suppression of blow-up for the parabolic-parabolic PKS model near the large strictly monotone shear flow in $\mathbb{T}\times\mathbb{R}$. 
Feng-Shi-Wang \cite{Feng1}  suppressed the blow-up for the advective Kuramoto-Sivashinsky and
the Keller-Segel equations via the planar helical flows. Shi-Wang \cite{wangweike2} considered the suppression effect of the Couette-Poiseuille flow $(z,z^2,0)$ in $\mathbb{T}^2\times\mathbb{R}$.  For more references, we refer to 
\cite{wangweike1,he24-1,he24-2} and the references therein.

Little progress has been made for the coupled PKS-NS system. Zeng-Zhang-Zi \cite{zeng} initiated the study of the 2D PKS-NS system near the Couette flow in $\mathbb{T}\times\mathbb{R}$, and
proved that the solution stays globally regular if the Couette flow is sufficiently strong. He \cite{he05} investigated the blow-up suppression for the parabolic-elliptic PKS-NS system in $\mathbb{T}\times\mathbb{R}$ with the coupling of buoyancy effects for a class of initial data with small vorticity.
Wang-Wang-Zhang \cite{Wanglili}  proved the blow-up suppression of 2D supercritical parabolic-elliptic PKS-NS system near the Couette flow in $\mathbb{T}\times\mathbb{R}.$
\textcolor[rgb]{0,0,0}{Li-Xiang-Xu \cite{Li0} suppressed the blow-up for the PKS-NS system via the Poiseuille flow  in $\mathbb{T}\times\mathbb{R},$
	and Cui-Wang \cite{cui1} considered the blow-up suppression for the PKS-NS system in $\mathbb{T}\times \mathbb{I}$ with Navier-slip boundary condition.
	In addition, Hu-Kiselev-Yao \cite{Hu0} considered the Patlak-Keller-Segel equation coupled
with a fluid flow that obeys Darcy’s law for incompressible porous media via
buoyancy force. Hu \cite{Hu2023} proved that sufficiently large buoyancy can also suppress the blow-up of the PKS system, see also the recent results by Hu-Kiselev \cite{Hu1}.} 
For the 3D PKS-NS system, fewer results are obtained. Cui-Wang-Wang \cite{CWW1} considered the PKS system coupled with
the linearized NS equations near $(Ay,0,0)$ in  $\mathbb{T}\times\mathbb{I}\times\mathbb{T},$ and showed that
 the solutions are global in time as long as  $M< \frac{8\pi}{9}$ and  $A$ is big enough.
Also Cui-Wang-Wang \cite{CWW2025} studied the blow-up suppression and the nonlinear stability of the PKS-NS system for the Couette flow $(Ay,0,0)$ in $\mathbb{T}\times\mathbb{R}\times\mathbb{T},$ and proved that
the solutions are global in time as long as the initial cell mass $M< \frac{24}{5}\pi^2$. For more progress, we refer to \cite{CWW20251} and the references therein.



When $v=0$ and $\phi=0,$ the system (\ref{ini}) becomes the 3D parabolic-elliptic PKS system with a quadratic logistic source term. The quadratic term $-\mu n^{2}$ introduces ecological carrying capacity, preventing infinite densities and aligning solutions with observed finite-size biological structures (e.g., bacterial colonies, tumor spheroids). At present, significant advances have been achieved in the analysis of the PKS system with a quadratic logistic source. Tello-Winkler \cite{TeM2007} established unique globally bounded solutions for arbitrary initial data when the spatial dimension $d\leq 2$ or $d\geq 3$ and $\mu>\frac{d-2}{d}$. Winkler \cite{Winkler0} proved global boundedness in the parabolic-parabolic case for $d\geq 3$ with sufficiently large $\mu$.  Tao-Winkler proved that the blow-up for the PKS-NS system can be suppressed by logistic source \cite{TW2016} in a 2D smooth bounded domain. 
 Xiang \cite{X2018,X20181} also showed that logistic source can prevent blow-up. 
 Fuest \cite{Fu2021} constructed some counterexample functions, which blow up at finite time when $\mu\in(0,\frac{d-4}{d})$ and the spacial dimension $d\geq 5$. Fuest-Lankeit-Mizukami \cite{FLM2025} studied the location of blow-up points in fully parabolic PKS system with spatially heterogeneous logistic source. For weak solutions, 
Lankeit \cite{L2015} demonstrated the existence of weak solutions in  for all $\mu>0$ in a smooth bounded convex domain $\Omega\subset\mathbb{R}^{d}$ with $d\geq 3$.
Tian-Xiang \cite{TX2025} established a result on global existence of weak solutions for the 3D parabolic-parabolic PKS-NS system with subquadratic logistic degradation in a bounded domain with smooth boundary. For more references, we refer to \cite{WMZ2014, GS2016, SS2017, SS2018_1, SS2018_2, zy2018, winkler2019, ILV2024, WCZ2025}. It is still unknown whether the bounded solutions will blow up in finite time for the three-dimensional space with $\mu\leq \frac13$. Next, we will investigate this issue.

Consider the nonlinear stability of the PKS-NS system (\ref{ini}) near the 3D Couette flow $( Ay,0,0 )$. Introduce a perturbation $u$ around the Couette flow $(Ay, 0, 0)$, which $u(t,x,y,z)=v(t,x,y,z)-( Ay,0,0 )$ satisfies $u |_{t=0}=u_{\rm in}=(u_{1,\rm in}, u_{2,\rm in}, u_{3,\rm in})$. Assuming $\phi=x$, we rewrite system (\ref{ini}) as
\begin{equation}\begin{aligned}
		\label{eq:main_ori}
		\left\{\begin{array}{l}
			\partial_t n+A y \partial_x n+u \cdot \nabla n- \Delta n=-\nabla \cdot(n \nabla c)-\mu n^{2}, \\
			\Delta c+n-c=0, \\
			\partial_t u+A y \partial_x u+Au_2e_1-\Delta u+u \cdot \nabla u+\nabla P^{N_1}+\nabla P^{N_2}+\nabla P^{N_3}=ne_1, \\
			\nabla \cdot u=0,
		\end{array}\right.
\end{aligned}\end{equation}
together with initial and boundary conditions
\begin{equation}\begin{aligned}\label{eq:boundary}
		\left\{\begin{array}{ll}
		(n,u)|_{t=0}=(n_{\rm in},u_{\rm in}),\\
		u(t,x,\pm1,z) = n(t,x,\pm1,z) = c(t,x,\pm1,z)=0.\end{array}\right.
\end{aligned}\end{equation}
Here, $e_1=(1,0,0)^{T}$ and the new pressures $P^{N_1}$, $P^{N_2}$ and $P^{N_3}$ are determined by
\begin{equation}\begin{aligned} \label{eq:pressure}
		\left\{\begin{array}{l}
			\Delta P^{N_1}=-2 A \partial_x u_2, \\
			\Delta P^{N_2}=\partial_x n, \\
			\Delta P^{N_3}=-\operatorname{div}(u \cdot \nabla u).
		\end{array}\right.
\end{aligned}\end{equation}
After time scaling $t \mapsto \frac{t}{A}$, it holds 
\begin{equation}\begin{aligned} \label{eq:main}
		\left\{\begin{array}{l}
			\partial_t n+y \partial_x n+\frac{1}{A} u \cdot \nabla n-\frac{1}{A} \Delta n=-\frac{1}{A} \nabla \cdot(n \nabla c)-\frac{1}{A}\mu n^{2}, \\
			\Delta c+n-c=0, \\
			\partial_t u+y \partial_x u+u_2e_1-\frac{1}{A} \Delta u+\frac{1}{A} u \cdot \nabla u+\frac{1}{A}\left(\nabla P^{N_1}+\nabla P^{N_2}+\nabla P^{N_3}\right)=\frac{n}{A}e_1, \\
			\nabla \cdot u=0.
		\end{array}\right.
\end{aligned}\end{equation}

Our main result is stated as follows.

\begin{Thm}\label{Thm:main}
	Assume that the non-negative initial cell density $ n_{\rm in} \in L^{\infty}\cap H^1(\mathbb{T} \times \mathbb{I} \times \mathbb{T})$ and the initial velocity $ u_{\rm in} \in H^2(\mathbb{T} \times \mathbb{I} \times \mathbb{T})$.
	There exists a positive constant $\mathcal{B}$ depending on $\|n_{\rm in}\|_{L^{\infty}\cap H^1(\mathbb{T} \times \mathbb{I} \times \mathbb{T})}$ and $\|u_{\rm in}\|_{H^2(\mathbb{T} \times \mathbb{I} \times \mathbb{T})}$, such that if $A \geq \mathcal{B}$ and
	\begin{equation}\begin{aligned}\label{eq:smallness of u_in}
			A^{\epsilon_{1}}\|u_{\rm in}\|_{H^2(\mathbb{T} \times \mathbb{I} \times \mathbb{T})} \leq C_0,
	\end{aligned}\end{equation}
	where $C_0$ is a positive constant and $\epsilon_{1}>\frac23$, then the solution of \eqref{eq:main} with the initial  boundary data \eqref{eq:boundary} is global in time, with the following two cases:
	
	\noindent
	\begin{itemize}
		\item Case I (With Logistic Source):  
	If \(\mu > 0\), no additional conditions are required.
	
	\noindent
	\item Case II (Without Logistic Source):  
	If \(\mu = 0\), the initial cell mass \(M = \|n_{\rm in}\|_{L^1}\) additionally satisfies \(M < \frac{8\pi}{9}\).
\end{itemize}
\end{Thm}

\begin{Rem} Even if the velocity $u$ of the system  \eqref{eq:main} vanishes, it is still unknown whether the bounded solutions of this system exist globally according to the counterexample of Fuest in \cite{Fu2021}. The above theorem shows that the bounded solutions of this system will  exist globally for large mass initial  data around the Couette flow. Moreover, the result establishes the global existence of solutions to the 3D PKS-NS system for the first time without any limitation for $\mu$.
	Besides, in the absence of the logistic source term $(\mu=0)$, our result provides a nonlinear extension of Theorem 1.2 in \cite{CWW1}. At this point, we need to overcome the great challenges from nonlinear interactions and the 3D lift-up effect.
\end{Rem}
\begin{Rem} The boundary data \eqref{eq:boundary}  indicates that the cell density (number of cells per unit area/volume) is fixed at the boundary and 
 the concentration of the chemical signal (the chemotactic substance secreted by cells) is fixed at the boundary. It corresponds to the actual boundary states in experiments or physiology. For example, when studying cell chemotaxis in a Petri dish, if no additional cells are supplied at the dish’s edge,  Dirichlet boundary condition for ``cell density" can be set to $0$ (indicating no cells at the boundary); alternatively, if the chemical signal concentration at the boundary is controlled to be constant via laboratory techniques,  Dirichlet boundary condition is used to fix this concentration.
\end{Rem}
\begin{Rem}
	In fact, the upper bound $\frac89\pi$ of $M$ in Theorem \ref{Thm:main} can be replaced by $\frac{2\pi}{C_{*}^3}$ (see around \eqref{eq:temp0.2} for more details), where $C_{*}$ is the sharp Sobolev constant of
		$$
		\left\{\begin{array}{l}
				\|f(y,z)\|_{L^3(\mathbb{I} \times \mathbb{T})} \leq C_{*}\|f(y,z)\|_{L^1(\mathbb{I} \times \mathbb{T})}^{\frac{1}{3}}\|\nabla f(y,z)\|_{L^2(\mathbb{I} \times \mathbb{T})}^{\frac{2}{3}}, \\
				\left.f(y, z)\right|_{y= \pm 1}=0, \quad \int_{\mathbb{T}} f(y, \cdot) d z=0.
			\end{array}\right.
		$$
		By Lemma A.3 in \cite{CWW1}, we can set $C_{*}^{3}=\frac94$, which provides an exact upper bound for the assumption on the initial mass $M$ in Theorem \ref{Thm:main}.
\end{Rem}
\begin{Rem}
  In Theorem \ref{Thm:main}, the smallness condition on $\|u_{\rm in}\|_{H^2(\mathbb{T} \times \mathbb{I} \times \mathbb{T})}$ arises from the well-known 3D lift-up effect. 
  To overcome the 3D lift-up effect,  we need to take more factors into consideration, and a potential approach would be to employ the method of freezing coefficient proposed by \cite{CWZ1}.
\end{Rem}

\begin{Rem}\label{local}
The local well-posedness of the system (\ref{eq:main}) follows from standard arguments, as detailed in \cite{CWW2025, Hu1, winkler1}, and is therefore omitted.
\end{Rem}

Here are some notations used in this paper.

\noindent\textbf{Notations}:
\begin{itemize}
	\item 
	For a given function $f(t,x,y,z)$, its Fourier transform can be defined by
	\begin{equation}
		f(t,x,y,z)=\sum_{k_{1}, k_{3}\in\mathbb{Z}}f^{k_{1},k_{3}}(t,y){\rm e}^{i(k_{1}x+k_{3}z)}, \nonumber
	\end{equation}
	where 
	$f^{k_{1},k_{3}}(t,y)=\frac{1}{|\mathbb{T}|^{2}}\int_{\mathbb{T}\times\mathbb{T}}f(t,x,y,z){\rm e}^{-i(k_{1}x+k_{3}z)}dxdz$. In addition, we denote $\eta :=\left(k_{1}^{2}+k_{3}^{2} \right)^{\frac12}$.
	
	\item 
	For a given function $f=f(t,x,y,z)$, define the following modes:
	$$P_{0}f=f_{0}=\frac{1}{|\mathbb{T}|}\int_{\mathbb{T}}f(t,x,y,z)dx,\quad P_{\neq}f=f_{\neq}=f-f_{0},$$
	$$P_{(0,0)}f=f_{(0,0)}=\frac{1}{|\mathbb{T}|^{2}}\int_{\mathbb{T}\times\mathbb{T}}f(t,x,y,z)dxdz,\quad P_{(0,\neq)}f=f_{(0,\neq)}=f_{0}-f_{(0,0)}.  $$ 
	Throughout this paper, $f_0$ and $f_{\neq}$ denote the zero mode and non-zero modes of $f$, respectively. Moreover, $f_{(0,0)}$ and $f_{(0,\neq)}$ represent the $z$-part zero mode and non-zero mode of $f_0$, respectively. Specifically, $u_{j,0}$ and $u_{j,\neq}$ refer to the zero and non-zero modes of the velocity component $u_j$ ($j \in \{1, 2, 3\}$), while $\omega_{2,0}$ and $\omega_{2,\neq}$ denote the zero and non-zero modes of the vorticity $\omega_2$. Similarly, $u_{j,(0,0)}$ and $u_{j,(0,\neq)}$ represent the $z$-part zero and non-zero modes of $u_{j,0}$, respectively.
			
	\item The  time-space norm $\|f\|_{L^{q}L^{p}}$ is defined by	
	$\|f\|_{L^{p}(\mathbb{T}\times\mathbb{I}\times\mathbb{T})}=\left(\int_{\mathbb{T}\times\mathbb{I}\times\mathbb{T}}|f|^p dxdydz\right)^{\frac{1}{p}}  $
	and 
	$\|f\|_{L^qL^p}=\left\|\|f\|_{L^p(\mathbb{T}\times\mathbb{I}\times\mathbb{T})}\ \right\|_{L^q(0,t)}.$
	Moreover, $\langle\cdot,\cdot\rangle$ denotes the standard $L^2$ scalar product. For simplicity, we write $\|f\|_{L^p(\mathbb{T}\times\mathbb{I}\times\mathbb{T})}$ as $\|f\|_{L^p}.$
	\item We define a norm related to the thermal diffusion
$\|f\|_{Y_{0}}^2:=\|f\|^2_{L^{\infty}L^{2}}
			+\frac{1}{A}\|\nabla f\|^2_{L^{2}L^{2}}.$
	\item  We sometimes denote the partial derivatives $\partial_x,$ $\partial_y$ and $\partial_z$ by
	$\partial_1,$ $\partial_2$ and $\partial_3,$  respectively.
			\item The total mass $ \|n_{\rm in}\|_{L^{1}}$ is denoted by $ M .$  Clearly,
			\begin{equation*}
					\begin{aligned}
							\|n(t)\|_{L^{1}}\leq \|n_{\rm in}\|_{L^{1}}=:M.
						\end{aligned}
				\end{equation*}
		\item Throughout this paper, we denote $C$ by  a positive constant independent of $A$, $t$ and the initial data, and it may be different from line to line.
\end{itemize}

The paper is organized as follows. Section \ref{Sec 2} presents the main ideas and steps for proving the Theorem \ref{Thm:main}.  Section \ref{Sec 4} is dedicated to deriving estimates for the energy $E_1(t)$ and proving Proposition~\ref{prop:F1}. In Section \ref{Sec 5}, we establish the estimates for the non-zero mode energy $E_2(t)$ and prove Proposition \ref{prop:F2}. Section \ref{Sec 6} completes the proof of Proposition \ref{prop:F3}. Some technical lemmas and auxiliary results are presented in the appendix.

\section{Key ideas and proof of the main results}\label{Sec 2}

\subsection{Process the cell density}\

First, let us consider the following simplified PKS system:
\begin{equation}\label{ini01}
	\left\{
	\begin{array}{lr}
		\partial_tn+Ay\partial_xn= \Delta n-\nabla\cdot(n\nabla c)-\mu n^{2}, \\
		\Delta c+n-c=0.
	\end{array}
	\right.
\end{equation}
When $A=0$ and $\mu>0,$ \eqref{ini01} is the so-called PKS system with logistic source;
when $A=0$ and $\mu=0,$ the logistic source disappears, then \eqref{ini01} is the classical PKS system. After time scaling, we get
\begin{equation}\label{ini02}
			\partial_tn+y\partial_xn=\frac{1}{A}\Delta n-\frac{1}{A}\nabla\cdot(n\nabla c)-\frac{1}{A}\mu n^{2}.
\end{equation}

\noindent \textbf{Case 1: $\mu>0$.} For \eqref{ini02}, the energy estimates show that 
\begin{equation*}
	\begin{aligned}
		\frac{d}{dt}\|n\|_{L^p}^p+\frac{4(p-1)}{Ap}\|\nabla (n^{\frac{p}{2}})\|_{L^2}^2
		=-\frac{p<\nabla\cdot(n\nabla c),n^{p-1}>}{A}-\frac{\mu p\|n\|_{L^{p+1}}^{p+1}}{A}.
	\end{aligned}
\end{equation*}
Using elliptic condition $\Delta c+n-c=0,$ we infer from the above equality that
\begin{equation*}
	\begin{aligned}
		\frac{d}{dt}\|n\|_{L^p}^p+\frac{4(p-1)}{Ap}\|\nabla (n^{\frac{p}{2}})\|_{L^2}^2
		=p\left(\frac{p-1}{p}-\mu\right)\frac{\|n\|_{L^{p+1}}^{p+1}}{A}+``{\rm good~terms}".
	\end{aligned}
\end{equation*}

When $p\geq\frac{3}{2}$ and $\mu>\frac{1}{3},$ by the Moser-Alikakos iteration, the solutions to \eqref{ini01} are global in time 
\begin{equation*}
	\begin{aligned}
		\|n\|_{L^{\infty}L^{\infty}}\leq C( \|n_{\rm in}\|_{L^1}, \|n_{\rm in}\|_{L^{\infty}}).
	\end{aligned}
\end{equation*} 

When $0 < \mu \leq \frac{1}{3}$, it remains unknown in three-dimension to determine whether solutions to \eqref{ini01} exist globally in time or blow up. This issue requires further investigation. 
According to the frequency $k_1$, we decompose the cell density $n$  as follows:
\begin{equation*}
	n(t,x,y,z) = n_{\neq}(t,x,y,z) + n_0(t,y,z),
\end{equation*}
where $n_{\neq}(t,x,y,z)$ denotes the $x$-dependent (non-zero mode) part and $n_0(t,y,z)$ denotes the $x$-independent (zero mode) part of $n$. Under the influence of Couette flow, when the flow strength is sufficiently large (i.e., $A$ is large), the non-zero mode $n_{\neq}$ experiences enhanced dispersion. In particular, it satisfies the decay estimate
\begin{equation*}
	\|n_{\neq}(t)\|_{L^2} \leq C {\rm e}^{-a A^{-\frac{1}{3}} t} \left( \|(n_{\rm in})_{\neq}\|_{L^2} + 1 \right),
\end{equation*}
for some constants $C > 0$ and $a > 0$. In this way, our focus shifts to the zero mode $n_0$.  
Using equation~\eqref{ini01} and accounting for the nonlinear interactions between different modes, we obtain the following evolution equation for $n_0$:
\begin{equation*}
	\partial_t n_0 - \frac{1}{A} \Delta n_0 = -\frac{1}{A} \nabla \cdot (n_0 \nabla c_0) - \frac{\mu}{A} n_0^2 + \text{``good terms''}, \quad \text{for}~0 < \mu < \frac{1}{2}.
\end{equation*}
Then, the $L^2$ energy estimate shows that 
\begin{equation}\label{n01}
	\begin{aligned}
	\frac{d}{dt}\|n_0\|_{L^2}^2+\frac{1}{A}\|\nabla n_0\|_{L^2}^2
		\leq\frac{1}{A}\|n_{0}\|_{L^{4}}^{2}\|\Delta c_{0}\|_{L^{2}}+``{\rm good~terms}".
	\end{aligned}
\end{equation}
By the 2D Gagliardo–Nirenberg inequality
\begin{equation*}
	\begin{aligned}
		\|n_0\|_{L^{4}(\mathbb{I}\times\mathbb{T})}^{2}
		\leq C\|n_0\|_{L^{2}(\mathbb{I}\times\mathbb{T})}\|\nabla n_0\|_{L^{2}(\mathbb{I}\times\mathbb{T})},
	\end{aligned}
\end{equation*}
it follows from \eqref{n01} that
\begin{equation*}\label{n02}
	\begin{aligned}
		\frac{d}{dt}\|n_{0}\|_{L^{2}}^{2}\leq\frac{C}{A}\|n_{0}\|_{L^{2}}^{2}\|\Delta c_{0}\|_{L^{2}}^{2}+``{\rm good~terms}".
	\end{aligned}
\end{equation*}
{\bf An important observation} is that the $L^{2}$- norm of cell density can be controlled by the $L^{1}$-norm of the initial mass (see Lemma \ref{lem:n n^2} for more details), that is
\begin{equation*}
	\begin{aligned}
		\frac{1}{A}\|n_{0}\|_{L^{2}L^{2}}^{2}\leq\frac{1}{\mu}\|(n_{\rm in})_{0}\|_{L^{1}}.
	\end{aligned}
\end{equation*}
Therefore, inspired by Tao-Winkler's work \cite{TW2016}, using the boundedness of the $L^{2}$-norm of the density function, elliptic estimates and the Gr\"{o}nwall inequality, we prove that (see more details in the proof of Lemma \ref{lem:n0 L2}) 
\begin{equation*}
	\|n_{0}(t)\|_{L^{2}}\leq C\left(\|(n_{\rm in})_{0}\|_{L^{2}}, \|(n_{\rm in})_{0}\|_{L^{1}} \right).
\end{equation*}

Finally, by the Moser-Alikakos iteration, it can be shown that the cell density stays globally regular with
\begin{equation*}
	\begin{aligned}
		\|n\|_{L^{\infty}L^{\infty}}\leq C(\|n_{\rm in}\|_{L^2}, \|n_{\rm in}\|_{L^{1}}, \|n_{\rm in}\|_{L^{\infty}}).
	\end{aligned}
\end{equation*} 

\noindent \textbf{Case 2: $\mu=0$.}
Similarly, the non-zero mode suffers from the enhanced dispersion:
\begin{equation*}
	\begin{aligned}
		\|n_{\neq}(t)\|_{L^2}\leq C{\rm e}^{-aA^{-\frac{1}{3}}t}(\|(n_{\rm in})_{\neq}\|_{L^2}+1).
	\end{aligned}
\end{equation*}
In order to handle the zero mode more precisely, we decompose $n_{0}$ into two parts:
\begin{equation*}
	n_0(t,y,z)=n_{(0,\neq)}(t,y,z)+n_{(0,0)}(t,y),
\end{equation*}
where $n_{(0,\neq)}(t,y,z)$ is the $z$-dependent part and $n_{(0,0)}(t,y)$ is the $z$-independent part.
It is easy to find that $n_{(0,0)}(t,y)$ satisfies 
\begin{equation}\begin{aligned}\label{nzero1}
	\partial_t n_{(0,0)}-\frac{1}{A} \partial_{y y} n_{(0,0)}= & -\frac{\partial_y(n_{(0,0)} \partial_y c_{(0,0)})+
		\partial_y(n_{(0, \neq)} \partial_y c_{(0, \neq)})_{(0,0)}}{A}+``{\rm good~terms}"
\end{aligned}\end{equation}
and that $n_{(0,\neq)}(t,y)$ satisfies 
\begin{equation*}
	\begin{aligned}
		\partial_t n_{(0, \neq)}-\frac{1}{A} \Delta n_{(0, \neq)}
		= -\frac{\py ( n_{(0, \neq)  } \py c_{\left(0, 0\right)  } ) +  \nabla \cdot ( n_{(0, \neq)  } \nabla c_{\left(0, \neq\right)  } )_{(0,\neq)}}{A}+``{\rm good~terms}".
	\end{aligned}
\end{equation*}
Since $n_{(0,0)}(t,y)$ depends only on the spatial variable $y$ and time $t$, we can prove, for \eqref{nzero1}, that when $A$ is sufficiently large, 
\begin{equation*}\begin{aligned}	 
		\|n_{(0,0)}\|_{L^{\infty}L^2}^2 \leq C \big( \| (n_{\rm in})_{(0, 0)}\|_{L^2}^2 + M^\frac32 \|n_{(0,\neq)}\|_{L^\infty L^2}^\frac{7}{6}+M^4 +1\big).
\end{aligned}\end{equation*}
Since $n_{(0,\neq)}(t,y,z)$ depends on the spatial variables $(y,z)$ and time $t$, it is critical in the $L^1$-norm.
By energy estimates and the Sobolev embedding, we obtain that 
\begin{equation*} 
	\begin{aligned}
		\frac{d}{d t}\|n_{(0, \neq)}\|_{L^2}^2+\frac{2}{A}\|\nabla n_{(0, \neq)}\|_{L^2}^2 
		\leq  \frac{C_{*}^3 M}{\pi A}\|\nabla n_{(0, \neq)}\|_{L^2}^2+``{\rm good~terms}",
	\end{aligned}
\end{equation*}
where $C_{*}$ is the optimal constant of 
\begin{equation*}\begin{aligned}
		\|n_{(0, \neq)}\|_{L^3(\mathbb{I}\times\mathbb{T})} \leq C_{*}\|n_{(0, \neq)}\|_{L^1(\mathbb{I}\times\mathbb{T})}^{\frac{1}{3}}\|\nabla n_{(0, \neq)}\|_{L^2(\mathbb{I}\times\mathbb{T})}^{\frac{2}{3}}.
\end{aligned}\end{equation*}
As long as $M<\frac{2\pi}{C_{*}^3},$ the related energy can be closed successfully, and it holds that
\begin{equation*}\begin{aligned}
		\|n_{\left(0, \neq\right)}\|_{L^\infty L^2}^2 \leq C \left( \|(n_{\rm in})_{(0, \neq)}\|_{L^2}^2 + \|(n_{\rm in})_{(0, 0)}\|_{L^2}^4 + M^{10}+ 1\right).
\end{aligned}\end{equation*}
Thus, we conclude that
\begin{equation*}
	\|n_{0}\|_{L^{\infty}L^2}^2\leq C(\|n_{(0,0)}\|_{L^{\infty}L^2}^2+\|n_{(0,\neq)}\|_{L^{\infty}L^2}^2 )
	\leq  C\left(\|(n_{\rm in})_{0}\|_{L^{2}}, \|(n_{\rm in})_{0}\|_{L^{1}}\right) 
\end{equation*}
and 
\begin{equation*}
	\begin{aligned}
		\|n\|_{L^{\infty}L^{\infty}}\leq C(\|n_{\rm in}\|_{L^2}, \|n_{\rm in}\|_{L^{1}}, \|n_{\rm in}\|_{L^{\infty}}).
	\end{aligned}
\end{equation*} 


\subsection{Velocity deformation}\

Recall that the zero mode $u_0$ satisfies
\begin{equation}\label{eq:u_k0}
	\begin{aligned} 
	\left\{
		\begin{array}{l}
		\partial_t u_{1,0}-\frac{1}{A} \Delta u_{1,0}+u_{2,0}+\frac{1}{A}\left(u_{2,0} \partial_y +u_{3,0} \partial_z \right)u_{1,0}
		+\frac{1}{A}\left(u_{\neq} \cdot \nabla u_{1, \neq}\right)_0 = \frac1A n_0,\\
		\partial_t u_{2, 0}-\frac1A \Delta u_{2, 0}  +  \frac1A\partial_y P_0  +  \frac1A\left(u_{2, 0} \partial_y+u_{3, 0} \partial_z\right) u_{2, 0}  +   \frac1A\lt(u_{\neq} \cdot \nabla u_{2, \neq}\rt)_{0}=0, \\
		\partial_t u_{3, 0}   - \frac1A \Delta u_{3, 0}  +  \frac1A\partial_z P_0  +  \frac1A\left(u_{2, 0} \partial_y+u_{3, 0} \partial_z\right) u_{3, 0}  +   \frac1A\lt(u_{\neq} \cdot \nabla u_{3, \neq}\rt)_{0}=0.
		\end{array}
	\right.
	\end{aligned}
\end{equation}
Among all  zero mode components, $u_{1,0}$ is affected by the 3D lift-up effect, making it the most challenging to control. Inspired by Chen-Wei-Zhang's work \cite{CWZ1}, we decompose $u_{1,0}$ into $u_{1,0}=\widehat{u_{1,0}}+\widetilde{u_{1,0}}$, satisfying
\begin{equation}\label{eq: decompose of u10}
	\begin{aligned} 
		\left\{\begin{array}{l}
			\partial_t \widehat{u_{1,0}}-\frac{1}{A} \Delta \widehat{u_{1,0}}=-\frac{1}{A}\left(u_{2,0} \partial_y \widehat{u_{1,0}}+u_{3,0} \partial_z \widehat{u_{1,0}}\right)-u_{2,0}, \\
			\partial_t \widetilde{u_{1,0}}-\frac{1}{A} \Delta \widetilde{u_{1,0}}=-\frac{1}{A}\left(u_{2,0} \partial_y \widetilde{u_{1,0}}+u_{3,0} \partial_z \widetilde{u_{1,0}}\right)-\frac{1}{A}\left(u_{\neq} \cdot \nabla u_{1, \neq}\right)_0  +  \frac1A n_0,  \\
			\left.\widehat{u_{1,0}}\right|_{t=0}=0,\left.\quad \widetilde{u_{1,0}}\right|_{t=0}=\left(u_{1, \rm in }\right)_0.
		\end{array}\right.
	\end{aligned}
\end{equation}
As a result, the velocity component $u_{1,0}$ is decomposed into a good part $\widetilde{u_{1,0}}$ and a bad part $\widehat{u_{1,0}},$
since $\widetilde{u_{1,0}}$ is not affected by the 3D lift-up effect. In addition, we have {\bf the following important observations}.
\begin{itemize}
	\item The good part $\widetilde{u_{1,0}}$ is not affected by the 3D lift-up effect. Its energy $\|\widetilde{u_{1,0}}\|_{L^\infty H^1}$ is sufficient to close the estimates. Furthermore, $\|\widetilde{u_{1,0}}\|_{L^\infty H^1}$ is well-related to $n_0$ through the inequality (see Lemma \ref{lem:est of E12} for more details)
	$$
	\|\widetilde{u_{1,0}}\|_{H^1} \leq C \left( \|(u_{1,\rm in})_0\|_{L^\infty H^1} + \|n_0\|_{L^\infty L^2} + 1 \right).
	$$

	\item Although the bad part $\widehat{u_{1,0}}$ is affected by the 3D lift-up effect, under the condition
	$	A^{\frac23+}\|u_{\rm in}\|_{H^2(\mathbb{T} \times \mathbb{I} \times \mathbb{T})} \leq C_0,$ 
	its energy $\|\widehat{u_{1,0}}\|_{L^{\infty}H^2}$ suffices to complete all the required computations.
\end{itemize}

In contrast to $u_{1,0}$, the well-behaved components $u_{2,0}$ and $u_{3,0}$ do not require any additional deformation. Due to incompressibility and the Sobolev embedding,  $\|u_{2,0}\|_{H^2}$ and  $\|u_{3,0}\|_{H^1}$ suffice to control
the $L^{\infty}$-norm of $u_{2,0}$ and $u_{3,0}$ as follows
$$\|u_{2,0}\|_{L^{\infty}}+\|u_{3,0}\|_{L^{\infty}}\leq C(\|u_{2,0}\|_{H^2}+\|u_{3,0}\|_{H^1}).$$

To facilitate the estimates of nonzero modes, we use the new vorticity $\omega_2=\partial_z u_1-\partial_x u_3$ and the new velocity $\Delta u_2,$ which satisfy
\begin{equation}\begin{aligned} \label{eq:main1}
		\left\{\begin{array}{l}
			\partial_t \omega_2+y \partial_x \omega_2-\frac{1}{A} \Delta \omega_2+\partial_z u_2   = \frac1A \partial_z n -\frac{1}{A} \partial_z\left(u \cdot \nabla u_1\right)+\frac{1}{A} \partial_x\left(u \cdot \nabla u_3\right), \\
			\partial_t \Delta u_2+y \partial_x \Delta u_2-\frac{1}{A} \Delta\left(\Delta u_2\right)= -\frac{1}{A} \partial_{x}\partial_{y} n -\frac{1}{A}\left(\partial_x^2+\partial_z^2\right)\left(u \cdot \nabla u_2\right) \\
			\qquad\qquad\qquad\qquad\qquad\qquad\qquad+\frac{1}{A} \partial_y\left[\partial_x\left(u \cdot \nabla u_1\right)+\partial_z\left(u \cdot \nabla u_3\right)\right], \\
			\nabla \cdot u=0 ,
		\end{array}\right.
\end{aligned}\end{equation}
with boundary conditions
\begin{equation*}\begin{aligned} \label{eq:boundary value}
		\left\{\begin{array}{l}
			\omega_2(t, x, \pm 1, z)=0, \\
			\partial_y u_2(t, x, \pm 1, z)=u_2(t, x, \pm 1, z)=0.
		\end{array}\right.
\end{aligned}\end{equation*}
We will use the resolvent estimates method developed by Chen-Wei-Zhang's method \cite{CWZ1}
to estimate the decoupled pair of $\{\partial_x\omega_{2,\neq}, \Delta u_{2,\neq}\}.$

\subsection{ Constructing of the energy functional}\

Denoting $\widehat{\Delta}=\Delta^{k_1, k_3}=\partial_y^2-k_1^2-k_3^2$ and $\eta=\sqrt{k_1^2+k_3^2}$, 
we introduce the following norms:
$$
\begin{aligned}
	\|f\|_{X_a^{k_1, k_3}}^2= & \eta | k_1| \| {\rm e}^{a A^{-\frac{1}{3}} t}\left(-\partial_y, i \eta\right) f\|_{L^2 L^2}^2+A^{-1} \eta^2
	\| {\rm e}^{a A^{-\frac{1}{3}} t}(\partial_y^2-\eta^2) f\|_{L^2 L^2}^2 \\
	& +A^{-\frac{3}{2}}\|{\rm e}^{a A^{-\frac{1}{3}}t} \partial_y(\partial_y^2-\eta^2) f\|_{L^2 L^2}^2
	+\eta^2\|{\rm e}^{a A^{-\frac{1}{3}} t}\left(-\partial_y, i \eta\right) f\|_{L^{\infty} L^2}^2 \\
	& +A^{-\frac{1}{2}}\|{\rm e}^{a A^{-\frac{1}{3} }t}(\partial_y^2-\eta^2) f\|_{L^{\infty} L^2}^2, \\
	\|f\|_{Y_a^{k_1, k_3}}^2= & \|{\rm e}^{a A^{-\frac{1}{3}} t} f\|_{L^{\infty} L^2}^2+A^{-1}\|{\rm e}^{a A^{-\frac{1}{3}} t} \partial_y f\|_{L^2 L^2}^2 
	 +\left(\left(A^{-1} k_1^2\right)^{\frac{1}{3}}+A^{-1} \eta^2\right)\|{\rm e}^{a A^{-\frac{1}{3} }t} f\|_{L^2 L^2}^2,
\end{aligned}
$$
and
$$
\|f\|_{X_a}^2=\sum_{k_1 \neq 0, k_3 \in \mathbb{Z}}\|f^{k_1, k_3}\|_{X_a^{k_1, k_3}}^2, \quad
\|f\|_{Y_a}^2=\sum_{k_1 \neq 0, k_3 \in \mathbb{Z}}\|f^{k_1, k_3}\|_{Y_a^{k_1, k_3}}^2 .
$$
Obviously, it follows that
\begin{equation}\begin{aligned} \label{eq:X_a}
		& \|{\rm e}^{a A^{-\frac{1}{3} }t} \partial_x \nabla f_{\neq}\|_{L^2 L^2}^2+A^{-1}\|{\rm e}^{a A^{-\frac{1}{3}} t}
		\left(\partial_x, \partial_z\right) \Delta f_{\neq}\|_{L^2 L^2}^2
		+A^{-\frac{3}{2}}\|{\rm e}^{a A^{-\frac{1}{3}} t} \partial_y \Delta f_{\neq}\|_{L^2 L^2}^2 \\
		& +\|{\rm e}^{a A^{-\frac{1}{3}} t}\left(\partial_x, \partial_z\right) \nabla f_{\neq}\|_{L^{\infty} L^2}^2
		+A^{-\frac{1}{2}}\|{\rm e}^{a A^{-\frac{1}{3} }t} \Delta f_{\neq}\|_{L^{\infty} L^2}^2 \leq C\|f\|_{X_a}^2
\end{aligned}\end{equation}
and
\begin{equation}\begin{aligned} \label{eq:Y_a}
		\|{\rm e}^{a A^{-\frac{1}{3}} t} f_{\neq}\|_{L^{\infty} L^2}^2+\frac{1}{A}\|{\rm e}^{a A^{-\frac{1}{3}} t} \nabla f_{\neq}\|_{L^2 L^2}^2+\frac{1}{A^{\frac{1}{3}}}\|{\rm e}^{a A^{-\frac{1}{3}} t} f_{\neq}\|_{L^2 L^2}^2 \leq C\|f\|_{Y_a}^2.
\end{aligned}\end{equation}

Next, we introduce the energy functional
\begin{equation*}
	\begin{aligned}
		& E_{1,1}(t)= \|n_0\|_{L^\infty L^2}, \\
		& E_{1,2}(t)= \| \widetilde{u_{1,0}}\|_{L^\infty H^1},\\
		& E_{1,3}(t)= A^\epsilon\big(A^{-1}\|\widehat{u_{1,0}}\|_{L^{\infty} H^2}   +  A^{-\frac32}\|\nabla \widehat{u_{1,0}}\|_{L^2 H^2} 
		+ \|u_{2,0}\|_{Y_0}+\|u_{3,0}\|_{Y_0} \big)\\
		&\qquad\qquad+ A^{\frac14\epsilon}\big(\|\nabla u_{2,0}\|_{Y_0}+\|\nabla u_{3,0}\|_{Y_0} +\|\Delta u_{2,0}\|_{Y_0}\big),\\
		& E_{2,1}(t) =  \|\px n_{\neq}\|_{Y_a} , \\
		& E_{2,2}(t)=A^{\frac{5}{12} \epsilon}\left(\|u_{2, \neq}\|_{X_a}+\|\partial_x \omega_{2, \neq}\|_{Y_a}\right),\\& E_3(t)=\|n\|_{L^{\infty} L^{\infty}},
	\end{aligned}
\end{equation*}
where 
\begin{equation*}\begin{aligned} 
		\epsilon=\left\{\begin{array}{l}
			\epsilon_{1},\quad \frac{2}{3}<\epsilon_{1}\leq\frac{4}{5},\\
			\frac45, \qquad \epsilon_{1}>\frac{4}{5}.
		\end{array}\right.
\end{aligned}\end{equation*}
Denote by
\begin{equation*}
	E_{1}(t):=E_{1,1}(t)+E_{1,2}(t)+E_{1,3}(t),\quad E_{2}(t):=E_{2,1}(t)+E_{2,2}(t).
\end{equation*}

\subsection{Closing the energy and proving the main result}\

	We will prove Theorem \ref{Thm:main} in the following three steps.
	
	\textbf{Step~1:} Let us designate $T$ as the terminal point of the largest range $[0, T]$ such that the following hypothesis hold:
	\begin{equation}\label{assumption}
		E_{1}(t)\leq 2F_{1},\quad E_{2}(t)\leq 2F_{2},\quad E_{3}(t)\leq 2F_{3},
	\end{equation}
	for any $t\in[0,T],$ where $F_{1}, F_{2}$ and $F_{3}$ are constants independent of $t$ and $A$ and will be decided during calculations.
	
	\textbf{Step~2:} We need to prove the following propositions:
	\begin{Prop}\label{prop:F1}
		Under the conditions of Theorem \ref{Thm:main}, assuming  \eqref{assumption} holds, there exists a positive constant $\mathcal{B}_{4}$ independent of $A$ and $t,$ such that if $A\geq \mathcal{B}_{4}$, it holds
		\begin{equation*}
			E_{1}(t)\leq F_{1}
		\end{equation*}
		for all $t\in[0,T].$
	\end{Prop}
	\begin{Prop}\label{prop:F2}
			Under the conditions of Theorem \ref{Thm:main}, assuming that \eqref{assumption} holds, there exists a positive constant $\mathcal{B}_{5}$ independent of $A$ and $t,$ such that if $A\geq \mathcal{B}_{5},$ it holds
		\begin{equation*}
			E_{2}(t)\leq F_{2}
		\end{equation*}
		for all $t\in[0,T].$
	\end{Prop}
	\begin{Prop}\label{prop:F3}
		Under the assumptions of Proposition \ref{prop:F1}, it holds
		\begin{equation*}
			\begin{aligned}
			E_{3}(t)\leq F_{3}
			\end{aligned}
		\end{equation*}
		for all $t\in[0,T].$
	\end{Prop}
	
	{\bf Step 3:} Combining all the above propositions with the well-posedness of the system (\ref{eq:main}) presented by Remark \ref{local} and taking $\mathcal{B}=\max\{\mathcal{B}_{4}, \mathcal{B}_{5}\},$ we complete the proof of Theorem \ref{Thm:main}.

\section{Energy estimates for zero modes $E_{1}(t)$: Proof of Proposition \ref{prop:F1}}\label{Sec 4}

\subsection{Energy estimate for $E_{1, 1}(t)$.}
We first estimate the $L^{2}$-norm of the density by considering the cases when $\mu>0$ and $\mu=0.$
\subsubsection{The case of $\mu>0$}

In this case, the logistic source term breaks the conservation of the mass, causing the total mass $M$ to decay over time $t$.
\begin{Lem}\label{lem:n n^2}
	For all $t\in[0,T],$ it holds that
	\begin{equation*}
		\|n_{0}(t)\|_{L^{1}}\leq\left(\frac{\mu t}{4\pi A}+\frac{1}{\|(n_{\rm in})_{0}\|_{L^{1}}} \right)^{-1}
	\end{equation*}
	and
	\begin{equation*}
		\frac{1}{A}\|n_{0}\|_{L^{2}L^{2}}^{2}\leq\frac{1}{\mu}\|(n_{\rm in})_{0}\|_{L^{1}}.
	\end{equation*}
\end{Lem}
\begin{proof}
	Recall that $n_{0}$ satisfies
	\begin{equation}\label{eq: n0}
		\partial_{t}n_{0}=\frac{1}{A}\Delta n_{0}-\frac{1}{A}\nabla\cdot(n\nabla c)_{0}-\frac{1}{A}\nabla\cdot(nu)_{0}-\frac{\mu}{A} n_{0}^{2}-\frac{\mu}{A}(n_{\neq}^{2})_{0}.
	\end{equation}
	As $n_{0}|_{y=\pm 1}=0$ and $n_{\rm in}\geq 0,$ we have
	\begin{equation}\label{n''}
		\int_{\mathbb{I}\times\mathbb{T}}\Delta n_{0}dydz=\int_\mathbb{T}\lt( \int_{\mathbb{I}}\partial_{yy}n_{0}dy\rt)dz=\int_\mathbb{T}\lt(\partial_{y}n_{0}|_{y=-1}^{y=1} \rt)dz \leq0.
	\end{equation}
    By (\ref{n''}), using ${\rm div}~u_{0}=0$ and integrating with respect to $(y,z)$ in (\ref{eq: n0}) over $\mathbb{I}\times\mathbb{T}$, we get
	\begin{equation}\label{n0 t}
		\frac{d}{dt}\int_{\mathbb{I}\times\mathbb{T}}n_{0}dydz\leq-\frac{\mu}{A}\int_{\mathbb{I}\times\mathbb{T}}n_{0}^{2}dydz-\frac{\mu}{A}\int_{\mathbb{I}\times\mathbb{T}}(n_{\neq}^{2})_{0}dydz\leq-\frac{\mu}{A}\int_{\mathbb{I}\times\mathbb{T}}n_{0}^{2}dydz.
	\end{equation}
	Due to $\int_{\mathbb{I}\times\mathbb{T}}n_{0}^{2}dydz\geq\frac{1}{|\mathbb{I}\times\mathbb{T}|}\left(\int_{\mathbb{I}\times\mathbb{T}}n_{0}dydz \right)^{2}=\frac{1}{4\pi}\left(\int_{\mathbb{I}\times\mathbb{T}}n_{0}dydz \right)^{2},$ there holds
	\begin{equation*}
		\frac{d}{dt}\int_{\mathbb{I}\times\mathbb{T}}n_{0}dydz\leq-\frac{\mu}{4\pi A}\left(\int_{\mathbb{I}\times\mathbb{T}}n_{0}dydz \right)^{2},
	\end{equation*}
	which implies that
	\begin{equation}\label{n0 L1}
		\int_{\mathbb{I}\times\mathbb{T}}n_{0}dydz\leq\left(\frac{\mu t}{4\pi A}+\frac{1}{\|(n_{\rm in})_{0}\|_{L^{1}}} \right)^{-1}.
	\end{equation}
	Lastly, integrating \eqref{n0 t} over $(0,t)$, we obtain that
	\begin{equation*}
		\frac{1}{A}\int_{0}^{t}\int_{\mathbb{I}\times\mathbb{T}}n_{0}^{2}dydzds\leq\frac{1}{\mu}\int_{\mathbb{I}\times\mathbb{T}}(n_{\rm in})_{0}dydz.
	\end{equation*}
    Combining the above with (\ref{n0 L1}), we complete the proof.
\end{proof}

\begin{Lem}\label{lem:n0 L2}
	Under the conditions of Theorem \ref{Thm:main} and the assumptions  \eqref{assumption}, there exists a positive constant $\mathcal{B}_{1,1}$ independent of $A$ and $t$, such that if $A\geq \mathcal{B}_{1,1},$ it holds
	\begin{equation*}
		\|n_{0}\|_{L^{\infty}L^{2}}\leq C{\rm e}^{C\|(n_{\rm in})_{0}\|_{L^{1}}}\left(\|(n_{\rm in})_{0}\|_{L^{2}}+\|(n_{\rm in})_{0}\|_{L^{1}}^{2}+1 \right).
	\end{equation*}
\end{Lem}
\begin{proof}
	{\bf{Step I: When $\mu\geq\frac12.$}}		
	Multiplying $(\ref{eq: n0})$ by $2n_{0}$ and integrating with respect to $(y,z)$ over $\mathbb{I}\times\mathbb{T}$,  we get
	\begin{equation}\label{n0 n0'}
		\begin{aligned}
			&	\frac{d}{dt}\|n_{0}\|_{L^{2}}^{2}+\frac{2}{A}\|\nabla n_{0}\|_{L^{2}}^{2}\\=&\frac{2}{A}\int_{\mathbb{I}\times\mathbb{T}}n_{0}\nabla n_{0}\cdot\nabla c_{0}dydz+\frac{2}{A}\int_{\mathbb{I}\times\mathbb{T}}\left[(n_{\neq}\nabla c_{\neq})_{0}+(u_{\neq}n_{\neq})_{0}\right]\cdot\nabla n_{0}dydz\\&-\frac{2\mu}{A}\int_{\mathbb{I}\times\mathbb{T}}n_{0}^{3}dydz-\frac{2\mu}{A}\int_{\mathbb{I}\times\mathbb{T}}(n_{\neq}^{2})_{0}n_{0}dydz\\ \leq&-\frac{1}{A}\int_{\mathbb{I}\times\mathbb{T}}n_{0}^{2}\Delta c_{0}dydz+\frac{2}{A}\|\nabla n_{0}\|_{L^{2}}\|(n_{\neq}\nabla c_{\neq})_{0}+(u_{\neq}n_{\neq})_{0}\|_{L^{2}}-\frac{2\mu}{A}\|n_{0}\|_{L^{3}}^{3},
		\end{aligned}
	\end{equation} 
	where we used $-\frac{2\mu}{A}\int_{\mathbb{I}\times\mathbb{T}}(n_{\neq}^{2})_{0}n_{0}dydz\leq 0$
	and  ${\rm div}~u_{0}=0.$ Due to Lemma \ref{lem:the zero mode of c and n} and the Gagliardo-Nirenberg inequality, there holds $$\|c_{0}\|_{L^{\infty}}\leq C\|c_{0}\|_{L^{2}}^{\frac12}\|\Delta 
		  c_{0}\|_{L^{2}}^{\frac12}\leq C\|n_{0}\|_{L^{2}}.$$ By $(\ref{eq:main})_{2}$, it follows that
	\begin{equation*}
		\begin{aligned}
			-\frac{1}{A}\int_{\mathbb{I}\times\mathbb{T}}n_{0}^{2}\Delta c_{0}dydz=&\frac{1}{A}\int_{\mathbb{I}\times\mathbb{T}}n_{0}^{3}dydz-\frac{1}{A}\int_{\mathbb{I}\times\mathbb{T}}n_{0}^{2}c_{0}dydz\\\leq&\frac{1}{A}\|n_{0}\|_{L^{3}}^{3}+\frac{1}{A}\|n_{0}\|_{L^{2}}^{2}\|c_{0}\|_{L^{\infty}}\leq \frac{1}{A}\|n_{0}\|_{L^{3}}^{3}+\frac{C}{A}\|n_{0}\|_{L^{2}}^{3},
		\end{aligned}
	\end{equation*}
    from which, along with (\ref{n0 n0'}), for  $\mu\geq\frac12$ we obtain
	\begin{equation}\label{n0 n0' 0}
		\begin{aligned}
			\frac{d}{dt}\|n_{0}\|_{L^{2}}^{2}+\frac{1}{A}\|\nabla n_{0}\|_{L^{2}}^{2}\leq&\frac{1-2\mu}{A}\|n_{0}\|_{L^{3}}^{3}+\frac{C}{A}\|n_{0}\|_{L^{2}}^{3}+\frac{C}{A}\|(n_{\neq}\nabla c_{\neq})_{0}+(u_{\neq}n_{\neq})_{0}\|_{L^{2}}^{2}\\\leq&\frac{C}{A}\|n_{0}\|_{L^{2}}^{3}+\frac{C}{A}\|(n_{\neq}\nabla c_{\neq})_{0}+(u_{\neq}n_{\neq})_{0}\|_{L^{2}}^{2}.
		\end{aligned}
	\end{equation}
	Using the Gagliardo-Nirenberg inequality 
	\begin{equation*}
		\|n_{0}\|_{L^{2}}\leq C\|n_{0}\|_{L^{1}}^{\frac12}\|\nabla n_{0}\|_{L^{2}}^{\frac12},
	\end{equation*}
	H$\ddot{\rm o}$lder inequality, Young inequality and Lemma \ref{lem:n n^2}, we infer from (\ref{n0 n0' 0}) that
	\begin{equation}\label{n0 n0' 2}
		\begin{aligned}
			\frac{d}{dt}\|n_{0}\|_{L^{2}}^{2}\leq&-\frac{\|n_{0}\|_{L^{2}}^{4}}{CA\|n_{0}\|_{L^{1}}^{2}}+\frac{C}{A}\|n_{0}\|_{L^{2}}^{3}+\frac{C}{A}\|(n_{\neq}\nabla c_{\neq})_{0}+(u_{\neq}n_{\neq})_{0}\|_{L^{2}}^{2}\\\leq&-\frac{\|n_{0}\|_{L^{2}}^{4}}{2CA\|(n_{\rm in})_{0}\|_{L^{1}}^{2}}
			+\frac{C^{4}\left(2CA\|(n_{\rm in})_{0}\|_{L^{1}}^{2}\right)^{3}}
			{4A^{4}(\frac43)^{3}}+\frac{C}{A}\|(n_{\neq}\nabla c_{\neq})_{0}+(u_{\neq}n_{\neq})_{0}\|_{L^{2}}^{2}.
		\end{aligned}
	\end{equation}
	We denote $G(t)$ by
	\begin{equation}\label{def G(t)}
		G(t):=\frac{C}{A}\int_{0}^{t}\left(\|(n_{\neq}\nabla c_{\neq})_{0}+(u_{\neq}n_{\neq})_{0}\|_{L^{2}} \right)ds,\quad {\rm for}~~t\geq 0.
	\end{equation}
Using Lemma \ref{lemma_u}, Lemma \ref{lem:est of the non-zero mode of c and n} and (\ref{assumption}), we arrive at
	\begin{equation}\label{G(t)}
		\begin{aligned}
			G(t)\leq&\frac{C}{A}\left(\|n_{\neq}\|_{L^{\infty}L^{\infty}}^{2}\|\nabla c_{\neq}\|_{L^{2}L^{2}}^{2}+\|n_{\neq}\|_{L^{\infty}L^{\infty}}^{2}\|u_{\neq}\|_{L^{2}L^{2}}^{2} \right)\\\leq&\frac{C}{A^{\frac23}}\|n\|_{L^{\infty}L^{\infty}}^{2}\left(\|\partial_{x}n_{\neq}\|_{Y_a}^{2}+\|\partial_{x}\omega_{2,\neq}\|_{Y_{a}}^{2}+\| u_{2,\neq}\|_{X_{a}}^{2} \right)\leq\frac{CF_{2}^{2}F_{3}^{2}}{A^{\frac23}}\leq C
		\end{aligned}
	\end{equation}
	provided with $$A\geq F_{2}^{3}F_{3}^{3}=:\mathcal{B}_{1,1}.$$ Then we rewrite \eqref{n0 n0' 2} as
	\begin{equation*}
		\begin{aligned}
			\frac{d}{dt}\left(\|n_{0}\|_{L^{2}}^{2}-G(t)\right)\leq&-\frac{1}{2CA\|(n_{\rm in})_{0}\|_{L^{1}}^{2}}\left[\|n_{0}\|_{L^{2}}^{2}-G(t)-\left(\frac{C^{4}(2CA\|(n_{\rm in})_{0}\|_{L^{1}}^{2})^{4}}{4A^{4}(\frac43)^{3}} \right)^{\frac12} \right]\\&\times\left[\|n_{0}\|_{L^{2}}^{2}+\left(\frac{C^{4}(2CA\|(n_{\rm in})_{0}\|_{L^{1}}^{2})^{4}}{4A^{4}(\frac43)^{3}} \right)^{\frac12} \right],
		\end{aligned}
	\end{equation*}
	which implies that
	\begin{equation*}
		\|n_{0}\|_{L^{2}}^{2}-G(t)\leq\|(n_{\rm in})_{0}\|_{L^{2}}^{2}+ 2\left(\frac{C^{4}(2CA\|(n_{\rm in})_{0}\|_{L^{1}}^{2})^{4}}{4A^{4}(\frac43)^{3}} \right)^{\frac12}.
	\end{equation*}
	Hence, using \eqref{G(t)}, one deduces
	\begin{equation}\label{end n0 1}
		\|n_{0}\|_{L^{2}}\leq C\left(\|(n_{\rm in})_{0}\|_{L^{2}}+\|(n_{\rm in})_{0}\|_{L^{1}}^{2}+1 \right).
	\end{equation}
	
	{\bf{Step II: When $0<\mu<\frac12.$}}            
%
%
	It follows from (\ref{n0 n0'}) that
	\begin{equation}\label{n0 1}
		\frac{d}{dt}\|n_{0}\|_{L^{2}}^{2}+\frac{1}{A}\|\nabla n_{0}\|_{L^{2}}^{2}\leq\frac{1}{A}\|n_{0}\|_{L^{4}}^{2}\|\Delta c_{0}\|_{L^{2}}+\frac{C}{A}\|(n_{\neq}\nabla c_{\neq})_{0}+(u_{\neq}n_{\neq})_{0}\|_{L^{2}}^{2}.
	\end{equation}
	Thanks to the Gagliardo-Nirenberg inequality, there holds
	\begin{equation*}
		\|n_{0}\|_{L^{4}}^{2}\leq C\|n_{0}\|_{L^{2}}\|\nabla n_{0}\|_{L^{2}},
	\end{equation*}
	 which, when substituted into \eqref{n0 1}, yields
	\begin{equation}\label{f g}
		\begin{aligned}
			\frac{d}{dt}\|n_{0}\|_{L^{2}}^{2}\leq\frac{C}{A}\|n_{0}\|_{L^{2}}^{2}\|\Delta c_{0}\|_{L^{2}}^{2}+\frac{C}{A}\left(\|(n_{\neq}\nabla c_{\neq})_{0}\|_{L^{2}}^{2}+\|(u_{\neq}n_{\neq})_{0}\|_{L^{2}}^{2} \right).
		\end{aligned}
	\end{equation}
	Using Lemma \ref{lem:the zero mode of c and n} and Lemma \ref{lem:n n^2}, for any $t\in[0,T]$, one deduces
	\begin{equation} \label{eq:temp0.1}
		\frac{C}{A}\int_{0}^{t}\|\Delta c_{0}(s)\|_{L^{2}}^{2}ds \leq\frac{C}{A}\int_{0}^{t}\|n_{0}(s)\|_{L^{2}}^{2}ds\leq C\|(n_{\rm in})_{0}\|_{L^{1}}.
	\end{equation}
	Using \eqref{eq:temp0.1}, integrating (\ref{f g}) in time over $(0,t)$ and applying the Gr$\ddot{\rm o}$nwall inequality, we get
	\begin{equation*}
		\begin{aligned}
			\|n_{0}(t)\|_{L^{2}}^{2}\leq&{\rm e}^{\frac{C}{A}\int_{0}^{t}\|\Delta c_{0}(s)\|_{L^{2}}^{2}ds}\left[\|(n_{\rm in})_{0}\|_{L^{2}}^{2}+\frac{C}{A}\int_{0}^{t}\left(\|(n_{\neq}\nabla c_{\neq})_{0}\|_{L^{2}}^{2}+\|(u_{\neq}n_{\neq})_{0}\|_{L^{2}}^{2} \right)ds \right]\\\leq&C{\rm e}^{C\|(n_{\rm in})_{0}\|_{L^{1}}}\left(\|(n_{\rm in})_{0}\|_{L^{2}}^{2}+1 \right),
		\end{aligned}
	\end{equation*}
	where we make use of \eqref{G(t)}. Therefore, we infer that
	\begin{equation}\label{end n0 2}
		\|n_{0}\|_{L^{2}}\leq C{\rm e}^{C\|(n_{\rm in})_{0}\|_{L^{1}}}\left(\|(n_{\rm in})_{0}\|_{L^{2}}+1\right).
	\end{equation}
	
	Combining (\ref{end n0 1}) with (\ref{end n0 2}), the proof is completed.
\end{proof}

\subsubsection{The case of  $\mu=0$}\

When $\mu=0$, $n_0$ satisfies the equation
\begin{equation*}
	\begin{aligned}
		\partial_t n_0-\frac{1}{A} \Delta n_0= & -\frac{1}{A}\left[\nabla \cdot(n_{\neq} \nabla c_{\neq})_0+\partial_y(n_0 \partial_y c_0)+\partial_z(n_0 \partial_z c_0)\right] \\
		& -\frac{1}{A}\left[\nabla \cdot(u_{\neq} n_{\neq})_0+\partial_y(u_{2,0} n_0)+\partial_z(u_{3,0} n_0)\right].
	\end{aligned}
\end{equation*}
Decomposing $n_0$ into $n_0=n_{(0,0)}+n_{(0,\neq)},$ we have
\begin{equation}
	\begin{aligned}\label{eq:n00}
		\partial_t n_{(0,0)}-\frac{1}{A} \partial_{y y} n_{(0,0)}= & -\frac{1}{A}\big(\partial_y(n_{(0,0)} \partial_y c_{(0,0)})
		+\partial_y(n_{\neq} \partial_y c_{\neq})_{(0,0)}+\partial_y(u_{2, \neq} n_{\neq})_{(0,0)}\\
		& + \partial_y(n_{(0, \neq)} \partial_y c_{(0, \neq)})_{(0,0)}+\partial_y(u_{2,(0, \neq)} n_{(0, \neq)})_{(0,0)}\big)
	\end{aligned}
\end{equation}
and
\begin{equation}\label{eq n_0 neq}
	\begin{aligned}
		&\partial_t n_{(0, \neq)}-\frac{1}{A} \Delta n_{(0, \neq)} \\
		=&  -\frac{1}{A}\left[\nabla \cdot(n_{\neq} \nabla c_{\neq})_{(0, \neq)} +  \nabla\cdot (n_0 \nabla c_0)_{(0, \neq)} \right] 
		-\frac{1}{A}\left[\nabla \cdot(u_{\neq} n_{\neq})_{(0, \neq)} + \nabla\cdot (u_{0} n_0)_{(0, \neq)} \right] \\
		=& -\frac1A \lt[\nabla \cdot(n_{\neq} \nabla c_{\neq})_{(0, \neq)}  + \nabla \cdot(u_{\neq} n_{\neq})_{(0, \neq)}  \rt] \\
		& -\frac1A\lt[ \nabla \cdot ( n_{(0, 0)  } \nabla c_{(0, \neq)  } ) + \py ( n_{(0, \neq)  } \py c_{(0, 0)  } ) +  \nabla \cdot ( n_{(0, \neq)  } \nabla c_{(0, \neq)  } )  - \py ( n_{(0, \neq)  } \py c_{(0, \neq)  } )_{(0, 0)}\rt] \\
		& -\frac1A\lt[ \pz ( u_{3, (0, 0)  } n_{(0, \neq)  } ) + \nabla\cdot ( u_{(0, \neq)  } n_{(0, 0)  } ) +  \nabla \cdot ( u_{(0, \neq)  } n_{(0, \neq)  } ) 
		-\py(u_{2, (0, \neq)  } n_{(0, \neq)  }  )_{(0, 0)}\rt].
	\end{aligned}
\end{equation}

Employing the methodology established in \cite{CWW1}, we proved the following two results.
\begin{Lem} \label{lem:est of n00}
	Under the conditions of Theorem \ref{Thm:main} and the assumptions \eqref{assumption}, it holds that 
	\begin{equation*}\begin{aligned}
			\|n_{\left(0, 0\right)}\|_{L^\infty L^2}^2 \leq C\left(\|(n_{\rm in})_{(0,0)}\|_{L^2}^2
			+A^{-\frac23} F_3^2 F_2^2 + M^4 + M^\frac32 \|n_{(0, \neq)}\|_{L^\infty L^2}^\frac{7}{6}+A^{-\frac{7}{13} \epsilon} F_{1}^\frac{28}{13} \right).
	\end{aligned}\end{equation*}
\end{Lem}
\begin{proof}
	The $L^2$ energy estimate for \eqref{eq:n00} implies that
	\begin{equation}
		\begin{aligned} \label{eq:est of n00}
			& \frac{d}{d t}\|n_{(0,0)}\|_{L^2}^2+\frac{1}{A}\|\partial_y n_{(0,0)}\|_{L^2}^2 \\
			\leq&     \frac{C}{A}\left(\|\left(n_{\neq} \py c_{\neq}\right)_{\left(0, 0\right)}\|_{L^2}^2+\|\left(u_{2, \neq} n_{\neq}\right)_{\left(0, 0\right)}\|_{L^2}^2\right)   +   \frac{C}{A}\|n_{(0,0)} \partial_y c_{(0,0)}\|_{L^2}^2 \\
			& + \frac{C}{A}  \|(n_{(0, \neq)} \partial_y c_{(0, \neq)})_{(0,0)}\|_{L^2}^2   +  \frac{C}{A}\|(u_{2,(0, \neq)} n_{(0, \neq)})_{(0,0)}\|_{L^2}^2.
		\end{aligned}
	\end{equation}
	It follows from Lemma \ref{lem:c00} that
	\begin{equation}
		\begin{aligned} \label{eq:est of n00 py c00}
			\|n_{(0,0)} \partial_y c_{(0,0)}\|_{L^2}^2  \leq  \| \partial_y c_{(0,0)}\|_{L^\infty}^2 \|n_{(0,0)} \|_{L^2}^2 \leq C\|  n_{(0,0)}\|_{L^2}^4,
		\end{aligned}
	\end{equation}
	from Lemma \ref{lem:est of 00} and Lemma \ref{lem:c00} that
	\begin{equation}
		\begin{aligned}  \label{eq:est of n0neq py c0neq}
			\|(n_{(0, \neq)} \partial_y c_{(0, \neq)})_{(0,0)}\|_{L^2}^2 & \leq C\|n_{(0, \neq)}\|_{L^2}^2\left(\|\partial_y c_{(0, \neq)}\|_{L^2}\|\partial_{y y} c_{(0, \neq)}\|_{L^2}+\|\partial_y c_{(0, \neq)}\|_{L^2}^2\right) \\
			& \leq C M^{\frac{1}{2}}\|n_{(0, \neq)}\|_{L^2}^{\frac{7}{2}}
		\end{aligned}
	\end{equation}
	and from Lemma \ref{lem:est of 00} that
	\begin{equation}
		\begin{aligned} \label{eq:est of u2n}
			\|(u_{2,(0, \neq)} n_{(0, \neq)})_{(0,0)}\|_{L^2}^2 & \leq C\|n_{(0, \neq)}\|_{L^2}^2\|u_{2,(0, \neq)}\|_{L^2}\|\partial_y u_{2,(0, \neq)}\|_{L^2} \\
			& \leq C\|n_{(0, \neq)}\|_{L^2}^2 \|\nabla u_{2, 0}\|_{L^\infty L^2}^2 \leq CA^{-\frac12 \epsilon} \|n_{(0, \neq)}\|_{L^2}^2 F_{1}^2.
		\end{aligned}
	\end{equation}
	Denote by
	\begin{equation*}
		\begin{aligned}
			H(t)  := \frac{C}{A} \int_{0}^{t}\left(\|\left(n_{\neq} \py c_{\neq}\right)_{\left(0, 0\right)}\|_{L^2}^2+\|\left(u_{2, \neq} n_{\neq}\right)_{\left(0, 0\right)}\|_{L^2}^2\right) ds,~~{\rm for}~~t\geq 0.
		\end{aligned}
	\end{equation*}
	Using Nash inequality
	\begin{equation*}
		-\frac1A\|\partial_y n_{(0,0)}\|_{L^2}^2 \leq-\frac{\|n_{(0,0)}\|_{L^2}^6}{CA\|n_{(0,0)}\|_{L^1}^4} =-\frac{\|n_{(0,0)}\|_{L^2}^6}{CA M^4}
	\end{equation*}
	and combining \eqref{eq:est of n00 py c00}, \eqref{eq:est of n0neq py c0neq} and \eqref{eq:est of u2n}, we rewrite \eqref{eq:est of n00} into
	\begin{equation*} 
		\begin{aligned}
			& \frac{d}{d t}\left(\|n_{(0,0)}\|_{L^2}^2-H(t)\right) \\
			\leq & -\frac{\|n_{(0,0)}\|_{L^2}^6}{C A M^4}+\frac{C}{A}\|n_{(0,0)}\|_{L^2}^4+\frac{CM^{\frac{1}{2}}}{A} \|n_{(0, \neq)}\|_{L^2}^{\frac{7}{2}}+\frac{C A^{-\frac12\epsilon} F_{1}^2 }{A}\|n_{(0, \neq)}\|_{L^2}^2   \\
			= & -\frac{1}{C A M^4}\left(\|n_{(0,0)}\|_{L^2}^6-C^2 M^4\|n_{(0,0)}\|_{L^2}^4-C^2 M^{\frac{9}{2}}\|n_{(0, \neq)}\|_{L^2}^{\frac{7}{2}}  -  C^2  A^{-\frac12\epsilon} M^4 F_{1}^2 \|n_{(0, \neq)}\|_{L^2}^2     \right) \\
			\leq & -\frac{1}{3C A M^4}\left( \|n_{(0,0)}\|_{L^2}^6 - \left(C^2 M^4\right)^3  -  3C^2 M^{\frac{9}{2}}\|n_{(0, \neq)}\|_{L^2}^{\frac{7}{2}}  -   3C^2  A^{-\frac12\epsilon} M^4 F_{1}^2 \|n_{(0, \neq)}\|_{L^2}^2 \right)\\
			\leq & -\frac{1}{3C A M^4}\left(\lt( \|n_{(0,0)}\|_{L^2}^2 - H(t)\rt)^3 - K_1(t) \right),
		\end{aligned}
	\end{equation*} 
	where $K_1(t) := \left(C^2 M^4\right)^3  +  3C^2 M^{\frac{9}{2}}\|n_{(0, \neq)}\|_{L^2}^{\frac{7}{2}}  +   3C^2  A^{-\frac12\epsilon} M^4 F_{1}^2 \|n_{(0, \neq)}\|_{L^2}^2$.
	Using the method of proof by contradiction, we infer from the inequality above that
	\begin{equation}\begin{aligned} \label{eq:claim1}
			\|n_{(0,0)}\|_{L^2}^2  - H(t) \leq \|(n_{\rm in})_{(0,0)}\|_{L^2}^2 + \| K_1(\cdot) \|_{L^\infty}^\frac13.
	\end{aligned}\end{equation}
	By \eqref{lem:est of the non-zero mode of c and n} and \eqref{eq:velocity transform}, we get that
	\begin{equation*}
		\begin{aligned}
			H(t) & \leq \frac{C}{A}\|n_{\neq}\|_{L^{\infty} L^{\infty}}^2(\|n_{\neq}\|_{L^2 L^2}^2  +  \|u_{2, \neq}\|_{L^2 L^2}^2) \\
			& \leq \frac{C}{A}\|n\|_{L^{\infty} L^{\infty}}^2 (\|\px n_{\neq}\|_{L^2 L^2}^2  +\|\partial_{x} \nabla u_{2, \neq}\|_{L^2 L^2}^2 ) \leq CA^{-\frac23} F_3^2 (\|\px n_{\neq}\|_{Y_a}^2   +  \|u_{2, \neq}\|_{X_a}^2),
		\end{aligned}
	\end{equation*}
	which together with \eqref{eq:claim1} gives
	\begin{equation*}
		\begin{aligned}
			\|n_{(0,0)}\|_{L^\infty L^2}^2 & \leq \|(n_{\rm in})_{(0,0)}\|_{L^2}^2 + CA^{-\frac23} F_3^2 
			(\|\px n_{\neq}\|_{Y_a}^2 +\|u_{2, \neq}\|_{X_a}^2) \\
			&\quad + CM^4 + CM^\frac32 \|n_{(0, \neq)}\|_{L^\infty L^2}^\frac{7}{6}  + CA^{-\frac16 \epsilon} M^{\frac43} F_{1}^\frac23  \|n_{(0, \neq)}\|_{L^\infty L^2}^\frac{2}{3}\\
			&\leq  \|(n_{\rm in})_{(0,0)}\|_{L^2}^2 + CA^{-\frac23} F_3^2 (\|\px n_{\neq}\|_{Y_a}^2+\|u_{2, \neq}\|_{X_a}^2)\\
			&\quad + CM^4 + CM^\frac32 \|n_{(0, \neq)}\|_{L^\infty L^2}^\frac{7}{6}  + CA^{-\frac{7}{13} \epsilon} F_{1}^\frac{28}{13}.
		\end{aligned}
	\end{equation*}

	The proof is complete by assumption (\ref{assumption}).
\end{proof}
\begin{Lem}\label{lem:est of n001}
	Under the conditions of Theorem \ref{Thm:main} and the assumptions \eqref{assumption}, as long as $M<\frac89\pi$, it holds that
	\begin{equation*}
		\begin{aligned}
			\|n_{\left(0, \neq\right)}\|_{L^\infty L^2}^2 \leq C 
			\left( \|(n_{\rm in})_{(0,\neq)}\|_{L^2}^2 + \|(n_{\rm in})_{(0,0)}\|_{L^2}^4 + A^{-2} F_2^6 F_3^6  + M^{10}  + A^{-2\epsilon}F_{1}^2 F_3^2 + A^{-7\epsilon} F_1^{28} + 1\right).
		\end{aligned}
	\end{equation*}
\end{Lem}
\begin{proof}
	As $n_{(0,\neq)}$ satisfies $(\ref{eq n_0 neq})$, the $L^2$ energy estimate gives
	\begin{equation} \label{eq:est of pt n0neq}
		\frac{d}{d t}\|n_{(0, \neq)}\|_{L^2}^2+\frac{2}{A}\|\nabla n_{(0, \neq)}\|_{L^2}^2 = \sum_{i=1}^{5} I_i,
	\end{equation}
	where
	\begin{equation*}
		\begin{aligned}
			&I_1 := -\frac2A \int_{  \mathbb{I} \times \mathbb{T}} n_{\left(0, \neq\right)} \nabla \cdot ( n_{\left(0, \neq\right)  } \nabla c_{\left(0, \neq\right)  } ) dydz, \\
			&I_2 := -\frac{2}{A} \int_{\mathbb{I} \times \mathbb{T}}  n_{(0, \neq)}  \left[\nabla \cdot\left(n_{\neq} \nabla c_{\neq}\right)_{(0, \neq)}+\nabla \cdot\left(u_{\neq} n_{\neq}\right)_{(0, \neq)}\right]  d y d z, \\
			&I_3 := -\frac{2}{A} \int_{\mathbb{I} \times \mathbb{T}}  n_{(0, \neq)}  \nabla \cdot ( n_{\left(0, 0\right)  } \nabla c_{\left(0, \neq\right)  } )  dydz, \\
			&I_4 := -\frac{2}{A} \int_{\mathbb{I} \times \mathbb{T}}  n_{(0, \neq)} \lt[\py ( n_{\left(0, \neq\right)  } \py c_{\left(0, 0\right)  } ) + \pz ( u_{3, \left(0, 0\right)  } n_{\left(0, \neq\right)  } ) + \nabla\cdot ( u_{\left(0, \neq\right)  } n_{\left(0, 0\right)  } )  \rt]  d y d z
		\end{aligned}
	\end{equation*}
	and
	\begin{equation*}\begin{aligned}
			I_5 := \frac{2}{A} \int_{\mathbb{I} \times \mathbb{T}}  n_{(0, \neq)} \lt[\py ( n_{\left(0, \neq\right)  } \py c_{\left(0, \neq\right)  } )_{\left(0, 0\right)}  + \py (u_{2, \left(0, \neq\right)  } n_{\left(0, \neq\right)  }  )_{\left(0, 0\right)} \rt]  d y d z.
	\end{aligned}\end{equation*}

	For $I_1$, noting that
	\begin{equation*}\begin{aligned}
			\|n_{(0, \neq)}\|_{L^3} \leq C_{*}\|n_{(0, \neq)}\|_{L^1}^{\frac{1}{3}}\|\nabla n_{(0, \neq)}\|_{L^2}^{\frac{2}{3}}.
	\end{aligned}\end{equation*}
Moreover, by Lemma \ref{lem:c00}, there holds
	\begin{equation*}\begin{aligned}
			\|c_{(0, \neq)}\|_{L^3} \leq C\|c_{(0, \neq)}\|_{L^2}^{\frac{2}{3}}\|\nabla c_{(0, \neq)}\|_{L^2}^{\frac{1}{3}} \leq C \|n_{(0, \neq)}\|_{L^2}.
	\end{aligned}\end{equation*}
	Then we have 
	\begin{equation}\begin{aligned} \label{eq:est of I1}
			I_1 & =\frac{1}{A} \int_{\mathbb{I} \times \mathbb{T}} n_{(0, \neq)}^3 d y d z-\frac{1}{A} \int_{\mathbb{I} \times \mathbb{T}} n_{(0, \neq)}^2 c_{(0, \neq)} d y d z \\
			& \leq \frac{C_{*}^3}{A}\|n_{(0, \neq)}\|_{L^1}\|\nabla n_{(0, \neq)}\|_{L^2}^2+\frac{1}{A}\|n_{(0, \neq)}\|_{L^3}^2\|c_{(0, \neq)}\|_{L^3} \\
			& \leq \frac{C_{*}^3 M}{\pi A}\|\nabla n_{(0, \neq)}\|_{L^2}^2+\frac{C}{A}\|n_{(0, \neq)}\|_{L^1}^{\frac{2}{3}}\|\nabla n_{(0, \neq)}\|_{L^2}^{\frac{4}{3}}\|n_{(0, \neq)}\|_{L^2} \\
			& \leq \frac{C_{*}^3 M}{\pi A}\|\nabla n_{(0, \neq)}\|_{L^2}^2+\frac{\delta}{A}\|\nabla n_{(0, \neq)}\|_{L^2}^2+\frac{C(\delta)}{A} M^2\|n_{(0, \neq)}\|_{L^2}^3,
	\end{aligned}\end{equation}
	where $\delta>0$ is a small constant and we use that $-\Delta c_{(0, \neq)}=n_{(0, \neq)}-c_{(0, \neq)}$ and $\|n_{(0, \neq)}\|_{L^1(\mathbb{I} \times \mathbb{T})} \leq \frac{2}{|\mathbb{T}|}\|n\|_{L^1(\mathbb{T} \times \mathbb{I} \times \mathbb{T})}=\frac{M}{\pi}$.
	
	For $I_2$, direct calculations indicate that
	\begin{equation}\begin{aligned} \label{eq:est of I2}
			I_2  \leq \frac{\delta}{A}\|\nabla n_{(0, \neq)}\|_{L^2}^2+\frac{C(\delta)}{A}\left(\|(n_{\neq} \nabla c_{\neq})_{0}\|_{L^2}^2+\|(u_{\neq} n_{\neq})_{0}\|_{L^2}^2\right).
	\end{aligned}\end{equation}

	For $I_3$, by  Lemma \ref{lem:c00} and Lemma \ref{lem:est of 0neq L^infty}, we arrive at
	\begin{equation*}\begin{aligned}
			\|\nabla c_{(0, \neq)}\|_{L^{\infty}} \leq C\|\partial_z \Delta c_{(0, \neq)}\|_{L^2}^\tau\|\Delta c_{(0, \neq)}\|_{L^2}^{1-\tau} \leq C\|\partial_z n_{(0, \neq)}\|_{L^2}^\tau\|n_{(0, \neq)}\|_{L^2}^{1-\tau},
	\end{aligned}\end{equation*}
	where $0<\tau \leq 1$. Then,
	\begin{equation}
		\begin{aligned} \label{eq:est of I3}
			I_3 
			& \leq \frac{C}{A}\|\nabla c_{(0, \neq)}\|_{L^{\infty}}\|n_{(0,0)}\|_{L^2}\|\nabla n_{(0, \neq)}\|_{L^2} \\
			& \leq \frac{C}{A}\|\partial_z n_{(0, \neq)}\|_{L^2}^\tau\|n_{(0, \neq)}\|_{L^2}^{1-\tau}\|n_{(0,0)}\|_{L^2}\|\nabla n_{(0, \neq)}\|_{L^2} \\
			& \leq \frac{\delta}{A}\|\nabla n_{(0, \neq)}\|_{L^2}^2+\frac{C(\delta)}{A}\|n_{(0, \neq)}\|_{L^2}^2\|n_{(0,0)}\|_{L^2}^{\frac{2}{1-\tau}}.
		\end{aligned}
	\end{equation}
	
	For $I_4$, using Lemma \ref{lem:c00}, there holds that
	\begin{equation}\begin{aligned} \label{eq:est of I4}
			I_4  
			& \leq \frac{2}{A}\|n_{(0, \neq)}\|_{L^2}\|\partial_y c_{(0,0)}\|_{L^{\infty}}\|\nabla n_{(0, \neq)}\|_{L^2} +  \frac{C}{A}\|n\|_{L^{\infty} L^{\infty}}\|(u_{2,0},u_{3,0})\|_{L^{\infty} L^2}\|\nabla n_{(0, \neq)}\|_{L^2} \\
			& \leq \frac{C}{A}\|n_{(0, \neq)}\|_{L^2}\|n_{(0,0)}\|_{L^2}\|\nabla n_{(0, \neq)}\|_{L^2} + \frac CA  A^{-\epsilon}F_{1}F_3 \|\nabla n_{(0, \neq)}\|_{L^2} \\
			& \leq \frac{\delta}{A}\|\nabla n_{(0, \neq)}\|_{L^2}^2+\frac{C(\delta)}{A}\|n_{(0, \neq)}\|_{L^2}^2\|n_{(0,0)}\|_{L^2}^2 + \frac {C(\delta)}{A} A^{-2\epsilon}F_{1}^2 F_3^2 .
	\end{aligned}\end{equation}

	For $I_5$, according to \eqref{eq:est of n0neq py c0neq} and \eqref{eq:est of u2n}, one deduces
	\begin{equation}\begin{aligned} \label{eq:est of I5}
			I_5 	& \leq \frac{2}{A}\|(n_{(0, \neq)} \partial_y c_{(0, \neq)})_{(0,0)}\|_{L^2}\|\nabla n_{(0, \neq)}\|_{L^2}  + \frac{2}{A}\|(u_{2,(0, \neq)} n_{(0, \neq)})_{(0,0)}\|_{L^2}\|\nabla n_{(0, \neq)}\|_{L^2}\\
			& \leq \frac{C}{A} M^{\frac{1}{4}}\|n_{(0, \neq)}\|_{L^2}^{\frac{7}{4}}\|\nabla n_{(0, \neq)}\|_{L^2}  +   \frac CA A^{-\frac14 \epsilon} F_{1} \|n_{(0, \neq)}\|_{L^2}  \|\nabla n_{(0, \neq)}\|_{L^2} \\
			& \leq \frac{\delta}{A}\|\nabla n_{(0, \neq)}\|_{L^2}^2  +  \frac{C(\delta)}{A} M^{\frac{1}{2}}\|n_{(0, \neq)}\|_{L^2}^{\frac{7}{2}}  +  \frac{C(\delta)}{A} A^{-\frac12 \epsilon} F_{1}^2 \|n_{(0, \neq)}\|_{L^2}^2 .
	\end{aligned}\end{equation}
	Summing up \eqref{eq:est of I1}, \eqref{eq:est of I2}, \eqref{eq:est of I3}, \eqref{eq:est of I4}, \eqref{eq:est of I5} and \eqref{eq:est of pt n0neq}, we get that
	\begin{equation} \label{eq:temp0.2}
		\begin{aligned}
			& \frac{d}{d t}\|n_{(0, \neq)}\|_{L^2}^2+\frac{2}{A}\|\nabla n_{(0, \neq)}\|_{L^2}^2 \\
			\leq&  \frac{C_{*}^3 M}{\pi A}\|\nabla n_{(0, \neq)}\|_{L^2}^2+\frac{5 \delta}{A}\|\nabla n_{(0, \neq)}\|_{L^2}^2  +  \frac{C(\delta)}{A}\left(\|(n_{\neq} \nabla c_{\neq})_{0}\|_{L^2}^2+\|(u_{\neq} n_{\neq})_{0}\|_{L^2}^2\right) \\
			&  +\frac{C(\delta)}{A}\|n_{(0, \neq)}\|_{L^2}^2\|n_{(0,0)}\|_{L^2}^{\frac{2}{1-\tau}} +\frac{C(\delta)}{A}\|n_{(0, \neq)}\|_{L^2}^2\|n_{(0,0)}\|_{L^2}^2 +\frac{C(\delta)}{A} M^{\frac{1}{2}}\|n_{(0, \neq)}\|_{L^2}^{\frac{7}{2}}\\
			& +\frac{C(\delta)}{A}  A^{-\frac12 \epsilon} F_{1}^2  \|n_{(0, \neq)}\|_{L^2}^2 +\frac{C(\delta)}{A} A^{-2\epsilon}F_{1}^2  F_3^2 +\frac{C(\delta)}{A} M^2\|n_{(0, \neq)}\|_{L^2}^3 .
		\end{aligned}
	\end{equation}
Letting $5\delta = 1 - \frac{C_{*}^3 M}{2\pi} > 0$, by Lemma A.3 in \cite{CWW1}, we can set $C_{*}^{3}=\frac94$, which provides an exact upper bound for the assumption on the initial mass $M$ , i.e., $M<\frac89\pi.$
	Besides, recalling the definition of $G(t)$ in (\ref{def G(t)}) and
	 using Nash inequality
	\begin{equation*}\begin{aligned}
			-\|\nabla n_{(0, \neq)}\|_{L^2}^2 \leq-\frac{\|n_{(0, \neq)}\|_{L^2}^4}{C\|n_{(0, \neq)}\|_{L^1}^2} \leq-\frac{\|n_{(0, \neq)}\|_{L^2}^4}{C M^2},
	\end{aligned}\end{equation*}
	we rewrite \eqref{eq:temp0.2} into 
	\begin{equation}
		\begin{aligned} \label{eq:temp0.3}
			& \frac{d}{d t}\left(\|n_{(0, \neq)}\|_{L^2}^2-G(t)\right) \\
			\leq & -\frac{5 \delta}{C A M^2}\|n_{(0, \neq)}\|_{L^2}^4  +  \frac{C}{A} M^{\frac{1}{2}}  \|n_{(0, \neq)}\|_{L^2}^{\frac{7}{2}}  +  \frac{C}{A} M^2\|n_{(0, \neq)}\|_{L^2}^3 \\
			& +\frac{C}{A}\left( A^{-\frac12 \epsilon} F_{1}^2  +  \|n_{(0,0)}\|_{L^2}^{\frac{2}{1-\tau}} + \|n_{(0,0)}\|_{L^2}^{2}\right)   \|n_{(0, \neq)}\|_{L^2}^2  +  \frac{C}{A} A^{-2\epsilon}F_{1}^2 F_3^2 \\
			\leq & -\frac{5 \delta}{4C A M^2} \left[  \|n_{(0, \neq)}\|_{L^2}^4  - \frac{7^7 C^{16} M^{20}}{2^{8} (5\delta)^8}   -  \frac{27C^8 M^{16}}{(5\delta)^4} \rt.\\
			&\lt. -  \frac{4C^4 M^4}{(5\delta)^2} \left( A^{-\frac12 \epsilon} F_{1}^2  +  \|n_{(0,0)}\|_{L^2}^{\frac{2}{1-\tau}} + \|n_{(0,0)}\|_{L^2}^{2}\right)^2 -  \frac{4C^2 M^2}{5\delta} A^{-2\epsilon}F_{1}^2 F_3^2    \right],
		\end{aligned}
	\end{equation}
	where we use that
	$\frac{C^2 M^\frac52}{5\delta} \|n_{(0, \neq)}\|_{L^2}^{\frac{7}{2}} \leq \frac14 \|n_{(0, \neq)}\|_{L^2}^{4}  + \frac{7^7 C^{16} M^{20}}{2^{10} (5\delta)^8} .$
	Denoting
	\begin{equation*}\begin{aligned}
			K_2 :=  & \frac{7^7 C^{16} M^{20}}{2^{8} (5\delta)^8}   +  \frac{27C^8 M^{16}}{(5\delta)^4}  +  \frac{4C^2 M^2}{5\delta} A^{-2\epsilon}F_{1}^2 F_3^2, \\
			K_3(t)  :=&\frac{4C^4 M^4}{(5\delta)^2} \left( A^{-\frac12 \epsilon} F_{1}^2  +  \|n_{(0,0)}\|_{L^2}^{\frac{2}{1-\tau}} + \|n_{(0,0)}\|_{L^2}^{2}\right)^2,
	\end{aligned}\end{equation*}
	we infer from (\ref{eq:temp0.3}) that 
	\begin{equation}\begin{aligned} \label{eq:claim2}
			\|n_{(0, \neq)} (t)\|_{L^2}^2 - G(t)  \leq \|(n_{\rm in})_{(0,\neq)}\|_{L^2}^2 + K_2^\frac12 + \| K_3(\cdot)\|_{L^\infty}^\frac12
	\end{aligned}\end{equation}
	for all $0\leq t \leq T$. 
	Then, using (\ref{G(t)}), \eqref{eq:claim2} and Lemma \ref{lem:est of n00}, by setting $\tau = \frac16$, there holds
	\begin{equation*}\begin{aligned} \label{eq:temp0.5}
			&\|n_{(0, \neq)}\|_{L^\infty L^2}^2 \\
			\leq& \|(n_{\rm in})_{(0,\neq)}\|_{L^2}^2 + CA^{-\frac23} F_2^2 F_3^2 + CM^{10}  +  C \left(A^{-\epsilon}F_{1} F_3\right)^\frac{10}{9} + C \\
			& +  C\left( A^{-\frac12\epsilon} F_1^2\right)^\frac54 +  CM^2 \|n_{(0,0)}\|_{L^2}^{\frac{12}{5}} \\
			\leq & \|(n_{\rm in})_{(0,\neq)}\|_{L^2}^2 + CA^{-\frac23} F_2^2 F_3^2 + CM^{10}  +  C \left(A^{-\epsilon}F_{1} F_3\right)^\frac{10}{9} +  C\left( A^{-\frac12\epsilon} F_1^2\right)^\frac54 + C \\
			& + \left(\|(n_{\rm in})_{(0,0)}\|_{L^2}^2 + CA^{-\frac23} F_3^2 F_2^2 + CM^4 + CM^\frac32 \|n_{(0, \neq)}\|_{L^\infty L^2}^\frac{7}{6}  + CA^{-\frac{7}{13} \epsilon} F_{1}^\frac{28}{13}\right)^\frac65 \\
			\leq  & 
			C \left( \|(n_{\rm in})_{(0,\neq)}\|_{L^2}^2 + \|(n_{\rm in})_{(0,0)}\|_{L^2}^4  + A^{-2} F_2^6 F_3^6  + M^{10}  + A^{-2\epsilon}F_{1}^2 F_3^2 + A^{-7\epsilon} F_1^{28} + 1\right),
	\end{aligned}\end{equation*}
	where we use
	\begin{equation*}\begin{aligned}
			CM^\frac95 \|n_{(0, \neq)} \|_{L^\infty L^2}^\frac75 \leq \frac12 \|n_{(0, \neq)}\|_{L^\infty L^2}^2 + CM^{6}.
	\end{aligned}\end{equation*}

	Using assumptions (\ref{assumption}), the proof is completed.
\end{proof}
By Lemma \ref{lem:est of n00} and Lemma \ref{lem:est of n001}, the $L^{2}$ estimate of the zero-mode density function can be obtained when $\mu=0.$
\begin{Lem}\label{lem:est of n0 mu=0}
	Under Lemma \ref{lem:est of n00} and Lemma \ref{lem:est of n001}, there exists a positive constant 
	\begin{equation*}
		\mathcal{B}_{1,2}:=F_1^{24}+F_2^{24}+F_3^{24}
	\end{equation*}
	independent of $t$ and $A$, such that if $A\geq \mathcal{B}_{1,2}$, there holds
	\begin{equation*}\begin{aligned}
			\|n_{0}\|_{L^{\infty}L^{2}} \leq C \left( \|(n_{\rm in})_{(0,\neq)}\|_{L^2} + \|(n_{\rm in})_{(0,0)}\|_{L^2}^2  + M^{5}  + 1\right).
	\end{aligned}\end{equation*}
\end{Lem}

Adding Lemma \ref{lem:n n^2} and Lemma \ref{lem:est of n0 mu=0} together, we immediately derive the following estimates for the energy $E_{1,1}(t)$ in the cases when $\mu>0$ and $\mu=0.$
\begin{Cor}\label{cor:est of n0}
	Let $\mathcal{B}_{1}:=\max\{\mathcal{B}_{1,1}, \mathcal{B}_{1,2}\}.$ If $A\geq \mathcal{B}_{1},$ it follows from Lemma \ref{lem:n n^2} and Lemma \ref{lem:est of n0 mu=0} that
	\begin{equation*}
		E_{1,1}(t)\leq C{\rm e}^{C\|(n_{\rm in})_{0}\|_{L^{1}}}\left(\|(n_{\rm in})_{0}\|_{L^{2}}+\|(n_{\rm in})_{0}\|_{L^{1}}^{2}+1 \right),~~{\rm for}~~\mu>0,
	\end{equation*}
	and
	\begin{equation*}
		E_{1,1}(t)\leq C \left( \|(n_{\rm in})_{(0,\neq)}\|_{L^2} + \|(n_{\rm in})_{(0,0)}\|_{L^2}^2  + M^{5}  + 1\right),~~{\rm for}~~\mu=0.
	\end{equation*}
\end{Cor}

\subsection{Energy estimate for $E_{1,2}(t)$.}\

The following nonlinear interaction between non-zero modes of the velocity will be used in estimating $E_{1,2}(t)$.  
\begin{Lem}  \label{lem:velocity estimate 2}
	It holds that
	\begin{equation}\begin{aligned} \label{eq:velocity estimate 2}
			& \|\mathrm{e}^{2 a A^{-\frac{1}{3}} t}\left|u_{\neq}\right|^2\|_{L^2 L^2}^2 \leq C \left(A^{\frac{2}{3}}\|\partial_x \omega_{2,  \neq}\|_{Y_a}^4 + A^{\frac{4}{9}}\| u_{2, \neq}\|_{X_a}^4\right) , \\
			& \|\mathrm{e}^{2 a A^{-\frac{1}{3}} t} u_{\neq} \cdot \nabla u_{\neq}\|_{L^2 L^2}^2 \leq C A \left(\|\partial_x \omega_{2, \neq}\|_{Y_a}^4+\| u_{2, \neq}\|_{X_a}^4\right) , \\
			& \|\mathrm{e}^{2 a A^{-\frac{1}{3}} t} \partial_x\left(u_{\neq} \cdot \nabla u_{\neq}\right)\|_{L^2 L^2}^2 \leq C A\left(\|\partial_x \omega_{2, \neq}\|_{Y_a}^4+\| u_{2, \neq}\|_{X_a}^4\right), \\
			& \|\mathrm{e}^{2 a A^{-\frac{1}{3}} t} \partial_z\left(u_{\neq} \cdot \nabla u_{3, \neq}\right)\|_{L^2 L^2}^2 \leq C A\left(\|\partial_x \omega_{2, \neq}\|_{Y_a}^4+\| u_{2, \neq}\|_{X_a}^4\right), \\
			& \|\mathrm{e}^{2 a A^{-\frac{1}{3}} t} \nabla\left(u_{\neq} \cdot \nabla u_{2, \neq}\right)\|_{L^2 L^2}^2 \leq C A\left(\|\partial_x \omega_{2, \neq}\|_{Y_a}^4+\| u_{2, \neq}\|_{X_a}^4\right).
	\end{aligned}\end{equation}
\end{Lem}
\begin{proof}
	{\bf Estimate of $\eqref{eq:velocity estimate 2}_1$.} By $\eqref{eq:velocity transform}_{1,2}$, (\ref{Sob})  and $\eqref{Sob neq}_{4}$, we have
	\begin{equation}\begin{aligned} \label{eq:temp0}
			\|u_{\neq}\|_{L_{x, z}^{\infty} L_y^2}^2 & \leq C(\|\partial_x u_{\neq}\|_{L^2}^2 + \|\partial_x \partial_z u_{\neq}\|_{L^2}^2) 
			\leq C(\|\partial_x \omega_{2, \neq}\|_{L^2}^2+\|\partial_x \nabla u_{2, \neq}\|_{L^2}^2).
	\end{aligned}\end{equation}
	Using (\ref{Sob}) and $\eqref{eq:velocity transform}_{1,3}$, it holds that
	\begin{equation}\begin{aligned} \label{eq:u_neq^2}
			\||u_{\neq}|^2\|_{L^2}^2&\leq \|u_{\neq}\|_{L_{x, z}^{\infty} L_y^2}^2\|u_{\neq}\|_{L_{y}^{\infty} L_{x,z}^2}^2 \\
			& \leq C(\|\partial_x \omega_{2, \neq}\|_{L^2}^3 +  \|\partial_x \nabla u_{2, \neq}\|_{L^2}^3)
			\left(\|\partial_{y} \omega_{2,\neq}\|_{L^2} + \|\Delta u_{2, \neq}\|_{L^2}\right),
	\end{aligned}\end{equation}
	which in conjunction with \eqref{eq:X_a} and \eqref{eq:Y_a} shows that 
	\begin{equation*}\begin{aligned}
			\|\mathrm{e}^{2 a A^{-\frac{1}{3}} t}\left|u_{\neq}\right|^2\|_{L^2 L^2}^2 
			\leq C\left(A^\frac23 \|\partial_{x}\omega_{2,\neq}\|_{Y_a}^4  +  A^\frac49 \|u_{2,\neq}\|_{X_a}^4 \right).
	\end{aligned}\end{equation*}
	
	{\bf Estimate of $\eqref{eq:velocity estimate 2}_2$.}
	Thanks to $\eqref{Sob g neq}_{4}$ and $\eqref{eq:velocity transform}_3$, we get
	\begin{equation}\begin{aligned} \label{eq:temp1}
			\|\left(\partial_x, \partial_z\right) u_{\neq}\|_{L_y^{\infty} L_{x, z}^2}^2 &\leq C( \|\left(\partial_x, \partial_z\right)\py u_{\neq}\|_{L^2}^2 
			+  \|\left(\partial_x, \partial_z\right) u_{\neq}\|_{L^2}^2)  \\
			&\leq C(\|\nabla \omega_{2, \neq}\|_{L^2}^2+\|\Delta u_{2, \neq}\|_{L^2}^2).
	\end{aligned}\end{equation}
	It follows from $\eqref{Sob g neq}_{4}$ that
	\begin{equation*}\begin{aligned}
			\|u_{2, \neq}\|_{L_{y}^{\infty} L_{x, z}^2}^2 \leq C \|\py u_{2, \neq}\|_{L^2} \| u_{2, \neq}\|_{L^2} 
			\leq C\|\nabla u_{2, \neq}\|_{L^2}^2.
	\end{aligned}\end{equation*}
	Using $\eqref{Sob neq}_{4}$ and $\eqref{eq:velocity transform}_4$, one has that
	\begin{equation}\begin{aligned} \label{eq:temp2}
			\|\partial_y u_{\neq}\|_{L_{x, z}^{\infty} L_{y}^2}^2 &\leq C (\|\px \py u_{\neq}\|_{L^2}^2   + \|\px \py \pz u_{\neq}\|_{L^2}^2)\\
			&\leq C(\|\px\nabla \omega_{2, \neq}\|_{L^2}^2+\|\px\Delta u_{2, \neq}\|_{L^2}^2).
	\end{aligned}\end{equation}
	Summing up \eqref{eq:temp1}-\eqref{eq:temp2}, \eqref{eq:temp0} yields
	\begin{equation}\begin{aligned} \label{eq:temp3}
			\|u_{\neq} \cdot \nabla u_{\neq}\|_{L^2}^2 & \leq\|u_{1, \neq} \partial_x u_{\neq}\|_{L^2}^2+\|u_{2, \neq} \partial_y u_{\neq}\|_{L^2}^2+\|u_{3, \neq} \partial_z u_{\neq}\|_{L^2}^2 \\
			& \leq\|u_{\neq}\|_{L_{x, z}^{\infty} L_y^2 }^2 \|\left(\partial_x, \partial_z\right) u_{\neq}\|_{L_y^{\infty} L_{x, z}^2}^2  +   \|u_{2, \neq}\|_{L_{y}^{\infty} L_{x, z}^2}^2\|\partial_y u_{\neq}\|_{L_{x, z}^{\infty} L_{y}^2}^2 \\
			& \leq C(\|\partial_x \omega_{2, \neq}\|_{L^2}^2+\|\px\nabla u_{2,  \neq}\|_{L^2}^2)(\|\px\nabla \omega_{2, \neq }\|_{L^2}^2  +  \|\px\Delta u_{2, \neq}\|_{L^2}^2).
	\end{aligned}\end{equation}    
	This gives 
	\begin{equation*}\begin{aligned}
			\|\mathrm{e}^{2 a A^{-\frac{1}{3}} t} u_{\neq} \cdot \nabla u_{\neq}\|_{L^2 L^2}^2 
			\leq C A(\|\partial_x \omega_{2, \neq}\|_{Y_a}^4+\| u_{2, \neq}\|_{X_a}^4).
	\end{aligned}\end{equation*}
	
	{\bf Estimate of $\eqref{eq:velocity estimate 2}_3$.}
    Using $\eqref{Sob neq}_{7}$ and $\eqref{eq:velocity transform}_4$, we get
	\begin{equation*}\begin{aligned} \label{eq:temp5}
			\|\px \left(\partial_x, \partial_z\right) u_{\neq}\|_{L_y^{\infty} L_{x, z}^2}^2 &\leq C( \|\px\left(\partial_x, \partial_z\right)\py u_{\neq}\|_{L^2}^2 +  \|\px\left(\partial_x, \partial_z\right) u_{\neq}\|_{L^2}^2)  \\
			&\leq C(\|\px\nabla \omega_{2, \neq}\|_{L^2}^2+\|\px\Delta u_{2, \neq}\|_{L^2}^2).
	\end{aligned}\end{equation*}
	By $\eqref{Sob neq}_{6}$ and $\eqref{eq:velocity transform}_2$, we have
	\begin{equation*}\begin{aligned} 
			\|\px u_{\neq}\|_{L_{z}^{\infty} L_{x, y}^2}^2 \leq C(\|\partial_x u_{\neq}\|_{L^2}^2  +  \|\partial_x\partial_z u_{\neq}\|_{L^2}^2) \leq C(\|\partial_x \omega_{2, \neq}\|_{L^2}^2+\|\px\nabla u_{2, \neq}\|_{L^2}^2).
	\end{aligned}\end{equation*}
	According to $\eqref{Sob g neq}_3$ and $\eqref{eq:velocity transform}_4$, one obtains 
	\begin{equation*}\begin{aligned}
			\|\left(\partial_x, \partial_z\right) u_{\neq}\|_{L_{x, y}^{\infty} L_{z}^2}^2 &\leq C \|\px\left(\partial_x, \partial_z\right)\py u_{\neq}\|_{L^2} \|\px\left(\partial_x, \partial_z\right) u_{\neq}\|_{L^2}  \\
			&\leq C(\|\px\nabla \omega_{2, \neq}\|_{L^2}^2+\|\px\Delta u_{2, \neq}\|_{L^2}^2).
	\end{aligned}\end{equation*}
	Due to $\eqref{Sob g neq}_{3,4}$, we arrive at
	\begin{equation*}\begin{aligned}
			\|\px u_{2, \neq}\|_{L_{y}^{\infty} L_{x, z}^2}^2 \leq C \|\py \px u_{2, \neq}\|_{L^2} \| \px u_{2, \neq}\|_{L^2} 
			\leq C\|\px\nabla u_{2, \neq}\|_{L^2}^2
	\end{aligned}\end{equation*}
	and
	\begin{equation}\begin{aligned}  \label{eq:temp5.5}
			\|u_{2, \neq}\|_{L_{x, y}^{\infty} L_{z}^2}^2 \leq C \|\px u_{2, \neq}\|_{L^2} \|\px\py u_{2, \neq}\|_{L^2}  
			\leq C\|\px\nabla u_{2, \neq}\|_{L^2}^2.
	\end{aligned}\end{equation}
	It follows from $\eqref{Sob neq}_6$ and $\eqref{eq:velocity transform}_4$ that
	\begin{equation*}\begin{aligned} 
			\|\px \partial_y u_{\neq}\|_{L_{ z}^{\infty} L_{x, y}^2}^2 \leq C (\|\px \py u_{\neq}\|_{L^2}^2   + \|\px \py \pz u_{\neq}\|_{L^2}^2)
			\leq C(\|\px\nabla \omega_{2, \neq}\|_{L^2}^2+\|\px\Delta u_{2, \neq}\|_{L^2}^2).
	\end{aligned}\end{equation*}
	Combining the above inequalities with \eqref{eq:temp1} and \eqref{eq:temp2}, we conclude that
	\begin{equation*}\begin{aligned} \label{eq:temp7}
			\|\px \lt(u_{\neq} \cdot \nabla u_{\neq}\rt)\|_{L^2}^2 & \leq \|\px\lt(u_{1, \neq} \partial_x u_{\neq}\rt)\|_{L^2}^2+\|\px\lt(u_{2, \neq} \partial_y u_{\neq}\rt)\|_{L^2}^2+\|\px\lt(u_{3, \neq} \partial_z u_{\neq}\rt)\|_{L^2}^2 \\
			& \leq\|u_{\neq}\|_{L_{x, z}^{\infty} L_y^2 }^2 \|\px\left(\partial_x, \partial_z\right) u_{\neq}\|_{L_y^{\infty} L_{x, z}^2}^2  + \|\px u_{\neq}\|_{L_{z}^{\infty} L_{x, y}^2 }^2 \|\left(\partial_x, \partial_z\right) u_{\neq}\|_{L_{x, y}^{\infty} L_{z}^2}^2  \\
			&\quad +   \|\px u_{2, \neq}\|_{L_{y}^{\infty} L_{x, z}^2}^2\|\partial_y u_{\neq}\|_{L_{x, z}^{\infty} L_{y}^2}^2 +  \| u_{2, \neq}\|_{L_{x, y}^{\infty} L_{ z}^2}^2\|\px\partial_y u_{\neq}\|_{L_{ z}^{\infty} L_{x, y}^2}^2\\
			& \leq C(\|\partial_x \omega_{2, \neq}\|_{L^2}^2+\|\px\nabla u_{2,  \neq}\|_{L^2}^2)(\|\px\nabla \omega_{2, \neq }\|_{L^2}^2  +  \|\px\Delta u_{2, \neq}\|_{L^2}^2), 
	\end{aligned}\end{equation*}           
	which shows that
	\begin{equation*}\begin{aligned}
			\|\mathrm{e}^{2 a A^{-\frac{1}{3}} t} \px \lt(u_{\neq} \cdot \nabla u_{\neq}\rt)\|_{L^2 L^2}^2 \leq C A(\|\partial_x \omega_{2, \neq}\|_{Y_a}^4+\| u_{2, \neq}\|_{X_a}^4).
	\end{aligned}\end{equation*}
	
	{\bf Estimate of $\eqref{eq:velocity estimate 2}_4$.}
	By $\eqref{Sob neq}_1$ and $\eqref{eq:velocity transform}$, it holds
	\begin{equation}\begin{aligned} \label{eq:temp8}
			\|u_{\neq}\|_{L^\infty}^2 &\leq C\big(\|\py\pz u_{\neq}\|_{L^2}\|\px\pz u_{\neq}\|_{L^2}^\frac12\|\px^2 u_{\neq}\|_{L^2}^\frac12  +   \|\px\py u_{\neq}\|_{L^2}\|\px u_{\neq}\|_{L^2}^\frac12\|u_{\neq}\|_{L^2}^\frac12\big) \\
			&\leq C(\|\px\nabla \omega_{2, \neq}\|_{L^2}^2+\|\px\Delta u_{2, \neq}\|_{L^2}^2).
	\end{aligned}\end{equation}
	Thanks to $\eqref{Sob neq}_6$ and $\eqref{eq:velocity transform}_6$, we have
	\begin{equation*}\begin{aligned}
			\|\left(\partial_x, \partial_z\right) u_{3, \neq}\|_{L_{x, y}^{2} L_{z}^\infty}^2 
			\leq C \|(\partial_x^2, \partial_z^2) u_{3, \neq}\|_{L^2}^2 \leq C(\|\partial_x \omega_{2, \neq}\|_{L^2}^2+\|\partial_z \nabla  u_{2, \neq}\|_{L^2}^2).
	\end{aligned}\end{equation*}
	Using $\eqref{Sob g neq}_{2,4}$, there holds
	\begin{equation}\begin{aligned} \label{eq:temp8.5}
			\|\pz u_{2, \neq}\|_{L_{y}^{\infty} L_{x, z}^2}^2+\| u_{2, \neq}\|_{L_{y, z}^{\infty} L_{ x}^2}^2 \leq C\|\left(\px, \pz\right) \nabla u_{2, \neq}\|_{L^2}^2.
	\end{aligned}\end{equation}
	Due to $\eqref{Sob neq}_5$ and $\eqref{Sob g neq}_3$, we get
	\begin{equation*}\begin{aligned} \label{eq:temp9}
			\|\pz u_{\neq}\|_{L_{x,y}^{\infty} L_{z}^2 }^2+\|\pz\partial_y u_{\neq}\|_{L_{x}^{\infty} L_{ y, z}^2}^2 \leq C \|\px\pz\nabla u_{\neq}\|_{L^2}^2 \leq  C(\|\px\nabla \omega_{2, \neq}\|_{L^2}^2+\|\px\Delta u_{2, \neq}\|_{L^2}^2).
	\end{aligned}\end{equation*}
	Combining the calculations above with \eqref{eq:temp2}, we infer that
	\begin{equation*}\begin{aligned} \label{eq:temp10}
			\|\pz \lt(u_{\neq} \cdot \nabla u_{3, \neq}\rt)\|_{L^2}^2 & \leq  \|\pz\lt(u_{1, \neq} \partial_x u_{3, \neq}\rt)\|_{L^2}^2+\|\pz\lt(u_{2, \neq} \partial_y u_{3, \neq}\rt)\|_{L^2}^2+\|\pz\lt(u_{3, \neq} \partial_z u_{3, \neq}\rt)\|_{L^2}^2 \\
			& \leq\|u_{\neq}\|_{L^\infty }^2 \|\pz\left(\partial_x, \partial_z\right) u_{3, \neq}\|_{L^2}^2  + \|\pz u_{\neq}\|_{L_{x,y}^{\infty} L_{z}^2 }^2 \|\left(\partial_x, \partial_z\right) u_{3, \neq}\|_{L_{x, y}^{2} L_{z}^\infty}^2  \\
			&\quad +   \|\pz u_{2, \neq}\|_{L_{y}^{\infty} L_{x, z}^2}^2\|\partial_y u_{3, \neq}\|_{L_{x, z}^{\infty} L_{y}^2}^2   
			+    \| u_{2, \neq}\|_{L_{y, z}^{\infty} L_{ x}^2}^2\|\pz\partial_y u_{\neq}\|_{L_{x}^{\infty} L_{ y, z}^2}^2\\
			& \leq C(\|\partial_x \omega_{2, \neq}\|_{L^2}^2+\|\left(\px, \pz\right)\nabla u_{2,\neq}\|_{L^2}^2)(\|\px\nabla \omega_{2, \neq }\|_{L^2}^2  +  \|\px\Delta u_{2, \neq}\|_{L^2}^2),
	\end{aligned}\end{equation*}    
	which implies that 
	\begin{equation*}\begin{aligned}
			\|\mathrm{e}^{2 a A^{-\frac{1}{3}} t} \pz \lt(u_{\neq} \cdot \nabla u_{\neq}\rt)\|_{L^2 L^2}^2 \leq C A(\|\partial_x \omega_{2, \neq}\|_{Y_a}^4+\| u_{2, \neq}\|_{X_a}^4).
	\end{aligned}\end{equation*}

	{\bf Estimate of $\eqref{eq:velocity estimate 2}_5$.}
	Thanks to $\eqref{Sob neq}_4$ and $\eqref{eq:velocity transform}_4$, we get
	\begin{equation*}\begin{aligned} \label{eq:temp12}
			\|\nabla u_{\neq}\|_{L_{x,z}^{\infty} L_{y}^2 }^2 \leq C(\|\px\nabla u_{\neq}\|_{L^2}^2  +   \|\px\pz\nabla u_{\neq}\|_{L^2}^2  )  
			\leq C(\|\px\nabla \omega_{2, \neq }\|_{L^2}^2  +  \|\px\Delta u_{2, \neq}\|_{L^2}^2).
	\end{aligned}\end{equation*}
	Due to  $\eqref{Sob g neq}_4$, we have
	\begin{equation*}\begin{aligned}
			\|\left(\partial_x, \partial_z\right) u_{2, \neq}\|_{L_{x, z}^{2} L_{y}^\infty}^2 \leq C\|\py\left(\partial_x, \partial_z\right) u_{2, \neq}\|_{L^2}\|\left(\partial_x, \partial_z\right) u_{2, \neq}\|_{L^2} \leq C\|\left(\partial_x, \partial_z\right)\nabla u_{2, \neq}\|_{L^2}^2 .
	\end{aligned}\end{equation*}
	By  $\eqref{Sob neq}_6$, it holds
	\begin{equation*}\begin{aligned}
			\|\nabla u_{2, \neq}\|_{L_{z}^{\infty} L_{x, y}^2}^2 \leq C(\|\nabla u_{2, \neq}\|_{L^2}^2  + \|\pz\nabla u_{2, \neq}\|_{L^2}^2 )  \leq C\|\left(\px, \pz\right)\nabla u_{2, \neq}\|_{L^2}^2.
	\end{aligned}\end{equation*}
	Using $\eqref{Sob neq}_5$ and $\eqref{Sob g neq}_3$, there holds
	\begin{equation*}\begin{aligned}
			&\|\partial_y u_{2, \neq}\|_{L_{x, y}^{\infty} L_{z}^2}^2 \leq C 	\|\px\partial_y u_{2, \neq}\|_{L^2}  \|\px\partial_y^2 u_{2, \neq}\|_{L^2}  \leq C\|\px \Delta u_{2, \neq}\|_{L^2}^2,\\
			&\|\partial_y\nabla u_{2, \neq}\|_{L_{x}^{\infty} L_{ y, z}^2}^2 \leq C \|\px\partial_y\nabla u_{2, \neq}\|_{L^2}^2 \leq C\|\partial_x\Delta u_{2, \neq}\|_{L^2}^2.
	\end{aligned}\end{equation*}
	Combining above with \eqref{eq:temp8} and \eqref{eq:temp8.5}, we get 
	\begin{equation*}\begin{aligned} 
			\|\nabla \lt(u_{\neq} \cdot \nabla u_{2, \neq}\rt)\|_{L^2}^2 & \leq \|\nabla \lt(u_{1, \neq} \partial_x u_{2, \neq}\rt)\|_{L^2}^2+\|\nabla \lt(u_{2, \neq} \partial_y u_{2, \neq}\rt)\|_{L^2}^2+\|\nabla \lt(u_{3, \neq} \partial_z u_{2, \neq}\rt)\|_{L^2}^2 \\
			& \leq\|u_{\neq}\|_{L^\infty }^2 \|\nabla \left(\partial_x, \partial_z\right) u_{2, \neq}\|_{L^2}^2  + \|\nabla u_{\neq}\|_{L_{x,z}^{\infty} L_{y}^2 }^2 \|\left(\partial_x, \partial_z\right) u_{2, \neq}\|_{L_{x, z}^{2} L_{y}^\infty}^2  \\
			&\quad +   \|\nabla u_{2, \neq}\|_{L_{z}^{\infty} L_{x, y}^2}^2   \|\partial_y u_{2, \neq}\|_{L_{x, y}^{\infty} L_{z}^2}^2   +    \| u_{2, \neq}\|_{L_{y, z}^{\infty} L_{ x}^2}^2  \|\partial_y\nabla u_{2, \neq}\|_{L_{x}^{\infty} L_{ y, z}^2}^2\\
			& \leq C\|\left(\px, \pz\right)\nabla u_{2,  \neq}\|_{L^2}^2(\|\px\nabla \omega_{2, \neq }\|_{L^2}^2  +  \|\px\Delta u_{2, \neq}\|_{L^2}^2),
	\end{aligned}\end{equation*} 
	which indicates  that
	\begin{equation*}\begin{aligned}
			\|\mathrm{e}^{2 a A^{-\frac{1}{3}} t} \nabla \lt(u_{\neq} \cdot \nabla u_{2, \neq}\rt)\|_{L^2 L^2}^2
			\leq C A(\|\partial_x \omega_{2, \neq}\|_{Y_a}^4+\| u_{2, \neq}\|_{X_a}^4).
	\end{aligned}\end{equation*}
	
	The proof is complete.
\end{proof}

\begin{Lem} \label{lem:est of E12}
Under the assumptions of \eqref{assumption}, there exists a positive constant $\mathcal{B}_{2}$ independent of $t$ and $A$, such that if $A\geq\mathcal{B}_{2}$, it holds that
	\begin{equation*}\begin{aligned}
			E_{1,2}(t)\leq C\left(\|(u_{1,\rm in})_{0}\|_{H^{1}}+\|n_{0}\|_{L^{\infty}L^{2}}+1 \right).
	\end{aligned}\end{equation*}
\end{Lem}
\begin{proof}
	For \eqref{eq: decompose of u10}, the energy estimate of $\widetilde{u_{1,0}} $ gives that
	$$
	\begin{aligned}
		&\quad\frac{d}{d t} \| \nabla \widetilde{u_{1,0}} \|_{L^2}^2 + \frac2A \| \Delta\widetilde{u_{1,0}} \|_{L^2}^2   \\
		&=\frac2A  \left\langle u_{2,0} \partial_y \widetilde{u_{1,0}}+u_{3,0} \partial_z \widetilde{u_{1,0}}  +  \left(u_{\neq} \cdot \nabla u_{1, \neq}\right)_0  ,  \Delta\widetilde{u_{1,0}}\right\rangle -\frac2A  \left\langle n_0,  \Delta\widetilde{u_{1,0}}\right\rangle,
	\end{aligned}
	$$
	which implies 
	\begin{equation}\begin{aligned} \label{eq:temp18.9}
			\frac{d}{d t}\|\nabla \widetilde{u_{1,0}}   \|_{L^2}^2  \leq& \frac{-\|\nabla \widetilde{u_{1,0}}   \|_{L^2}^2   + 2C\|n_0\|_{L^2}^2 }{AC}   \\&+ \frac{ C\left( \|   u_{2,0} \partial_y \widetilde{u_{1,0}}  +   u_{3,0} \partial_z \widetilde{u_{1,0}}  \|_{L^2}^2 + \|\left(u_{\neq} \cdot \nabla u_{1, \neq}\right)_0\|_{L^2}^2\right)}{A},
	\end{aligned}\end{equation}
	where we use that
	\begin{equation*}\begin{aligned}
			\|\nabla \widetilde{u_{1,0}}  \|_{L^2}^2 \leq C
			\|\Delta \widetilde{u_{1,0}}  \|_{L^2}^2.
	\end{aligned}\end{equation*}
	Denoting 
	$$
	F(t) :=\frac{C}{A}\int_{0}^{t} \left(\|u_{2,0} \partial_y \widetilde{u_{1,0}}  +   u_{3,0} \partial_z \widetilde{u_{1,0}}  \|_{L^2}^2 + \|\left(u_{\neq} \cdot \nabla u_{1, \neq}\right)_0\|_{L^2}^2\right) ds,\quad{\rm for}~~t\geq 0.
	$$
		By $\eqref{eq:velocity estimate 2}_2$ and \eqref{eq:est of u20 u30 L_infty}, when
		$$ A\geq \mathcal{B}_{2}:=\left(F_{1}^{4}+F_{2}^{4}\right)^{\frac{2}{\epsilon}}, $$
		direct calculations indicate that
	\begin{equation}\begin{aligned}\label{F(t) bound}
			F(t) 
			\leq& \frac CA\left(\|\nabla u_{2,0}\|_{L^2L^2}^2 + \|\Delta u_{2,0}\|_{L^2L^2}^2 +  \|\nabla u_{3,0}\|_{L^2L^2}^2\right) \|\widetilde{u_{1,0}} \|_{L^\infty H^1}^2 \\& + \frac CA \|u_{\neq} \cdot \nabla u_{1, \neq}\|_{L^2L^2}^2\leq CA^{-\frac{\epsilon}{2}}F_{1}^{4} +  CA^{-\frac53 \epsilon} F_{2}^4\leq C.
	\end{aligned}\end{equation}
Then we rewrite \eqref{eq:temp18.9} into
	\begin{equation*}\begin{aligned} 
		\frac{d}{dt}\left(\|\nabla\widetilde{u_{1,0}}\|_{L^{2}}^{2}-F(t) \right)\leq -\frac{1}{AC}\left(\|\nabla\widetilde{u_{1,0}}\|_{L^{2}}^{2}-F(t)-2C\|n_{0}\|_{L^{2}}^{2} \right).
	\end{aligned}\end{equation*}
Using proof by contradiction, we deduce that
\begin{equation*}
\|\nabla\widetilde{u_{1,0}}\|_{L^{2}}^{2}-F(t)\leq \|(\nabla u_{1,\rm in})_{0}\|_{L^{2}}^{2}+3C\|n_{0}\|_{L^{\infty}L^{2}}^{2}.
\end{equation*}
Combining the above inequality with (\ref{F(t) bound}), there holds
\begin{equation}\label{widetilde u10}
	\|\nabla\widetilde{u_{1,0}}\|_{L^{\infty}L^{2}}\leq C\left(\|(u_{1,\rm in})_{0}\|_{H^{1}}+\|n_{0}\|_{L^{\infty}L^{2}}+1 \right).
\end{equation}
	Noting that $ \left.\widetilde{ u_{1,0}}\right|_{y=\pm1}  =0$ and using (\ref{widetilde u10}), we have
\begin{equation*}\begin{aligned}
		E_{1, 2}(t) &\leq C \|\nabla \widetilde{u_{1,0}}\|_{L^\infty L^2}\leq   C\left(\|(u_{1,\rm in})_{0}\|_{H^{1}}+\|n_{0}\|_{L^{\infty}L^{2}}+1 \right).
\end{aligned}\end{equation*}
The proof is complete.
\end{proof}

\subsection{Energy estimate for $E_{1,3}(t)$} \label{subsection4.3}

\begin{Lem}\label{lem:E13 14}
	Under the assumptions of \eqref{assumption}, there exists a positive constant $\mathcal{B}_{3}$ independent of $t$ and $A$, such that if $A\geq \mathcal{B}_{3}$, it holds that
	\begin{equation}\label{eq:E13 14}
		\begin{aligned}
		A^{\epsilon}\left(\|u_{2,0}\|_{Y_{0}}+\|u_{3,0}\|_{Y_{0}}\right)\leq C\left(A^{\epsilon}\|(u_{2,\rm in}, u_{3,\rm in})_{0}\|_{L^{2}}+1 \right),\\
			A^{\epsilon-1}\|\Delta\widehat{u_{1,0}}\|_{Y_{0}}\leq C\left(A^{\epsilon}\|(u_{2,\rm in}, u_{3,\rm in})_{0}\|_{L^{2}}+1 \right),\\
			A^{\frac{\epsilon}{4}}\|\nabla (u_{2,0}, u_{3,0})\|_{L^{\infty}L^{2}}+A^{\frac12+\frac{\epsilon}{4}} \|\pt \lt( u_{2, 0}, u_{3, 0}\rt)\|_{L^2 L^2} \leq C\left(A^{\frac{\epsilon}{4}}\|(u_{2,\rm in}, u_{3,\rm in})_{0}\|_{H^{1}}+1 \right),
			\\
			A^{\frac{\epsilon}{4}}\|\nabla u_{2,0}\|_{Y_{0}}\leq C\left(A^{\frac{\epsilon}{4}}\|(u_{2,\rm in}, u_{3,\rm in})_{0}\|_{H^{1}}+1\right),\\
				A^{\frac{\epsilon}{4}}\|\Delta u_{2,0}\|_{Y_{0}}\leq C\left(A^{\frac{\epsilon}{4}}\|(u_{2,\rm in}, u_{3, \rm in})_0\|_{H^2} + 1\right),\\
				A^{\frac{\epsilon}{4}-\frac12}\|\Delta u_{3,0}\|_{L^{2}L^{2}}\leq C\left(A^{\frac{\epsilon}{4}}\|(u_{2,\rm in}, u_{3, \rm in})_0\|_{H^2} + 1\right).
		\end{aligned}
	\end{equation}
\end{Lem}
\begin{proof}
	{\bf Estimate of $(\ref{eq:E13 14})_{1}.$} 
	Due to $\left.u_{k, 0}\right|_{y= \pm 1}=0$ for $k\in\{2,3\}$ and $\nabla\cdot u_0=0$, the $L^2$ energy estimates of \eqref{eq:u_k0} give
	\begin{equation}\begin{aligned} \label{eq:temp0.23}
			\frac{d}{d t}\left(\|u_{2, 0}\|_{L^2}^2+\|u_{3, 0}\|_{L^2}^2\right)+\frac2A \left(\|\nabla u_{2, 0}\|_{L^2}^2+\|\nabla u_{3, 0}\|_{L^2}^2\right)
			= - \frac2A \sum_{k\in\{2,3\}}\left\langle \nabla\cdot \left(u_{\neq} u_{k, \neq}\right)_0, u_{k, 0}\right\rangle.
	\end{aligned}\end{equation}
	 Combining \eqref{eq:temp0.23} with $\eqref{eq:velocity estimate 2}_1$, we get
	\begin{equation*}
		\begin{aligned}
		&A^{2\epsilon}\left(\|u_{2,0}\|_{Y_{0}}^{2}+\|u_{3,0}\|_{Y_{0}}^{2}\right)\leq C\Big(A^{2\epsilon}\|(u_{2,\rm in}, u_{3,\rm in})_{0}\|_{L^{2}}^{2}+\frac{\|{\rm e}^{2aA^{-\frac13}t}|u_{\neq}|^{2}\|_{L^{2}L^{2}}^{2}}{A^{1-2\epsilon}} \Big)\\\leq& C\left(A^{2\epsilon}\|(u_{2,\rm in}, u_{3,\rm in})_{0}\|_{L^{2}}^{2}+A^{-\frac13+\frac{\epsilon}{3}}F_{2}^{4} \right)\leq C\left(A^{2\epsilon}\|(u_{2,\rm in}, u_{3,\rm in})_{0}\|_{L^{2}}^{2}+1 \right)
		\end{aligned}
	\end{equation*}
	provided with $A\geq \mathcal{B}_{3,1}:=F_{2}^{\frac{12}{1-\epsilon}}.$
	
	{\bf Estimate of $(\ref{eq:E13 14})_{2}.$} 	Recall that $\widehat{u_{1, 0}}$ satisfies
	\begin{equation} \left\{\begin{array}{l} \label{eq:widehat u_1,0}
			\partial_t \widehat{u_{1,0}}-\frac{1}{A} \Delta \widehat{u_{1,0}}=-\frac{1}{A}\left(u_{2,0} \partial_y \widehat{u_{1,0}}+u_{3,0} \partial_z \widehat{u_{1,0}}\right)-u_{2,0}, \\
			\left.\widehat{u_{1, 0}}\right|_{t=0} = 0, \quad \left.\widehat{u_{1, 0}}\right|_{y=\pm1} = 0, \quad \left.\Delta\widehat{u_{1, 0}}\right|_{y=\pm1} = 0.
		\end{array}\right. \end{equation}
	Taking ``$\Delta$" on both sides of $\eqref{eq:widehat u_1,0}_1$, we get
	\begin{equation*}\begin{aligned}
			\partial_{t} \Delta\widehat{u_{1,0}}-\frac{1}{A}  \Delta^2 \widehat{u_{1,0}}   =     -\frac{1}{A}\Delta\left(u_{2,0} \partial_y \widehat{u_{1,0}}  +   u_{3,0} \partial_z \widehat{u_{1,0}}\right)  -   \Delta u_{2, 0}.
	\end{aligned}  \end{equation*}
	Thanks to $\left.\Delta\widehat{ u_{1,0}}\right|_{y=\pm1}=0$, the energy estimate gives that
	\begin{equation}\begin{aligned} \label{eq:pt delta hat u10 energy}
			&\frac{d}{d t} \left(A^{2\epsilon-2}\| \Delta\widehat{u_{1,0}} \|_{L^2}^2  \right)
			+ \frac{\| \nabla \Delta\widehat{u_{1,0}} \|_{L^2}^2}{A^{3-2\epsilon}} 
			\\\leq& \frac{C}{A^{3-2\epsilon}}\left(  \|  \nabla\left(u_{2,0} \partial_y \widehat{u_{1,0}}  +   u_{3,0} \partial_z \widehat{u_{1,0}}  \right)    \|_{L^2}^2   
			+  A^2\| \nabla u_{2, 0}\|_{L^2}^2 \right).
	\end{aligned}\end{equation}
	Note that $\eqref{eq:0_norm_L_infty}_{3,4}$ and $\eqref{Sob f0}_3$ imply
	\begin{equation}
		\begin{aligned}  \label{eq:temp21.5}
			\| \nabla\left(u_{2,0} \partial_y \widehat{u_{1,0}}\right) \|_{L^2 L^2}^2 
			& \leq C\left(\| \nabla u_{2,0}\|_{L_t^2 L_y^\infty L_z^2}^2\|\partial_y \widehat{u_{1,0}}\|_{L_t^{\infty} L_y^2 L_z^{\infty}}^2+\|u_{2,0}\|_{L^2 L^{\infty}}^2\| \partial_y \nabla  \widehat{u_{1,0}}\|_{L^{\infty} L^2}^2\right)\\
			& \leq C\| u_{2,0}\|_{L^2 H^2}^2\|\widehat{u_{1,0}}\|_{L^{\infty} H^2}^2 
			\leq C A^{3-\frac52 \epsilon}  F_{1}^4.
		\end{aligned}
	\end{equation}
	By using $\left.\py u_{2, 0}\right|_{y=\pm1} = \left. u_{2, 0}\right|_{y=\pm1}  =0 $, the Gagliardo-Nirenberg inequality and $\nabla\cdot u_{0}  =  0$, we get
	\begin{equation}\begin{aligned}\label{eq:temp26}
			\|u_{3, 0}\|_{L^\infty L^\infty} \leq C\left(\|\nabla u_{3, 0}\|_{L^\infty L^2}  +  \|\Delta u_{2, 0}\|_{L^\infty L^2} \right)  .
	\end{aligned}\end{equation}
	By \eqref{eq:temp26}, one gets that
	\begin{equation}
		\begin{aligned} \label{eq:temp27}
			\| \nabla\left(u_{3,0} \partial_z \widehat{u_{1,0}}\right)\|_{L^2 L^2}^2 
			& \leq C\left( \|\nabla u_{3,0}\|_{L^{\infty}L^2 }^2   \| \partial_z \widehat{u_{1,0}}\|_{L^2L^{\infty} }^2+\|u_{3,0}\|_{L^{\infty} L^{\infty}}^2\| \nabla \partial_z \widehat{u_{1,0}}\|_{L^2 L^2}^2 \right) \\
			& \leq CA^{3-\frac52\epsilon} F_{1}^4,
		\end{aligned}
	\end{equation}
	where we use that
	\begin{equation*}
		\begin{aligned}
			\| \partial_z \widehat{u_{1,0}}\|_{L^2L^{\infty} }+
			\| \nabla \partial_z \widehat{u_{1,0}}\|_{L^2L^2}
			\leq C
			\| \partial_z \Delta\widehat{u_{1,0}} \|_{L^2L^2}.
		\end{aligned}
	\end{equation*}

	Summing up $\eqref{eq:E13 14}_{1}$, \eqref{eq:pt delta hat u10 energy}, \eqref{eq:temp21.5} and \eqref{eq:temp27}, when
	$$ A\geq \max\{\mathcal{B}_{3,1}, F_{1}^{\frac{8}{\epsilon}} \}=:\mathcal{B}_{3,2}, $$
	 we infer that
	\begin{equation*}\begin{aligned}
			A^{2\epsilon-2}\| \Delta\widehat{u_{1,0}} \|_{L^\infty L^2}^2 + A^{2\epsilon-3} \| \nabla \Delta\widehat{u_{1,0}} \|_{L^2 L^2}^2
			\leq &CA^{-\frac12\epsilon} F_{1}^4 + C\left(A^{2\epsilon}\|(u_{2,\rm in}, u_{3,\rm in})_{0}\|_{L^{2}}^{2}+1\right)\\\leq& C\left(A^{2\epsilon}\|(u_{2,\rm in}, u_{3,\rm in})_{0}\|_{L^{2}}^{2}+1 \right).
	\end{aligned}\end{equation*}

{\bf Estimate of $(\ref{eq:E13 14})_{3}.$} 
Multiplying by $\pt u_{k,0}$ on both sides of \eqref{eq:u_k0} and integrating over $\mathbb{I}\times \mathbb{T}$, we have
\begin{equation*}\begin{aligned}
		&\quad\frac1A\frac{d}{dt}\|\nabla \left(u_{2, 0}, u_{3, 0}\right)\|_{L^2}^2  +  2\|\pt\left(u_{2, 0}, u_{3, 0}\right)\|_{L^2}^2 \\
		&= -\frac2A \sum\limits_{k\in\{2,3\}} \langle\left(u_{2, 0} \partial_y+u_{3, 0} \partial_z\right) u_{k, 0}   + \lt(u_{\neq} \cdot \nabla u_{k, \neq}\rt)_0  , \pt u_{k, 0}\rangle,
\end{aligned}\end{equation*}
which gives
\begin{equation*}\begin{aligned} \label{eq:temp28}
		&\quad A^{\frac{\epsilon}{2}}\|\left(\nabla u_{2, 0}, \nabla u_{3, 0}\right)\|_{L^{\infty} L^2}^2  +  A^{1+\frac{\epsilon}{2}} \|\pt \lt( u_{2, 0}, u_{3, 0}\rt)\|_{L^2 L^2}^2 \\
		& \leq C
		\Big(A^{\frac{\epsilon}{2}}\|(u_{2,\rm in}, u_{3, \rm in})_0\|_{H^1}^2  +  \frac{\sum_{k\in\{2,3\}}( \| (u_{\neq} \cdot \nabla u_{k, \neq})_0  \|_{L^2 L^2}^2+\|(u_{2, 0} \partial_y+u_{3, 0} \partial_z) u_{k, 0} \|_{L^2 L^2}^2)}{A^{1-\frac{\epsilon}{2}}} 
		\Big).
\end{aligned}\end{equation*}
Using $\eqref{eq:velocity estimate 2}_2$, \eqref{eq:est of u20 u30 L_infty} and $\eqref{eq:E13 14}_{1}$, we have
\begin{equation}\begin{aligned}  \label{eq:temp29}
	\|\left(u_{2, 0} \partial_y+u_{3, 0} \partial_z\right) u_{k, 0} \|_{L^2L^2}^2 &\leq (\|u_{2, 0}\|_{L^\infty L^\infty}^2  +  \|u_{3, 0}\|_{L^\infty L^\infty}^2) \| \nabla u_{k, 0}\|_{L^2L^2}^2\\
		&\leq CA^{1-\epsilon}F_{1}^{2}\left(A^{\frac{\epsilon}{2}}\|(u_{2,\rm in}, u_{3, \rm in})_0\|_{L^2}^2 + 1\right)
\end{aligned}\end{equation}
and
\begin{equation}\begin{aligned}  \label{eq:temp30}
		\| \lt(u_{\neq} \cdot \nabla u_{k, \neq}\rt)_0  \|_{L^2 L^2}^2 \leq CA^{1- \frac53\epsilon}F_{2}^4.
\end{aligned}\end{equation}
Adding all the above estimations together, when
$$ A\geq \max\{ \mathcal{B}_{3,2}, F_{1}^{\frac{4}{\epsilon}}, F_{2}^{\frac{8}{3\epsilon}}\}=:\mathcal{B}_{3,3},$$
 there holds
\begin{equation*}\begin{aligned} \label{eq:nabla u20 nabla u30}
		A^{\frac{\epsilon}{2}}\|\left(\nabla u_{2, 0}, \nabla u_{3, 0}\right)\|_{L^{\infty} L^2}^2  +  A^{1+\frac{\epsilon}{2}} \|\pt \lt( u_{2, 0}, u_{3, 0}\rt)\|_{L^2 L^2}^2 
		\leq C\left(A^{\frac{\epsilon}{2}}\|(u_{2,\rm in}, u_{3,\rm in})_{0}\|_{H^{1}}^{2}+1\right).
\end{aligned}\end{equation*}

	{\bf Estimate of $(\ref{eq:E13 14})_{4}.$} 
As $\Delta u_{2, 0}$ satisfies
	\begin{equation}\begin{aligned} \label{eq:delta u20}
			\partial_t \Delta u_{2, 0}   - \frac1A \Delta^2 u_{2, 0}  +  \frac1A\partial_y\Delta P^{N_3}_0  +  \frac1A\Delta\left( u_{2, 0} \partial_yu_{2, 0}  +  u_{3, 0} \partial_z u_{2, 0}\right)   +   \frac1A\Delta \lt(u_{\neq} \cdot \nabla u_{2, \neq}\rt)_{0}=0,
	\end{aligned}\end{equation}
	the energy estimate implies 
	\begin{equation*}\begin{aligned}
			&\quad\frac{d}{dt} \|\nabla u_{2, 0}\|_{L^2}^2  +  \frac{2}{A} \|\Delta u_{2, 0}\|_{L^2}^2 + \frac2A \langle  \Delta P_0^{N_3}, \py u_{2, 0} \rangle \\
			&=\frac2A \langle u_{2, 0} \partial_y u_{2, 0}  +  u_{3, 0} \partial_z u_{2, 0} + \lt(u_{\neq} \cdot \nabla u_{2, \neq}\rt)_{0}, \Delta u_{2, 0}  \rangle,
	\end{aligned}\end{equation*}
    which gives that
	\begin{equation*}\begin{aligned}
			&\frac{d}{dt} \left(A^{\frac{\epsilon}{2}}\|\nabla u_{2, 0}\|_{L^2}^2 \right) +  \frac{\|\Delta u_{2, 0}\|_{L^2}^2}{A^{1-\frac{\epsilon}{2}}}   \\\leq& \frac{C( \|\Delta P_0^{N_3}\|_{L^2}^2    + \|u_{2, 0} \partial_y u_{2, 0}  +  u_{3, 0} \partial_z u_{2, 0}\|_{L^2}^2  +   \|\lt(u_{\neq} \cdot \nabla u_{2, \neq}\rt)_{0}\|_{L^2}^2 )}{A^{1-\frac{\epsilon}{2}}}.
	\end{aligned}\end{equation*}
	Integrating the above inequality on $(0, t)$, we get
	\begin{equation}\begin{aligned} \label{eq:delta u20 l2l2}
			A^{\frac{\epsilon}{2}}\|\nabla u_{2,0}\|_{Y_{0}}^{2}
			 \leq& C\Big(A^{\frac{\epsilon}{2}}\|(u_{2,\rm in})_0\|_{H^1}^2+ \frac{1}{A^{1-\frac{\epsilon}{2}}}\|\Delta P_0^{N_3}\|_{L^2 L^2}^2+ \frac{1}{A^{1-\frac{\epsilon}{2}}}\| \lt(u_{\neq} \cdot \nabla u_{2, \neq}\rt)_{0} \|_{L^2 L^2}^2 \\
			&+ \frac{1}{A^{1-\frac{\epsilon}{2}}}\|\left(u_{2, 0} \partial_y u_{2, 0}  +  u_{3, 0} \partial_z u_{2, 0}\right)\|_{L^2 L^2}^2     \Big).
	\end{aligned}\end{equation}
	 By \eqref{eq:pressure} and $\nabla\cdot u = 0$, we have
	\begin{equation}\begin{aligned} \label{eq:temp31}
			\Delta P_0^{N_3} = -\lt( \partial_j u_i \partial_i u_j   \rt)_{0} = -\partial_j u_{i , 0} \partial_i u_{j , 0} - \lt( \partial_j u_{i, \neq} \partial_i u_{j, \neq}  \rt)_{0}.
	\end{aligned}\end{equation}
	By $\eqref{eq:E13 14}_{3}$, when $A\geq \mathcal{B}_{3,3}$, one deduces
	\begin{equation}\begin{aligned} \label{eq:temp32}
			&\frac{1}{A^{1-\frac{\epsilon}{2}}}\| \partial_j u_{i , 0} \partial_i u_{j , 0} \|_{L^2L^2}^2 \leq \frac{C}{A^{1-\frac{\epsilon}{2}}}\sum_{k\in\{2,3\}} \| \partial_k \left(u_0\cdot\nabla u_{k ,0}\right) \|_{L^2L^2}^2 \\
			\leq& \frac{C}{A^{1-\frac{\epsilon}{2}}}\|\left(\nabla u_{2,0}, \nabla u_{3,0}\right)\|_{L^\infty L^2}^2  \|\left(\Delta u_{2,0}, \Delta u_{3, 0}, \nabla\Delta u_{2,0}\right)\|_{L^2 L^2}^2\\\leq& C\left(A^{\frac{\epsilon}{2}}\|(u_{2,\rm in}, u_{3,\rm in} )_0\|_{H^1}^2+1 \right) .
	\end{aligned}\end{equation}
	For $\lt( \partial_j u_{i, \neq} \partial_i u_{j, \neq}  \rt)_{0}$, by $\eqref{eq:velocity estimate 2}_{3,4,5}$, when 
	$$A\geq\max\{\mathcal{B}_{3,3}, F_{2}^{\frac{24}{7\epsilon}} \}=:\mathcal{B}_{3},$$ 
	there holds
	\begin{equation}\begin{aligned} \label{eq:temp33}
			\frac{\|\lt( \partial_j u_{i, \neq} \partial_i u_{j, \neq}  \rt)_{0}\|_{L^2L^2}^2 }{A^{1-\frac{\epsilon}{2}}} \leq \frac{C}{A^{1-\frac{\epsilon}{2}}}\sum_{j=1}^3 \|\lt(\partial_j (u_{\neq}\cdot\nabla u_{j ,\neq})\rt)_0 \|_{L^2L^2}^2 
			\leq C A^{-\frac76\epsilon}F_{2}^4\leq C.
	\end{aligned}\end{equation}
	Combining \eqref{eq:temp31}, \eqref{eq:temp32} and \eqref{eq:temp33}, we have
	\begin{equation}\begin{aligned} \label{eq:temp34}
			\frac{1}{A^{1-\frac{\epsilon}{2}}}\|\Delta P_0^{N_3}\|_{L^2 L^2}^2  \leq C\left(A^{\frac{\epsilon}{2}}\|(u_{2,\rm in}, u_{3,\rm in} )_0\|_{H^1}^2 + 1\right).
	\end{aligned}\end{equation}
	Hence, using \eqref{eq:temp29}, \eqref{eq:temp30}, \eqref{eq:delta u20 l2l2} and \eqref{eq:temp34}, as long as $A\geq \mathcal{B}_{3}$, we arrive at
	\begin{equation*}\begin{aligned}
			A^{\frac{\epsilon}{2}}\|\nabla u_{2,0}\|_{Y_{0}}^{2} \leq  C \left(A^{\frac{\epsilon}{2}}\|(u_{2,\rm in}, u_{3, \rm in})_0\|_{H^1}^2 + 1 \right).
	\end{aligned}\end{equation*}

	{\bf Estimate of $(\ref{eq:E13 14})_{5}.$} 
	By taking the $L^2$ inner product with $\pt u_{2,0}$ to (\ref{eq:delta u20}), we obtain
	$$
	2\|\partial_t \nabla u_{2, 0}\|_{L^2}^2   +   \frac{1}{A}\frac{d}{d t}\|\Delta u_{2, 0}\|_{L^2}^2    +  \frac2A \langle\Delta P_0^{N_3}, \partial_t \partial_y u_{2, 0}\rangle  +  \frac2A\left\langle\nabla\left(u \cdot \nabla u_2\right)_0, \partial_t \nabla u_{2, 0}\right\rangle=0,
	$$
	which implies
	$$
	A^{1+\frac{\epsilon}{2}}\|\partial_t \nabla u_{2, 0}\|_{L^2}^2+\frac{d}{dt}\left(A^{\frac{\epsilon}{2}}\|\Delta u_{2, 0}\|_{L^2}^2\right) \leq \frac{C}{A^{1-\frac{\epsilon}{2}}}\left(\|\Delta P_0^{N_3} \|_{L^2}^2+\|\nabla\left(u \cdot \nabla u_2\right)_0\|_{L^2}^2\right).
	$$
	Integrating the above inequality with time on $(0, t)$, we get
	\begin{equation}\begin{aligned} \label{eq:temp35}
		& A^{1+\frac{\epsilon}{2}} \|\partial_t \nabla u_{2, 0}\|_{L^2 L^2}^2 + A^{\frac{\epsilon}{2}}\| \Delta u_{2, 0}\|_{L^{\infty} L^2}^2  \\\leq& C\left(A^{\frac{\epsilon}{2}}\|(u_{2,\rm in})_0\|_{H^2}^2  +  \frac{1}{A^{1-\frac{\epsilon}{2}}}\|\Delta P_0^{N_3}\|_{L^2 L^2}^2  +  \frac{1}{A^{1-\frac{\epsilon}{2}}}\| \nabla\left(u \cdot \nabla u_{2}\right)_0\|_{L^2 L^2}^2 \right).
	\end{aligned}\end{equation}
	By $\eqref{eq:velocity estimate 2}_5$ and \eqref{eq:temp32}, when $A\geq \mathcal{B}_{3}$, it holds that
	\begin{equation*}\begin{aligned}
			\frac{1}{A^{1-\frac{\epsilon}{2}}}\| \nabla\left(u \cdot \nabla u_{2}\right)_0\|_{L^2 L^2}^2 
			&\leq \frac{1}{A^{1-\frac{\epsilon}{2}}}\| \nabla\left(u_0 \cdot \nabla u_{2, 0}\right)\|_{L^2 L^2}^2+ \frac{1}{A^{1-\frac{\epsilon}{2}}}\|\nabla\left(u_{\neq} \cdot \nabla u_{2, \neq}\right)\|_{L^2 L^2}^2\\
			&\leq C\left(A^{\frac{\epsilon}{2}}\|(u_{2,\rm in}, u_{3, \rm in})_0\|_{H^1}^2+1\right) ,
	\end{aligned}\end{equation*}
	which along with \eqref{eq:temp34} and \eqref{eq:temp35} gives
	\begin{equation}\begin{aligned} \label{eq:pt nabla u20}
		 A^{1+\frac{\epsilon}{2}} \| \partial_t \nabla u_{2, 0}\|_{L^2 L^2}^2 + A^{\frac{\epsilon}{2}}\|\Delta u_{2, 0}\|_{L^{\infty} L^2}^2 \leq C \left(A^{\frac{\epsilon}{2}}\|(u_{2,\rm in}, u_{3, \rm in})_0\|_{H^2}^2 + 1\right).
	\end{aligned}\end{equation}
	
	Next we estimate $A^{\frac{\epsilon}{2}-1}\|\nabla\Delta u_{2, 0}\|_{L^2L^2}^2$.
	Firstly, recall that $\pz u_{2, 0}$ satisfies
	\begin{equation*}\begin{aligned}
			\partial_t \pz u_{2, 0}   - \frac1A \Delta \pz u_{2, 0}  +  \frac1A\partial_y\pz P_0   +   \frac1A\pz\lt(u \cdot \nabla u_{2}\rt)_{0}=0,
	\end{aligned}\end{equation*}
	from which we get that
	\begin{equation*}\begin{aligned}
			\|A\partial_t \pz u_{2, 0}  + \pz\lt(u \cdot \nabla u_{2}\rt)_{0} \|_{L^2}^2  &=  \|\partial_y\pz P_0 - \Delta \pz u_{2, 0}\|_{L^2}^2\\
			&=  \| \partial_y\pz P_0 \|_{L^2}^2  + \| \Delta \pz u_{2, 0} \|_{L^2}^2 -2\langle\py\pz^2 u_{2, 0}, \Delta P^{N_3}_0\rangle,
	\end{aligned}\end{equation*}
	where we use $\left.\nabla u_{2,0}\right|_{y=\pm1} = 0$ and $\Delta P_{0}^{N_{1}}=\Delta P_{0}^{N_{2}}=0.$
	This implies
	\begin{equation*}\begin{aligned}
			&\quad\| \partial_y\pz P_0 \|_{L^2}^2  + \| \Delta \pz u_{2, 0} \|_{L^2}^2 \\
			&\leq  C \lt(\| \Delta P^{N_3}_0 \|_{L^2}^2 + A^2 \| \partial_t \pz u_{2, 0}\|_{L^2}^2 +  \| \pz\lt(u_0 \cdot \nabla u_{2, 0}\rt) \|_{L^2}^2  + \| \pz\lt(u_{\neq} \cdot \nabla u_{2, \neq}\rt) \|_{L^2}^2 \rt).
	\end{aligned}\end{equation*}
	Therefore, by $\eqref{eq:velocity estimate 2}_5$, \eqref{eq:temp32}, \eqref{eq:temp34} and \eqref{eq:pt nabla u20}, when $A\geq \mathcal{B}_{3}$, we get
	\begin{equation}\begin{aligned} \label{eq:delta pz u20}
		\frac{1}{A^{1-\frac{\epsilon}{2}}}	\| \partial_y\pz P_0 \|_{L^2L^2}^2  + \frac{1}{A^{1-\frac{\epsilon}{2}}}\|\Delta \pz u_{2, 0} \|_{L^2L^2}^2 \leq C \left(A^{\frac{\epsilon}{2}}\|(u_{2,\rm in}, u_{3, \rm in})_0\|_{H^2}^2 + 1 \right).
	\end{aligned}\end{equation}
	Secondly, it holds for $\py u_{2, 0}$ that
	\begin{equation*}\begin{aligned}
			\partial_t \py u_{2, 0}   - \frac1A \Delta \py u_{2, 0}  +  \frac1A\partial_y^2 P_0   +   \frac1A\py\lt(u \cdot \nabla u_{2}\rt)_{0}=0,
	\end{aligned}\end{equation*}
	which gives
	\begin{equation*}\begin{aligned}
			\frac{1}{A^{1-\frac{\epsilon}{2}}}\|\Delta\py u_{2, 0}\|_{L^2}^2 \leq C\left(A^{1+\frac{\epsilon}{2}}\|\pt\py u_{2, 0}\|_{L^2}^2  + \frac{1}{A^{1-\frac{\epsilon}{2}}}\|\partial_y^2 P_0\|_{L^2}^2  + \frac{1}{A^{1-\frac{\epsilon}{2}}}\|\py\lt(u \cdot \nabla u_{2}\rt)_{0}\|_{L^2}^2 \right).
	\end{aligned}\end{equation*}
	Noting that by Lemma \ref{lem:transform for p}, we have
	\begin{equation}\begin{aligned} \label{eq:pz2 py2 p}
			\|\py^2 P_0\|_{L^2}^2  +  \|\pz^2 P_0\|_{L^2}^2  \leq C\left(\|\py\pz P_0\|_{L^2}^2  +  \|\Delta P_0^{N_{3}}\|_{L^2}^2\right).
	\end{aligned}\end{equation}
	Hence, we get that
	\begin{equation*}\begin{aligned}
			\frac{\| \Delta\py u_{2, 0}\|_{L^2L^2}^2 }{A^{1-\frac{\epsilon}{2}}}
			&\leq	 C\left(A^{1+\frac{\epsilon}{2}}\|\pt\py u_{2, 0}\|_{L^2L^2}^2+ \frac{1}{A^{1-\frac{\epsilon}{2}}}\|\Delta P^{N_3}_0\|_{L^2L^2}^2 +\frac{1}{A^{1-\frac{\epsilon}{2}}}\|\partial_{y}\partial_{z}P_{0}\|_{L^{2}L^{2}}^{2} \right.\\
			&\left.\quad  + \frac{1}{A^{1-\frac{\epsilon}{2}}}\|\py\lt(u_0 \cdot \nabla u_{2, 0}\rt)\|_{L^2L^2}^2 + \frac{1}{A^{1-\frac{\epsilon}{2}}}\|\py\lt(u_{\neq} \cdot \nabla u_{2, \neq}\rt)\|_{L^2L^2}^2 \right),
	\end{aligned}\end{equation*}		
	which combined with  $\eqref{eq:velocity estimate 2}_5$, \eqref{eq:temp32}, \eqref{eq:temp34}, \eqref{eq:pt nabla u20} and \eqref{eq:delta pz u20} shows
	\begin{equation}\begin{aligned} \label{eq:delta py u20}
			\frac{1}{A^{1-\frac{\epsilon}{2}}}\|{\rm e}^{A^{-1}t} \Delta\py u_{2, 0}\|_{L^2L^2}^2 \leq   C \left( A^{\frac{\epsilon}{2}}\|(u_{2,\rm in}, u_{3, \rm in})_0\|_{H^2}^2   + 1 \right)
	\end{aligned}\end{equation}
	provided with $A\geq \mathcal{B}_{3}$.
	Together with \eqref{eq:temp35}, \eqref{eq:delta pz u20} and \eqref{eq:delta py u20}, we infer that
	\begin{equation*}\begin{aligned} \label{eq:nabla delta u20}
		A^{\frac{\epsilon}{2}}\|\Delta u_{2,0}\|_{Y_{0}}^{2}\leq C\left( A^{\frac{\epsilon}{2}} \|(u_{2,\rm in}, u_{3, \rm in})_0\|_{H^2}^2 + 1\right).
	\end{aligned}\end{equation*} 

	{\bf Estimate of $(\ref{eq:E13 14})_{6}.$} 	As $u_{3, 0}$ satisfies
	\begin{equation*}\begin{aligned}
			\partial_t  u_{3, 0}   - \frac1A \Delta  u_{3, 0}  +  \frac1A\partial_z P_0   +   \frac1A  \lt(u \cdot \nabla u_{3}\rt)_{0}=0,
	\end{aligned}\end{equation*}
	the energy estimate yields 
	\begin{equation*}\begin{aligned}
			\frac{1}{A^{1-\frac{\epsilon}{2}}}\|\Delta  u_{3, 0}\|_{L^2}^2 \leq& C\Big(A^{1+\frac{\epsilon}{2}}\|\pt  u_{3, 0}\|_{L^2}^2  + \frac{1}{A^{1-\frac{\epsilon}{2}}}\|\partial_z P_0\|_{L^2}^2  + \frac{1}{A^{1-\frac{\epsilon}{2}}}\|u_0 \cdot \nabla u_{3, 0}\|_{L^2}^2  \\&+ \frac{1}{A^{1-\frac{\epsilon}{2}}}\|u_{\neq} \cdot \nabla u_{3, \neq}\|_{L^2}^2\Big).
	\end{aligned}\end{equation*}
	Then we get by \eqref{eq:pz2 py2 p} that
	\begin{equation*}\begin{aligned}
			\frac{1}{A^{1-\frac{\epsilon}{2}}}\| \Delta  u_{3, 0}\|_{L^2L^2}^2 \leq& C\left(A^{1+\frac{\epsilon}{2}}\|\pt  u_{3, 0}\|_{L^2L^2}^2  + \frac{1}{A^{1-\frac{\epsilon}{2}}}\|\partial_z^2 P^{N_3}_0\|_{L^2L^2}^2 \right.\\
			&\left. + \frac{1}{A^{1-\frac{\epsilon}{2}}}\| u_0 \cdot \nabla u_{3, 0}\|_{L^2L^2}^2  + \frac{1}{A^{1-\frac{\epsilon}{2}}}\| u_{\neq} \cdot \nabla u_{3, \neq}\|_{L^2L^2}^2\right)\\
			\leq& C\left(A^{1+\frac{\epsilon}{2}}\| \pt  u_{3, 0}\|_{L^2L^2}^2  + \frac{1}{A^{1-\frac{\epsilon}{2}}}\|\partial_z\py P_0\|_{L^2L^2}^2 + \frac{1}{A^{1-\frac{\epsilon}{2}}}\|\Delta P_0^{N_{3}}\|_{L^2L^2}^2 \right.\\
			&\left. + \frac{1}{A^{1-\frac{\epsilon}{2}}}\|u_0 \cdot \nabla u_{3, 0}\|_{L^2L^2}^2  + \frac{1}{A^{1-\frac{\epsilon}{2}}}\|u_{\neq} \cdot \nabla u_{3, \neq}\|_{L^2L^2}^2\right),
	\end{aligned}\end{equation*}
	which coupled with $\eqref{eq:velocity estimate 2}_2$,  $\eqref{eq:E13 14}_{3}$, \eqref{eq:temp32}, \eqref{eq:temp34} and \eqref{eq:delta pz u20} implies
	\begin{equation}\begin{aligned} \label{eq:delta u30}
			\frac{1}{A^{1-\frac{\epsilon}{2}}}\|{\rm e}^{A^{-1}t} \Delta  u_{3, 0}\|_{L^2L^2}^2 \leq C \left(A^{\frac{\epsilon}{2}}\|(u_{2,\rm in}, u_{3,\rm in})_{0}\|_{H^{2}}^{2}+1 \right),
	\end{aligned}\end{equation}
	as long as $A\geq \mathcal{B}_{3}.$	
	
	To sum up, the proof is complete.
\end{proof}

Thanks to $\widehat{u_{1,0}}\big|_{y=\pm 1}=\Delta\widetilde{u_{1,0}}\big|_{y=\pm 1}=0,$ we infer that
\begin{equation*}
	\|\widetilde{u_{1,0}}\|_{L^{\infty}H^{2}}+A^{-\frac12}\|\nabla\widehat{u_{1,0}}\|_{L^{2}H^{2}}\leq C\|\Delta\widetilde{u_{1,0}}\|_{Y_{0}}.
\end{equation*}
Hence, by making use of Lemma \ref{lem:E13 14}, we directly obtain the estimate of the energy $E_{1,3}(t).$
\begin{Cor}\label{cor: E13}
Under the condition of \eqref{eq:smallness of u_in}, it follows from Lemma \ref{lem:E13 14} that
	\begin{equation*}
		E_{1,3}(t)\leq C\left(A^{\epsilon}\|(u_{2,\rm in}, u_{3,\rm in})_{0}\|_{L^{2}}+A^{\frac{\epsilon}{4}}\|(u_{2,\rm in}, u_{3,\rm in})_{0}\|_{H^{2}}+1 \right)\leq C.
	\end{equation*}
\end{Cor}

Recall that $E_{1}(t)=E_{1,1}(t)+E_{1,2}(t)+E_{1,3}(t).$ Based on Corollary \ref{cor:est of n0}, Lemma \ref{lem:est of E12} and Corollary \ref{cor: E13}, the upper bound estimate for the zero mode energy $E_{1}(t)$ is stated below.
	\begin{Cor}\label{Cor:E1 end}
		Let $\mathcal{B}_{4}:=\max\{\mathcal{B}_{1}, \mathcal{B}_{2}, \mathcal{B}_{3}\}.$ If $A\geq \mathcal{B}_{4},$ the following estimates hold, as a consequence of Corollary \ref{cor:est of n0}, Lemma \ref{lem:est of E12} and Corollary \ref{cor: E13}:
		\begin{itemize}
			\item[(i)] \underline{ Case $\mu>0$:}
			\begin{equation*}
			\begin{aligned}
				E_{1}(t)\leq C{\rm e}^{C\|(n_{\rm in})_{0}\|_{L^{1}}}\left(\|(n_{\rm in})_{0}\|_{L^{2}}+\|(n_{\rm in})_{0}\|_{L^{1}}^{2}+\|(u_{1,\rm in})_{0}\|_{H^{1}}+1 \right)=:\mathcal{C}_{1},
			\end{aligned}
			\end{equation*}
			\item[(ii)] \underline{Case $\mu=0$:}
			\begin{equation*}
				\begin{aligned}
					E_{1}(t)\leq C\left(\|(n_{\rm in})_{(0,\neq)}\|_{L^{2}}+\|(n_{\rm in})_{(0,0)}\|_{L^{2}}^{2}+M^{5}+\|(u_{1,\rm in})_{0}\|_{H^{1}}+1 \right)=:\mathcal{C}_{2}.
				\end{aligned}
			\end{equation*}
		\end{itemize}
		Denote $ F_{1}:=\max\{\mathcal{C}_{1}, \mathcal{C}_{2}\}$.
		 For both $\mu>0$ and $\mu=0$, we can prove that $E_{1}(t)\leq F_{1}$.
	\end{Cor}

\section{Estimates for non-zero modes $E_{2}(t)$: Proof of Proposition \ref{prop:F2}}\label{Sec 5}
Recalling $\widehat{\Delta}={\Delta}^{k_1, k_3}=\partial_y^2-k_1^2-k_3^2$, $\eta=\sqrt{k_1^2+k_3^2}$ and applying Fourier transform to \eqref{eq:main} and \eqref{eq:main1} with respect to $(x, z)$, we obtain
\begin{equation}\begin{aligned} \label{eq:mian eq after fourier}
		\left\{\begin{array}{l}
				\partial_t n^{k_1, k_3}-\frac{1}{A}\left(\partial_y^2-\eta^2\right) n^{k_1, k_3}+i k_1 y n^{k_1, k_3} \\
			\quad	=-\frac{1}{A}\left(i k_1, \partial_y, i k_3\right) \cdot(u n)^{k_1, k_3}-\frac{1}{A}\left(i k_1, \partial_y, i k_3\right) \cdot(n \nabla c)^{k_1, k_3}  - \frac{\mu}{A} (n^2)^{k_1, k_3},  \\
			\partial_t \omega_2^{k_1, k_3}-\frac{1}{A}\left(\partial_y^2-\eta^2\right) \omega_2^{k_1, k_3}+i k_1 y \omega_2^{k_1, k_3} \\
			\quad  =-i k_3 u_2^{k_1, k_3}+\frac{1}{A} i k_3 n^{k_1, k_3} - \frac1A ik_3 \left(u\cdot\nabla u_1\right)^{k_1, k_3}  +  \frac1A ik_1 \left(u\cdot\nabla u_3\right)^{k_1, k_3}, \\
			\partial_t \widehat{\Delta} u_2^{k_1, k_3}-\frac{1}{A}\left(\partial_y^2-\eta^2\right) \widehat{\Delta} u_2^{k_1, k_3}+i k_1 y u_2^{k_1, k_3}  \\
			\quad =  -\frac{1}{A} i k_1 \partial_y n^{k_1, k_3}  +  \frac1A \left(k_1^2  +  k_3^2\right)\left(u\cdot\nabla u_2\right)^{k_1, k_3} + \frac1A \py \left[\px\left(u\cdot\nabla u_1\right) +  \pz\left(u\cdot\nabla u_3\right)\right]^{k_1 , k_3} , \\
			n^{k_1, k_3}|_{y= \pm 1}= \omega_2^{k_1, k_3}\big|_{y = \pm 1}=0, \quad
			\partial_y u_2^{k_1, k_3}|_{y= \pm 1}=u_2^{k_1, k_3}\big|_{y= \pm 1}=0.
		\end{array}\right.
\end{aligned}\end{equation}
\subsection{Energy estimate for $E_{2, 1}(t)$.}

\begin{Lem}
	It holds that 
	\begin{equation}\begin{aligned} \label{eq:u and cell}
			& \|\mathrm{e}^{a A^{-\frac{1}{3}} t}\px \left(u_{\neq}n_{\neq}\right)  \|_{L^2 L^2}^2  \leq C  \left( A^\frac13 F_3^2 
			+A^\frac23 \|   \px n_{\neq}  \|_{Y_a}^2\right) (\|\px\omega_{2,  \neq}\|_{Y_a}^2  + \| u_{2, \neq}\|_{X_a}^2  ) ,\\
			& \|\mathrm{e}^{a A^{-\frac{1}{3}} t} \px\left( n_{\neq} \nabla c_{\neq}\right) \|_{L^2 L^2}^2 \leq  C\left( A^\frac13 F_3^2 
			+A^\frac56 \|   \px n_{\neq}  \|_{Y_a}^2\right) \|\px n_{  \neq}\|_{Y_a}^2.
	\end{aligned}\end{equation}
\end{Lem}
\begin{proof}
	By $\eqref{eq:velocity transform}_{2}$ and \eqref{Sob}, there holds
	\begin{equation*}\begin{aligned}
			\| n_{\neq}   \px  u_{\neq}   \|_{L^2}^2  \leq \|\px u_{\neq}\|_{L^2}^2 \| n_{\neq}\|_{L^\infty}^2  \leq \|\partial_{x}^{2}u_{\neq}\|_{L^{2}}^{2}\|n_{\neq}\|_{L^{\infty}}^{2}
			\leq C \| n  \|_{L^\infty}^2 (\|\px\omega_{2,  \neq}\|_{L^2}^2  + \|\px\nabla u_{2, \neq}\|_{L^2}^2  ).
	\end{aligned}\end{equation*}
	Using Lemma \ref{lemma_u}, \eqref{Sob} and $\eqref{Sob neq}_{1}$, we arrive
	\begin{equation*}\begin{aligned}
			\|u_{\neq}\px n_{\neq}\|_{L^2}^2 &\leq \|\px n_{\neq}\|_{L^2}^2  \| u_{\neq}\|_{L^\infty}^2
			\leq C \|\px n_{\neq}\|_{L^2}^2 \| \px \left(\px, \pz \right)  u_{\neq}   \|_{L^2}\|  \left(\px, \pz\right)  u_{\neq}\|_{H^1}\\
			&\leq C \left(\|\px\omega_{2,  \neq}\|_{L^2}  + \|\px\nabla u_{2, \neq}\|_{L^2}  \right)  \left(\|\px\nabla \omega_{2,  \neq}\|_{L^2} 
			+ \|\px\Delta u_{2, \neq}\|_{L^2}  \right) \|   \px n_{\neq}   \|_{L^2}^2.
	\end{aligned}\end{equation*}
	Combining the two inequalities above with (\ref{eq:X_a}) and (\ref{eq:Y_a}), one deduces that
	\begin{equation*}\begin{aligned}
			\|\mathrm{e}^{a A^{-\frac{1}{3}} t}\px \left(u_{\neq}n_{\neq}\right)  \|_{L^2 L^2}^2 
			\leq C  \left( A^\frac13 F_3^2 +A^\frac23 \|   \px n_{\neq} \|_{Y_a}^2\right) (\|\px\omega_{2,  \neq}\|_{Y_a}^2  + \| u_{2, \neq}\|_{X_a}^2 ).
	\end{aligned}\end{equation*}
	It follows from the Gagliardo-Nirenberg inequality that
	\begin{equation*}\begin{aligned}
			\|\partial_{x}n_{\neq}\|_{L^{4}}\leq C\|\partial_{x}n_{\neq}\|_{L^{2}}^{\frac14}\|\nabla\partial_{x}n_{\neq}\|_{L^{2}}^{\frac34},
	\end{aligned}\end{equation*} 
	from which, combined with Lemma \ref{lem:est of the non-zero mode of c and n}, we get
	\begin{equation*}\begin{aligned}
			\| \px \left(n_{\neq}\nabla c_{\neq}\right) \|_{L^2}^2 &\leq C(\| n_{\neq} \px \nabla c_{\neq} \|_{ L^2}^2 
			+ \|\px n_{\neq} \nabla c_{\neq}\|_{ L^2}^2 ) \\
			&\leq C( \|\px  \nabla c_{\neq}\|_{L^2}^2\| n_{\neq}\|_{L^\infty}^2  + \|\px  n_{\neq}\|_{L^4}^2  \|\nabla c_{\neq} \|_{L^4}^2) \\
			&\leq C \big(\|\partial_{x}n_{\neq}\|_{L^{2}}^{2}\|n_{\neq}\|_{L^{\infty}}^{2}+\|\partial_{x}n_{\neq}\|_{L^{2}}^{\frac52}\|\nabla\partial_{x}n_{\neq}\|_{L^{2}}^{\frac32} \big).
	\end{aligned}\end{equation*}
	Hence, the above inequality along with (\ref{eq:X_a}) and (\ref{eq:Y_a}) gives
	\begin{equation*}\begin{aligned}
			\|\mathrm{e}^{a A^{-\frac{1}{3}} t}\px \left(n_{\neq} \nabla c_{\neq}\right) \|_{L^2 L^2}^2 \leq  C\left( A^\frac13 F_3^2 
			+A^\frac56 \|\px n_{\neq}\|_{Y_a}^2\right) \|\px n_{  \neq}\|_{Y_a}^2  .
	\end{aligned}\end{equation*}
	
	The proof is complete.
\end{proof}

\begin{Lem} 
	Under the assumptions of \eqref{assumption}, it holds that
	\begin{equation}\label{0 n_neq}
		\begin{aligned}
			&	\|\mathrm{e}^{ a A^{-\frac{1}{3}} t} u_{1,0} \px n_{ \neq}\|_{L^2 L^2}^2 \leq C A \lt(A^{-\frac16} F_{1}^2  +  A^{\frac43 - 2 \epsilon} F_{1}^2 \rt) \|\px n_{\neq}\|_{Y_a}^2,\\&
			\|\mathrm{e}^{ a A^{-\frac{1}{3}} t} (u_{2,0}, u_{3,0})\px n_{ \neq}\|_{L^2 L^2}^2  +  \|\mathrm{e}^{ a A^{-\frac{1}{3}} t} n_0 \px u_{ \neq}  \|_{L^2 L^2}^2 \\&\qquad\leq C  A^{\frac13 -\frac12 \epsilon }  F_{1}^2 \|\px n_{\neq}\|_{Y_a}^2+ C A^{ \frac13 }F_3^2  ( \|\px\omega_{2,  \neq}\|_{Y_a}^2 + \|u_{2, \neq}\|_{X_a}^2  ),\\&
			\|\mathrm{e}^{ a A^{-\frac{1}{3}} t}  \px n_{ \neq} \nabla c_0\|_{L^2 L^2}^2  +  
			\|\mathrm{e}^{ a A^{-\frac{1}{3}} t} n_0 \px\nabla c_{ \neq} \|_{L^2 L^2}^2 \leq C A^{\frac{5}{6}}(F_3^2 + M^2) \|\px n_{\neq}\|_{Y_a}^2  .
		\end{aligned}
	\end{equation}
\end{Lem}
\begin{proof}
	{\bf Estimate of $(\ref{0 n_neq})_{1}.$}
	Direct calculations show that
	\begin{equation}\begin{aligned} \label{eq:temp17.1}
			\|\mathrm{e}^{ a A^{-\frac{1}{3}} t} \widehat{ u_{1,0}} \px n_{ \neq}\|_{L^2 L^2}^2  \leq C  \| \widehat{ u_{1,0}} \|_{L^\infty H^2}^2 \| \mathrm{e}^{ a A^{-\frac{1}{3}} t} \px n_{ \neq}\|_{L^2 L^2}^2 \leq C A^{\frac73 - 2 \epsilon} F_{1}^2 \|\px n_{\neq}\|_{Y_a}^2.
	\end{aligned}\end{equation}
	Owing to
	\begin{equation*}\begin{aligned}
			\| \widetilde{ u_{1,0}} \px n_{ \neq}\|_{L^2}^2 
			\leq \| \widetilde{ u_{1,0}} \|_{L_y^\infty L_z^2 }^2 \| \px n_{ \neq}\|_{L_{x, y}^2 L_z^\infty}^2 
			\leq C \| \widetilde{ u_{1,0}} \|_{H^1}^2  \| \px n_{ \neq}\|_{L^2}^\frac12 \| \px\nabla n_{ \neq}\|_{L^2}^\frac32 ,
	\end{aligned}\end{equation*}
	we have that
	\begin{equation}\begin{aligned} \label{eq:temp17.2}
			\|\mathrm{e}^{ a A^{-\frac{1}{3}} t} \widetilde{ u_{1,0}} \px n_{ \neq}\|_{L^2 L^2}^2  &\leq C  \| \widetilde{ u_{1,0}} \|_{L^\infty H^1}^2 \| \mathrm{e}^{ a A^{-\frac{1}{3}} t} \px n_{ \neq}\|_{L^2 L^2}^\frac12 \| \mathrm{e}^{ a A^{-\frac{1}{3}} t} \px\nabla n_{ \neq}\|_{L^2 L^2}^\frac32\\
			&\leq C A^{\frac56} F_{1}^2 \|\px n_{\neq}\|_{Y_a}^2.
	\end{aligned}\end{equation}
	The proof of $(\ref{0 n_neq})_{1}$ is given by combining \eqref{eq:temp17.1} and \eqref{eq:temp17.2}.
	
	{\bf Estimate of $(\ref{0 n_neq})_{2}.$} Noticing that
	\begin{equation}
		\begin{aligned} \label{eq:est of u20 u30 L_infty}
			\|u_{2, 0}\|_{L^\infty}  +  \|u_{3, 0}\|_{L^\infty} &\leq C\left(\|u_{2, 0}\|_{H^1}  +  \|\pz u_{2, 0}\|_{H^1}  +  \|u_{3, 0}\|_{H^1}  +  \|\pz u_{3, 0}\|_{H^1}\right) \\
			&\leq  C\left(\|\nabla u_{2, 0}\|_{L^2}  +  \|\Delta u_{2, 	0}\|_{L^2}  +  \|\nabla u_{3, 0}\|_{L^2}  \right), 
		\end{aligned}
	\end{equation}
	for $k\in\{2,3\}$, we have
	\begin{equation*}
		\begin{aligned}
			\|\mathrm{e}^{ a A^{-\frac{1}{3}} t} u_{k, 0} \px n_{ \neq}\|_{L^2 L^2}^2 
			\leq \|u_{k, 0}\|_{L^\infty L^\infty}^2 \|\mathrm{e}^{ a A^{-\frac{1}{3}} t} \px n_{ \neq}\|_{L^2 L^2}^2
			\leq C  A^{\frac13 -\frac{\epsilon}{2}  } F_{1}^2\|\px n_{\neq}\|_{Y_a}^2.
		\end{aligned}
	\end{equation*}
	Similarly, by $\eqref{eq:velocity transform}_2$, we get
	\begin{equation*}
		\begin{aligned}
			\|\mathrm{e}^{ a A^{-\frac{1}{3}} t} n_0 \px u_{ \neq}  \|_{L^2 L^2}^2 &\leq \|n_0\|_{L^\infty L^\infty}^2 \|\mathrm{e}^{ a A^{-\frac{1}{3}} t} \px^2 u_{ \neq}\|_{L^2 L^2}^2 \leq C A^{ \frac13 }F_3^2  (\|\px\omega_{2,  \neq}\|_{Y_a}^2   + \|u_{2, \neq}\|_{X_a}^2).
		\end{aligned}
	\end{equation*}
	
	{\bf Estimate of $(\ref{0 n_neq})_{3}.$} 
	Using $\eqref{eq:0_norm_L_infty}_4$, $\eqref{Sob neq}_6$ and Lemma \ref{lem:the zero mode of c and n}, one has
	$$
	\begin{aligned}
		\|\mathrm{e}^{a A^{-\frac{1}{3}} t} \px n_{\neq} \nabla c_0\|_{L^2 	L^2}^2 & 
		\leq\|\nabla c_0\|_{L^{\infty}_{t,y} L^2_{z}}^2 \|\mathrm{e}^{a 	A^{-\frac{1}{3}} t} \px n_{\neq}\|_{L^{\infty}_z L^2_{t,x,y}}^2 \\
		& \leq C F_3 M\|\mathrm{e}^{a A^{-\frac{1}{3}} t} \px n_{\neq}\|_{L^2 	L^2}^{\frac{1}{2}}
		\|{\rm e}^{a A^{-\frac{1}{3}} t} \nabla \px n_{\neq}\|_{L^2 	L^2}^{\frac{3}{2}} \\
		& \leq C A^{\frac{5}{6}}(F_3^2 + M^2)\|\px n_{\neq}\|_{Y_a}^2 .
	\end{aligned}
	$$
	Then by Lemma \ref{lem:est of the non-zero mode of c and n}, we obtain
	$$
	\begin{aligned}
		\|\mathrm{e}^{a A^{-\frac{1}{3}} t} n_0 \px \nabla c_{\neq}\|_{L^2 	L^2}^2 & \leq \|n_0\|_{L^{\infty} L^{\infty}}^2  
		\|\mathrm{e}^{a A^{-\frac{1}{3}} t}\px \nabla c_{\neq}\|_{L^2 L^2}^2 \\
		& \leq C\|n\|_{L^{\infty} L^{\infty}}^2 \|\mathrm{e}^{a A^{-\frac{1}{3}} 	t} \px n_{\neq}\|_{L^2 L^2}^2 \leq C  A^{\frac{1}{3}} F_3^2 \|\px n_{\neq}\|_{Y_a}^2.
	\end{aligned}
	$$
	The proof is completed.
\end{proof}

\begin{Lem} \label{lem:est of E21}
Under the assumptions of \eqref{assumption}, there exists a positive constant $\mathcal{B}_{5}$ independent of $t$ and $A$, such that if $A\geq\mathcal{B}_{5}$,	it holds that 
	\begin{equation*}\begin{aligned}
			E_{2,1}(t)=\|\partial_{x}n_{\neq}\|_{Y_{a}}\leq C\left(\|(\partial_{x}n_{\rm in})_{\neq}\|_{L^{2}}+1\right).
	\end{aligned}\end{equation*}
\end{Lem}
\begin{proof}
	Applying Proposition \ref{prop:key prop for omega2neq} to $\eqref{eq:mian eq after fourier}_1$, we get for $k_1 \neq 0$ that
		\begin{equation*}\begin{aligned}
			& \|{\rm e}^{a A^{-\frac13} t} n^{k_1, k_3} \|_{L^{\infty} L^2}^2   +  \frac{\|{\rm e}^{a A^{-\frac13} t} \py n^{k_1, k_3}\|_{L^2 L^2}^2}{A}   +  \left(A^{-1} \eta^2+(A^{-1} k_1^2)^{1 / 3}\right)\|{\rm e}^{a A^{-\frac13} t} n^{k_1, k_3}\|_{L^2 L^2}^2 \\
			& \leq C\left(\|n^{k_1, k_3}_{\rm in}\|_{L^2}^2  +  \frac1A \|{\rm e}^{a A^{-\frac13} t} \left(u_2 n\right)^{k_1, k_3}\|_{L^2 L^2}^2  +   \frac1A \|{\rm e}^{a A^{-\frac13} t} \left(n\py c\right)^{k_1, k_3}\|_{L^2 L^2}^2   \right. \\
			& \left.\quad + \frac{1}{A^2}\min \left((A^{-1} \eta^2)^{-1},(A^{-1} k_1^2)^{-1 / 3}\right)\|{\rm e}^{a A^{-\frac13} t}\left(k_1 \left(u_1 n\right)^{k_1, k_3}  +  k_3 \left(u_3 n\right)^{k_1, k_3}\right)\|_{L^2 L^2}^2\right.\\
			&\left. \quad + \frac{1}{A^2}\min \left((A^{-1} \eta^2)^{-1},(A^{-1} k_1^2)^{-1 / 3}\right)\|{\rm e}^{a A^{-\frac13} t}\left(k_1 \left(n\px c\right)^{k_1, k_3}  +  k_3 \left(n\pz c\right)^{k_1, k_3}\right)\|_{L^2 L^2}^2\rt.\\
			&\lt. \quad+ \frac{1}{A^2}\min \left((A^{-1} \eta^2)^{-1},(A^{-1} k_1^2)^{-1 / 3}\right)\|{\rm e}^{a A^{-\frac13} t} (n^2)^{k_1, k_3} \|_{L^2 L^2}^2 \right),
	\end{aligned}\end{equation*}
	which gives
	\begin{equation*}\begin{aligned}
			&\|\px n_{\neq}\|_{Y_a}^2  \\
			\leq& C \sum_{k_1 \neq 0, k_3 \in \mathbb{Z}} \left(\| k_1 n^{k_1, k_3}_{\rm i n}\|_{L^2}^2  +  \frac1A \|{\rm e}^{a A^{-\frac13} t} k_1 \left(u_2 n\right)^{k_1, k_3}\|_{L^2 L^2}^2  +   \frac1A \|{\rm e}^{a A^{-\frac13} t} k_1 \left(n\py c\right)^{k_1, k_3}\|_{L^2  L^2}^2   \right. \\
			& \left. + \frac{1}{A} \|{\rm e}^{a A^{-\frac13} t}k_1 \left(  \left(u_1 n\right)^{k_1, k_3}  +   \left(u_3 n\right)^{k_1, k_3}\right)\|_{L^2 L^2}^2  +  \frac{1}{A} \|{\rm e}^{a A^{-\frac13} t} \left(n^2\right)^{k_1, k_3}  \|_{L^2 L^2}^2  \rt.\\
			&\lt. + \frac{1}{A}\|{\rm e}^{a A^{-\frac13} t}k_1 \left(  \left(n\px c\right)^{k_1, k_3}  +   \left(n\pz c\right)^{k_1, k_3}\right)\|_{L^2 L^2}^2\right) \\
			\leq& C\big( \|(\px n_{\rm in})_{\neq}\|_{L^2}^2  +  \frac{\|{\rm e}^{a A^{-\frac13} t} \px \left(u n\right)_{\neq}\|_{L^2 L^2}^2 }{A}  +   \frac{\|{\rm e}^{a A^{-\frac13} t} \px \left(n\nabla c\right)_{\neq}\|_{L^2 L^2}^2}{A}   +   \frac{\|{\rm e}^{a A^{-\frac13} t}  (n^2)_{\neq}\|_{L^2 L^2}^2}{A}  \big),
	\end{aligned}\end{equation*}
	where we use that
	$$
	\frac{k_1^2}{A^2}\min \left((A^{-1} \eta^2)^{-1},(A^{-1} k_1^2)^{-1 / 3}\right) \leq \frac1A,\quad\frac{k_3^2}{A^2}\min \left((A^{-1} \eta^2)^{-1},(A^{-1} k_1^2)^{-1 / 3}\right) \leq \frac1A.
	$$
		According to $\eqref{eq:u and cell}_1$ and $(\ref{0 n_neq})_{2}$, when
		$$ A\geq\max\{ F_{1}^{12}F_{2}^{12}, (F_{1}^{2}F_{2}^{2})^{\frac{3}{6\epsilon-4}}, (F_{2}^{2}F_{3}^{2})^{\frac{6}{5\epsilon+4}}, F_{2}^{\frac{24}{5\epsilon}}, 1\}=:\mathcal{B}_{5,1}, $$
		there holds
	\begin{equation*}\begin{aligned}
			&\frac1A \|{\rm e}^{a A^{-\frac13} t} \px \left(u n\right)_{\neq}\|_{L^2 L^2}^2 \\
			\leq&  \frac{C}{A}\left(\|{\rm e}^{a A^{-\frac13} t} u_0 \px n_{\neq}\|_{L^2 L^2}^2 + \|{\rm e}^{a A^{-\frac13} t} n_0 \px u_{\neq}\|_{L^2 L^2}^2  + \|{\rm e}^{a A^{-\frac13} t} \px \left(u_{\neq} n_{\neq}\right)_{\neq}\|_{L^2 L^2}^2  \right)  \\
			\leq&  C  \lt(A^{-\frac16} F_{1}^2  + A^{\frac43 - 2 \epsilon} F_{1}^2   +  A^{-\frac23 -\frac12 \epsilon } F_{1}^2 \rt) F_{2}^2 + CA^{-\frac56 \epsilon} \left( A^{-\frac23} F_3^2 +A^{-\frac13} F_{2}^2 \right) F_{2}^2\leq C.
	\end{aligned}\end{equation*}
	Using $\eqref{eq:u and cell}_2$ and $(\ref{0 n_neq})_{3}$, as long as $A\geq \max\{\mathcal{B}_{5,1}, F_{2}^{24}, (F_{3}^{2}+M^{2})^{6}F_{2}^{12}\}=:\mathcal{B}_{5}$, we obtain
	\begin{equation*}\begin{aligned}
			&\frac{1}{A}\|{\rm e}^{a A^{-\frac13} t} \px \left(n\nabla c\right)_{\neq}\|_{L^2 L^2}^2\\
			\leq&  \frac{1}{A}\left( \|{\rm e}^{a A^{-\frac13} t} n_0 \px \nabla c_{\neq}\|_{L^2 L^2}^2 +  \|{\rm e}^{a A^{-\frac13} t} \nabla c_0 \px n_{\neq}\|_{L^2 L^2}^2  +  \|{\rm e}^{a A^{-\frac13} t} \px \left( n_{\neq} \nabla c_{\neq}\right)_{\neq}\|_{L^2 L^2}^2\right)  \\
			\leq&  C A^{-\frac{1}{6}}(F_3^2 + M^2) \|\px n_{\neq}\|_{Y_a}^2 +  CA^{-\frac16}\|\px n_{  \neq}\|_{Y_a}^4
			\leq C A^{-\frac{1}{6}}(F_3^2 + M^2) F_{2}^2 +  CA^{-\frac16}F_{2}^4\leq C.
	\end{aligned}\end{equation*}
 By H\"{o}lder inequality, direct calculations indicate that
\begin{equation*}\begin{aligned}
		\frac1A \|{\rm e}^{a A^{-\frac13} t}  (n^2)_{\neq}\|_{L^2 L^2}^2 &\leq \frac{C}{A}\big( \|{\rm e}^{a A^{-\frac13} t} n_0n_{\neq}\|_{L^2 L^2}^2  + \|{\rm e}^{a A^{-\frac13} t}  \left(n_{\neq} n_{\neq}\right)_{\neq}\|_{L^2 L^2}^2\big)\\
		&\leq \frac C A \|n\|_{L^\infty L^\infty}^2 \|{\rm e}^{a A^{-\frac13} t} \px n_{\neq}\|_{L^2 L^2}^2 \leq CA^{-\frac23} F_3^2 F_{2}^2\leq C
\end{aligned}\end{equation*}
provided with $A\geq\mathcal{B}_{5}.$

	The proof is completed by summing up all above inequalities.
\end{proof}
\subsection{Energy estimate for $E_{2,2}(t)$.}\

The following nonlinear interactions between the zero and non-zero modes will be 
used in estimating $E_{2,2}(t).$

\begin{Lem} \label{lem:hatu10}
	Under the assumptions of \eqref{assumption}, it holds that
	\begin{equation*}\begin{aligned}
			&\quad\|\mathrm{e}^{ a A^{-\frac{1}{3}} t}  \widehat{ u_{1,0}} \px\left(\px, \pz\right)u_{\neq}\|_{L^2 L^2}^2  +  \|\mathrm{e}^{ a A^{-\frac{1}{3}} t}  |\nabla \widehat{ u_{1,0}}| \px u_{\neq}\|_{L^2 L^2}^2 \\
			&\leq CA^{\frac73 - 2\epsilon } F_{1}^2 ( \|\px\omega_{2,  \neq}\|_{Y_a}^2   + \|u_{2, \neq}\|_{X_a}^2  ).
	\end{aligned}\end{equation*}
\end{Lem}

\begin{proof}
	By $\eqref{eq:0_norm_L_infty}_4$, $\eqref{Sob f0}_1$, $\eqref{Sob g neq}_3$ and $\eqref{eq:velocity transform}_2$,  one has
	\begin{equation*}\begin{aligned}
			\|\widehat{ u_{1,0}} \px\left(\px, \pz\right)u_{\neq}\|_{L^2}^2  &\leq \|\widehat{ u_{1,0}} \|_{L^\infty}^2  \| \px\left(\px, \pz\right)u_{\neq}\|_{L^2}^2\\
			&\leq C\|\widehat{ u_{1,0}} \|_{H^2}^2 (\|\px\omega_{2,  \neq} \|_{L^2}^2 + \|\px\nabla u_{2 ,\neq} \|_{L^2}^2)
	\end{aligned}\end{equation*}
	and
	\begin{equation*}\begin{aligned}
			\| |\nabla \widehat{ u_{1,0}}| \px u_{\neq} \|_{L^2}^2  &\leq \|\nabla \widehat{ u_{1,0}} \|_{L_y^\infty  L_z^2}^2\| \px u_{\neq}\|_{L_{x, y}^2 L_z^\infty }^2\\
			&\leq C\|\widehat{ u_{1,0}} \|_{H^2}^2 (\|\px\omega_{2,  \neq}\|_{L^2}^2 + \|\px\nabla u_{2 ,\neq}\|_{L^2}^2).
	\end{aligned}\end{equation*}
	Hence, we arrive at 
	\begin{equation*}\begin{aligned}
			&\quad\|\mathrm{e}^{ a A^{-\frac{1}{3}} t}  \widehat{ u_{1,0}} \px\left(\px, \pz\right)u_{\neq}\|_{L^2 L^2}^2  +  \|\mathrm{e}^{ a A^{-\frac{1}{3}} t} |\nabla \widehat{ u_{1,0}}| \px u_{\neq} \|_{L^2 L^2}^2 \\
			& \leq CA^\frac13 \|\widehat{ u_{1,0}} \|_{L^\infty H^2}^2( \|\px\omega_{2,  \neq}\|_{Y_a}^2   + \|u_{2, \neq}\|_{X_a}^2)
			\leq CA^{\frac73 - 2\epsilon } F_{1}^2 ( \|\px\omega_{2,  \neq}\|_{Y_a}^2   + \|u_{2, \neq}\|_{X_a}^2  ),
	\end{aligned}\end{equation*}
	which completes the proof.
\end{proof}

\begin{Lem} \label{lem:tildeu10}
	Under the assumptions of \eqref{assumption}, it holds that
	\begin{equation*}\begin{aligned}
			&\quad\|\mathrm{e}^{ a A^{-\frac{1}{3}} t}  \widetilde{ u_{1,0}} \px\left(\px, \pz\right)u_{\neq}\|_{L^2 L^2}^2  +  \|\mathrm{e}^{ a A^{-\frac{1}{3}} t} |\nabla \widetilde{ u_{1,0}}| \px u_{\neq} \|_{L^2 L^2}^2 \\
			&\leq CA^\frac56 F_{1}^2 ( \|\px\omega_{2,  \neq}\|_{Y_a}^2   + \|u_{2, \neq}\|_{X_a}^2  ).
	\end{aligned}\end{equation*}
\end{Lem}

\begin{proof}
	By the use of $\eqref{Sob f0}_3$, $\eqref{Sob neq}_6$, $\eqref{Sob g neq}_2$ and $\eqref{eq:velocity transform}_4$, we get that
	\begin{equation*}\begin{aligned}
			&\quad \|\widetilde{ u_{1,0}} \px\left(\px, \pz\right)u_{\neq}\|_{L^2}^2
			\leq \|\widetilde{ u_{1,0}} \|_{L_y^\infty L_z^2}^2  \| \px\left(\px, \pz\right)u_{\neq}\|_{L_{x, y}^2 L_z^\infty}^2\\
			&\leq C\|\widetilde{ u_{1,0}} \|_{H^1}^2 \big(\|\px \omega_{2,  \neq} \|_{L^2}^\frac12 + \|\px\nabla u_{2 ,\neq} \|_{L^2}^\frac12\big)
			\big(\|\px\nabla \omega_{2,\neq}\|_{L^2}^\frac32 + \|\px\Delta u_{2 ,\neq}\|_{L^2}^\frac32\big)
	\end{aligned}\end{equation*}
	and
	\begin{equation*}\begin{aligned}
			&\quad\| |\nabla \widetilde{ u_{1,0}}| \px u_{\neq} \|_{L^2}^2 \leq \|\nabla \widetilde{ u_{1,0}} \|_{L^2}^2\| \px u_{\neq}\|_{L_{x}^2 L_{y,z}^\infty }^2\\
			&\leq C\|\widetilde{ u_{1,0}} \|_{H^1}^2 \left(\|\px\omega_{2,  \neq} \|_{L^2} + \|\px\nabla u_{2 ,\neq} \|_{L^2}\right)\left(\|\px \nabla \omega_{2,  \neq} \|_{L^2} 
			+ \|\px\Delta u_{2 ,\neq} \|_{L^2}\right),
	\end{aligned}\end{equation*}
	which indicate that
	\begin{equation*}\begin{aligned}
			&\quad\|\mathrm{e}^{ a A^{-\frac{1}{3}} t}  \widetilde{ u_{1,0}} \px\left(\px, \pz\right)u_{\neq}\|_{L^2 L^2}^2  +  \|\mathrm{e}^{ a A^{-\frac{1}{3}} t} \px u_{\neq} \cdot \nabla \widetilde{ u_{1,0}}\|_{L^2 L^2}^2 \\
			& \leq CA^\frac56 \|\widetilde{ u_{1,0}} \|_{L^\infty H^1}^2 (\|\px\omega_{2,  \neq}\|_{Y_a}^2   + \|u_{2, \neq}\|_{X_a}^2).
	\end{aligned}\end{equation*}

	The proof is complete.
\end{proof}

\begin{Lem}\label{lem:px u0u1}
	Under the assumptions of \eqref{assumption}, it holds that
	\begin{itemize}
		\item[(i)] 
		$
		\quad\|\mathrm{e}^{ a A^{-\frac{1}{3}} t} \px \left(u_0\cdot \nabla u_{1, \neq}\right)\|_{L^2 L^2}^2  +  \|\mathrm{e}^{ a A^{-\frac{1}{3}} t} \px \left(u_{\neq}\cdot \nabla u_{1, 0}\right)\|_{L^2 L^2}^2 \\
		\leq  C(A^\frac56 F_{1}^2   +  A^{\frac73 - 2\epsilon } F_{1}^2 +  A^{1 - \frac12\epsilon} F_{1}^2 )(\|\px \omega_{2,  \neq}\|_{Y_a}^2  + \| u_{2,  \neq}\|_{X_a}^2),$
		\item[(ii)] 
		$\quad\|\mathrm{e}^{ a A^{-\frac{1}{3}} t} \left(\px, \pz\right) \left(u_0\cdot \nabla u_{2, \neq}\right)\|_{L^2 L^2}^2  +  \|\mathrm{e}^{ a A^{-\frac{1}{3}} t} \left(\px, \pz\right) \left(u_{\neq}\cdot \nabla u_{2, 0}\right)\|_{L^2 L^2}^2 \\
		\leq  C(A^\frac56 F_{1}^2   +  A^{\frac73 - 2\epsilon } F_{1}^2 +  A^{1 - \frac12\epsilon} F_{1}^2 )(\|\px \omega_{2,  \neq}\|_{Y_a}^2  + \| u_{2,  \neq}\|_{X_a}^2  ),$
		\item[(iii)] 
		$
		\quad\|\mathrm{e}^{ a A^{-\frac{1}{3}} t} \left(\px, \pz\right) \left(u_0\cdot \nabla u_{3, \neq}\right)\|_{L^2 L^2}^2  +  \|\mathrm{e}^{ a A^{-\frac{1}{3}} t} \left(\px, \pz\right) \left(u_{\neq}\cdot \nabla u_{3, 0}\right)\|_{L^2 L^2}^2 \\
		\leq CA^{1- \frac32\epsilon}  F_{2}^2 \left( A^{-\frac16} F_{1}^2  +  A^{\frac43 - 2 \epsilon} F_{1}^2  +  A^{ - \frac12 \epsilon} F_{1}^2\right).
		$
	\end{itemize}
	
\end{Lem}
\begin{proof}
	{\bf Estimate of (i).}	Using  $\eqref{eq:velocity transform}_4$ and \eqref{eq:est of u20 u30 L_infty}, we get
	\begin{equation*}\begin{aligned}
			&\quad\| \px \left(u_0\cdot \nabla u_{1, \neq}\right)\|_{ L^2}^2\\
			&\leq C\left(\| \widehat{ u_{1,0}} \px^2 u_{1, \neq}\|_{ L^2}^2 + \| \widetilde{ u_{1,0}} \px^2 u_{1, \neq}\|_{ L^2}^2  + \|u_{2, 0}\|_{ L^\infty }^2  
			\|\px\py u_{1, \neq}\|_{ L^2 }^2 +\|u_{3, 0}\|_{ L^\infty }^2  \|\px\pz u_{1, \neq}\|_{ L^2 }^2\right)\\
			&\leq C\| (\widehat{ u_{1,0}},\widetilde{ u_{1,0}} )\px^2 u_{ \neq}\|_{ L^2}^2 + C(\| u_{2, 0}\|_{H^2}^2 +  \|\nabla u_{3, 0}\|_{L^2}^2  ) (\|\px \nabla \omega_{2,  \neq}\|_{L^2}^2  + \|\px \Delta u_{2,  \neq}\|_{L^2}^2  ),
	\end{aligned}\end{equation*}
	from which, along with Lemma \ref{lem:hatu10} and Lemma \ref{lem:tildeu10}, we infer that
	\begin{equation*}\begin{aligned}
			&\quad\|\mathrm{e}^{a A^{-\frac{1}{3}} t} \px \left(u_0\cdot \nabla u_{1, \neq}\right)\|_{L^2 L^2}^2 +  \|\mathrm{e}^{ a A^{-\frac{1}{3}} t} \px \left(u_{\neq}\cdot \nabla u_{1, 0}\right)\|_{L^2 L^2}^2\\
			&\leq CA^{\frac73 - 2\epsilon } F_{1}^2 ( \|\px\omega_{2,  \neq}\|_{Y_a}^2   + \|u_{2, \neq}\|_{X_a}^2  )  +   CA^\frac56 F_{1}^2 ( \|\px\omega_{2,  \neq}\|_{Y_a}^2   + \|u_{2, \neq}\|_{X_a}^2  )\\
			&\quad + CA(\|\nabla u_{2, 0}\|_{L^\infty L^2}^2  +  \|\Delta u_{2, 0}\|_{L^\infty L^2}^2  +  \|\nabla u_{3, 0}\|_{L^\infty L^2}^2  ) (\|\px \omega_{2,  \neq}\|_{Y_a}^2  + \| u_{2,\neq}\|_{X_a}^2  ) \\
			&\leq C(A^\frac56 F_{1}^2   +  A^{\frac73 - 2\epsilon } F_{1}^2 +  A^{1 - \frac12\epsilon} F_{1}^2 )(\|\px \omega_{2, \neq}\|_{Y_a}^2  + \| u_{2, \neq}\|_{X_a}^2  ).
	\end{aligned}\end{equation*}
	
	{\bf Estimate of (ii).}
	There holds that
	\begin{equation*}\begin{aligned}
			\|\left(\px, \pz\right)\left(u_{1, 0}\px u_{2, \neq}\right)\|_{L^2}^2  \leq \|\pz u_{1, 0}\px u_{2, \neq}\|_{L^2}^2 + \|u_{1, 0} \px \left(\px, \pz\right) u_{2, \neq}\|_{L^2}^2.
	\end{aligned}\end{equation*}
	For $k\in\{2,3\}$, one gets
	\begin{equation*}\begin{aligned}
			\|\left(\px, \pz\right)\left(u_{k, 0}\partial_{k} u_{2, \neq}\right)\|_{L^2}^2  &\leq \|u_{k, 0}\|_{L^\infty}^2 \|\partial_{k}\left(\px, \pz\right)u_{2, \neq}\|_{L^2}^2  
			+ \|\pz u_{k, 0}\|_{L_y^\infty L_z^2}^2 \|\partial_{k} u_{2, \neq}\|_{L_z^\infty L_{x, y}^2}^2,
	\end{aligned}\end{equation*}
	which, combined with \eqref{eq:est of u20 u30 L_infty} and $\py u_{2, 0}  +  \pz u_{3, 0}  = 0$, implies
	\begin{equation*}\begin{aligned}
			\|\left(\px, \pz\right)\left(u_{k, 0}\partial_{k} u_{2, \neq}\right)\|_{L^2}^2  &\leq C(\|\nabla u_{2, 0}\|_{L^2}^2+  \|\Delta u_{2, 0}\|_{L^2}^2  +  \|\nabla u_{3, 0}\|_{L^2}^2) \|\Delta u_{2, \neq}\|_{L^2}^2.
	\end{aligned}\end{equation*}
	Similarly, by  $\eqref{eq:0_norm_L_infty}_4$, $\eqref{Sob neq}_6$, $\eqref{Sob g neq}_1$ and $\eqref{eq:velocity transform}_4$, we have
	\begin{equation*}\begin{aligned}
			\|\left(\px, \pz\right)\left(u_{\neq}\cdot\nabla u_{2, 0}\right)\|_{L^2}^2  &\leq C\left( \|\nabla u_{2, 0}\|_{L_{y}^\infty L_z^2}^2 \|\left(\px, \pz\right)u_{\neq}\|_{L_{x, y}^2 L_z^\infty}^2  + \|u_{\neq}\|_{L^\infty}^2 \|\pz\nabla u_{2, 0}\|_{ L^2}^2\right) \\
			& \leq C\| u_{2, 0}\|_{H^2}^2\left( \|\px\nabla \omega_{2, \neq}\|_{L^2}^2  +  \|\px\Delta  u_{2, \neq}\|_{L^2}^2 \right).
	\end{aligned}\end{equation*}
	By collecting the inequalities above and applying Lemma \ref{lem:hatu10} and Lemma \ref{lem:tildeu10}, we obtain that
	\begin{equation*}\begin{aligned}
			&\quad\|\mathrm{e}^{ a A^{-\frac{1}{3}} t} \left(\px, \pz\right) \left(u_0\cdot \nabla u_{2, \neq}\right)\|_{L^2 L^2}^2+  \|\mathrm{e}^{ a A^{-\frac{1}{3}} t} \left(\px, \pz\right) \left(u_{\neq}\cdot \nabla u_{2, 0}\right)\|_{L^2 L^2}^2 \\
			&\leq C\left(A^\frac56 F_{1}^2   +  A^{\frac73 - 2\epsilon } F_{1}^2 +  A^{1 - \frac12\epsilon} F_{1}^2 \right)(\|\px \omega_{2,  \neq}\|_{Y_a}^2  + \| u_{2,  \neq}\|_{X_a}^2  ).
	\end{aligned}\end{equation*}
	
	{\bf Estimate of (iii).} 
	Since
	\begin{equation*}\begin{aligned}
			\|\left(\px, \pz\right)\left(u_{1, 0}\px u_{3, \neq}\right)\|_{L^2}^2  \leq \| \pz u_{1, 0}\px u_{3, \neq}\|_{L^2}^2 + \|u_{1, 0} \px \left(\px, \pz\right)u_{3, \neq}\|_{L^2}^2.
	\end{aligned}\end{equation*}
	For $k\in\{2,3\}$, there holds
	\begin{equation*}\begin{aligned}
			\|\left(\px, \pz\right)\left(u_{k, 0}\partial_{k} u_{3, \neq}\right)\|_{L^2}^2  &\leq \|u_{k, 0}\|_{L^\infty}^2 \|\partial_{k}\left(\px, \pz\right)u_{3, \neq}\|_{L^2}^2  
			+ \|\pz u_{k, 0}\|_{L_y^\infty L_z^2}^2 \|\partial_{k} u_{3, \neq}\|_{L_z^\infty L_{x, y}^2}^2,
	\end{aligned}\end{equation*}
	which  implies
	\begin{equation*}\begin{aligned}
			\|\left(\px, \pz\right)\left(u_{k, 0}\partial_{k} u_{3, \neq}\right)\|_{L^2}^2  
			\leq C\left(\| u_{2, 0}\|_{H^2}^2  +   \| u_{3, 0}\|_{H^1}^2  \right) \left(  \|\px\nabla \omega_{2, \neq}\|_{L^2}^2 +  \|\px \Delta u_{2, \neq}\|_{L^2}^2\right).
	\end{aligned}\end{equation*}
	Similarly, by $\eqref{eq:0_norm_L_infty}_3$, $\eqref{Sob g neq}_{1,5}$ and $\eqref{eq:velocity transform}_4$, we get
	\begin{equation*}\begin{aligned}
			\|\left(\px, \pz\right)\left(u_{\neq}\cdot \nabla u_{3, 0}\right)\|_{L^2}^2  &\leq \|\nabla u_{3, 0}\|_{L_{z}^\infty L_y^2}^2 \|\left(\px, \pz\right)u_{\neq}\|_{L_{x, z}^2 L_y^\infty}^2  + \|u_{\neq}\|_{L^\infty}^2 \|\pz\nabla u_{3, 0}\|_{ L^2}^2 \\
			& \leq C\left(\|\Delta u_{2, 0}\|_{L^2}^2  +  \|\nabla u_{3, 0}\|_{L^2}^2 \right)\left( \|\px\nabla \omega_{2, \neq}\|_{L^2}^2  +  \|\px\Delta  u_{2, \neq}\|_{L^2}^2 \right).
	\end{aligned}\end{equation*}
	By Lemma \ref{lem:hatu10} and Lemma \ref{lem:tildeu10}, we infer from above results that
	\begin{equation*}\begin{aligned}
			&\quad\|\mathrm{e}^{ a A^{-\frac{1}{3}} t} \left(\px, \pz\right) \left(u_0\cdot \nabla u_{3, \neq}\right)\|_{L^2 L^2}^2  +  \|\mathrm{e}^{ a A^{-\frac{1}{3}} t} \left(\px, \pz\right) \left(u_{\neq}\cdot \nabla u_{3, 0}\right)\|_{L^2 L^2}^2 \\
			&\leq  C\left(A^\frac56 F_{1}^2   +  A^{\frac73 - 2\epsilon } F_{1}^2 +  A^{1 - \frac12\epsilon} F_{1}^2 \right)\left(\|\px \omega_{2,  \neq}\|_{Y_a}^2  + \| u_{2,  \neq}\|_{X_a}^2  \right).
	\end{aligned}\end{equation*}
	The proof is complete.
\end{proof}

\begin{Lem} \label{lem:est of E22}
	Under the assumptions of \eqref{eq:smallness of u_in} and \eqref{assumption}, if $A\geq \mathcal{B}_{5}$, it holds that
	\begin{equation*}\begin{aligned}
			E_{2,2}(t)\leq C\left(\|(\partial_{x}n_{\rm in})_{\neq}\|_{L^{2}}+1\right).
	\end{aligned}\end{equation*}
\end{Lem}
\begin{proof}
	Applying Proposition \ref{prop:key prop for omega2neq} to $\eqref{eq:mian eq after fourier}_2$, we get
	\begin{equation}\begin{aligned} \label{eq:est of omega2neq}
			&\|{\rm e}^{a A^{-\frac13} t} \omega_2^{k_1, k_3}\|_{L^{\infty} L^2}^2   +  \frac1A\|{\rm e}^{a A^{-\frac13} t} \py\omega_2^{k_1, k_3}\|_{L^2 L^2}^2  
			+\big(A^{-1} \eta^2+(A^{-1} k_1^2)^{1 / 3}\big)\|{\rm e}^{a A^{-\frac13} t} \omega_2^{k_1, k_3}\|_{L^2 L^2}^2 \\
			\leq& C\left(\|\omega_{2,\rm in}^{k_1, k_3}\|_{L^2}^2    +  k_3^2(\eta|k_1|)^{-1}\|{\rm e}^{a A^{-\frac13} t} \partial_y u_2^{k_1, k_3}\|_{L^2 L^2}^2
			+  k_3^2\eta|k_1|^{-1}\|{\rm e}^{a A^{-\frac13} t} u_2^{k_1, k_3}\|_{L^2 L^2}^2\right) \\
			&+ C\frac{ \min \big((A^{-1} \eta^2)^{-1},(A^{-1} k_1^2)^{-1 / 3}\big)}{A^2}\|{\rm e}^{a A^{-\frac13} t}\big(k_1 (u\cdot\nabla u_3)^{k_1, k_3}+k_3 n^{k_1, k_3}\big)\|_{L^2 L^2}^2\\&+C\frac{ \min \big((A^{-1} \eta^2)^{-1},(A^{-1} k_1^2)^{-1 / 3}\big)}{A^2}\|{\rm e}^{a A^{-\frac13} t}k_3 (u\cdot\nabla u_1)^{k_1, k_3}\|_{L^2 L^2}^2.
	\end{aligned}\end{equation}
	Then, noting that for $k_1\neq 0$, it holds
	$$
	\frac{k_1^2}{A^2}\min \big((A^{-1} \eta^2)^{-1},(A^{-1} k_1^2)^{-1 / 3}\big) \leq \frac1A,\quad\frac{k_3^2}{A^2}\min \big((A^{-1} \eta^2)^{-1},(A^{-1} k_1^2)^{-1 / 3}\big) \leq \frac1A
	$$
	and 
	\begin{equation*}\begin{aligned}
			&\sum_{k_1\neq 0, k_3 \in \mathbb{Z}}  k_1^2 \left( k_3^2(\eta|k_1|)^{-1} \|{\rm e}^{a A^{-\frac13} t} \partial_y u_2^{k_1, k_3}\|_{L^2 L^2}^2
			+  k_3^2\eta|k_1|^{-1}\|{\rm e}^{a A^{-\frac13} t} u_2^{k_1, k_3}\|_{L^2 L^2}^2 \right) \\
			&\leq \sum_{k_1\neq 0, k_3 \in \mathbb{Z}} \abs{k_1} \eta \|{\rm e}^{a A^{-\frac13} t} \left(-\partial_y, i\eta\right) u_2^{k_1, k_3}\|_{L^2 L^2}^2 \leq C \|u_{2, \neq}\|_{X_a}^2.
	\end{aligned}\end{equation*}
	Multiplying $k_1^2$ on both sides of \eqref{eq:est of omega2neq} and summing over $k_1, k_3$, we infer that
	\begin{equation}\begin{aligned} \label{eq:est of px omega2neq Ya}
			&\|\px \omega_{2, \neq}\|_{Y_a}^2 \leq C\Big(  \sum_{k_1\neq 0, k_3 \in \mathbb{Z}} \|k_1\omega_{2, \rm in}^{k_1, k_3}\|_{L^2}^2   +   \|u_{2, \neq}\|_{X_a}^2 \\
			&+ \frac1A \sum_{k_1\neq 0, k_3 \in \mathbb{Z}} \|{\rm e}^{a A^{-\frac13} t}  \left(k_1 \left(u\cdot\nabla u_3\right)^{k_1, k_3}  +  k_1 n^{k_1, k_3} - k_1 \left(u\cdot\nabla u_1\right)^{k_1, k_3}\right)\|_{L^2 L^2}^2 \Big) \\
			\leq& C\left(  \|(\omega_{2, \rm in})_{\neq}\|_{H^1}^2  +  \frac1A \|{\rm e}^{a A^{-\frac13} t} \px\left(u\cdot \nabla u_1\right)_{\neq}\|_{L^2L^2}^2   +  \frac1A \|{\rm e}^{a A^{-\frac13} t} \px\left(u\cdot \nabla u_3\right)_{\neq}\|_{L^2L^2}^2  \rt.\\
			&\lt. +  \frac1A \|{\rm e}^{a A^{-\frac13} t} \px n_{\neq}\|_{L^2L^2}^2  + \|u_{2, \neq}\|_{X_a}^2 \right).
	\end{aligned}\end{equation}
	Applying Proposition \ref{prop:key prop for delta u2neq} to $\eqref{eq:mian eq after fourier}_3$, we get
	$$
	\begin{aligned}
		& |k_1 \eta|\|{\rm e}^{a A^{-\frac{1}{3}} t}\left(\partial_y, \eta\right) u_2^{k_1, k_3}\|_{L^2 L^2}^2 +  A^{-\frac{3}{2}}\|{\rm e}^{a A^{-\frac{1}{3}} t} \partial_y {\Delta} u_2^{k_1, k_3} \|_{L^2 L^2}^2 + A^{-1} \eta^2\|{\rm e}^{a A^{-\frac{1}{3}} t} {\Delta} u_2^{k_1, k_3} \|_{L^2 L^2}^2 \\
		& +   \eta^2\|{\rm e}^{a A^{-\frac{1}{3}} t}\left(\partial_y, \eta\right)  u_2^{k_1, k_3}\|_{L^{\infty} L^2}^2  + A^{-\frac{1}{2}}\|{\rm e}^{a A^{-\frac{1}{3}} t} {\Delta} u_2^{k_1, k_3}\|_{L^{\infty} L^2}^2 \\
		\leq&  C\left(\eta^{-2}\|\partial_y {\Delta} u_{2, \rm in}^{k_1, k_3}\|_{L^2}^2  +  \|{\Delta} u_{2, \rm in}^{k_1, k_3}\|_{L^2}^2\right)  +   \frac CA \|{\rm e}^{a A^{-\frac{1}{3}} t}\left(k_1, k_3\right)\left(u\cdot\nabla u_2\right)^{k_1, k_3}\|_{L^2 L^2}^2 \\
		&+ \frac CA  \|{\rm e}^{a A^{-\frac{1}{3}} t}  \left[\px\left(u\cdot\nabla u_1\right) +  \pz\left(u\cdot\nabla u_3\right)\right]^{k_1 , k_3}  \|_{L^2 L^2}^2  + \frac CA  \|{\rm e}^{a A^{-\frac{1}{3}} t} k_1 n^{k_1, k_3}  \|_{L^2 L^2}^2,
	\end{aligned}
	$$
	which implies
	\begin{equation*}\begin{aligned}  \label{eq:est of u2neq Xa}
			\| u_{2, \neq} \|_{X_a}^2  \leq& C \left(   \| (u_{\rm in})_{\neq}\|_{H^2}^2  + \frac1A \|{\rm e}^{a A^{-\frac{1}{3}} t}\px\left(u\cdot\nabla u_1\right)_{\neq}\|_{L^2 L^2}^2
			+  \frac1A \|{\rm e}^{a A^{-\frac{1}{3}} t}\pz\left(u\cdot\nabla u_3\right)_{\neq}\|_{L^2 L^2}^2  \rt.\\
			&\lt.+  \frac1A \|{\rm e}^{a A^{-\frac{1}{3}} t}\left(\px, \pz\right)\left(u\cdot\nabla u_2\right)_{\neq}\|_{L^2 L^2}^2
			+  \frac1A \|{\rm e}^{a A^{-\frac13} t} \px n_{\neq}\|_{L^2L^2}^2 \right).
	\end{aligned}\end{equation*}
	Combining the above inequality with \eqref{eq:est of px omega2neq Ya}, it holds
	\begin{equation}\begin{aligned} \label{eq:est of omega2neq Ya and u2neq Xa}
			&A^{\frac56\epsilon}\left(\|\px \omega_{2, \neq}\|_{Y_a}^2  +\| u_{2, \neq} \|_{X_a}^2 \right)
			\\\leq& C\Big(A^{\frac56\epsilon}\|(u_{\rm in})_{\neq}\|_{H^2}^2   +  \frac{1}{A^{1-\frac56\epsilon}} \|{\rm e}^{a A^{-\frac13} t} \px\left(u\cdot \nabla u_1\right)_{\neq}\|_{L^2L^2}^2+  \frac{1}{A^{1-\frac56\epsilon}} \|{\rm e}^{a A^{-\frac13} t} \px n_{\neq}\|_{L^2L^2}^2 \\
			&+\frac{1}{A^{1-\frac56\epsilon}}\|{\rm e}^{a A^{-\frac{1}{3}} t}\left(\px, \pz\right)\left(u\cdot\nabla u_2\right)_{\neq}\|_{L^2 L^2}^2  +   \frac{1}{A^{1-\frac56\epsilon}} \|{\rm e}^{a A^{-\frac13} t} \left(\px, \pz\right)\left(u\cdot \nabla u_3\right)_{\neq}\|_{L^2L^2}^2\Big).
	\end{aligned}\end{equation}
	Next, we estimate each term on the right-hand side of equation (\ref{eq:est of omega2neq Ya and u2neq Xa}).
	
	{\bf Estimate of $\frac{1}{A^{1-\frac56\epsilon}} \|{\rm e}^{a A^{-\frac13} t} \px\left(u\cdot \nabla u_1\right)_{\neq}\|_{L^2L^2}^2$.}
	By Lemma \ref{lem:px u0u1}, one gets that
	\begin{equation*}\begin{aligned} 
			&\|{\rm e}^{a A^{-\frac13} t} \px\left(u\cdot \nabla u_1\right)_{\neq}\|_{L^2L^2}^2
			\\\leq& C\Big( \|{\rm e}^{a A^{-\frac13} t} \px\left(u_0\cdot \nabla u_{1, \neq }\right)\|_{L^2L^2}^2 + \|{\rm e}^{a A^{-\frac13} t} \px\left(u_{\neq}\cdot \nabla u_{1, 0}\right)\|_{L^2L^2}^2
			\\&+  \|{\rm e}^{a A^{-\frac13} t} \px\left(u_{\neq}\cdot \nabla u_{1, \neq}\right)_{\neq}\|_{L^2L^2}^2\Big) \\
			\leq& CA^{1-\frac56\epsilon}\left(A^{-\frac16} F_{1}^2   +  A^{\frac43 - 2\epsilon } F_{1}^2 +  A^{ - \frac12\epsilon} F_{1}^2 \right)F_{2}^2 +  \|{\rm e}^{a A^{-\frac13} t} \px\left(u_{\neq}\cdot \nabla u_{1, \neq}\right)\|_{L^2L^2}^2,
	\end{aligned}\end{equation*}
	which along with $\eqref{eq:velocity estimate 2}_3$ shows that
	\begin{equation}\begin{aligned} \label{eq: est of non u1}
			\frac{\|{\rm e}^{a A^{-\frac13} t} \px\left(u\cdot \nabla u_1\right)_{\neq}\|_{L^2L^2}^2}{A^{1-\frac56\epsilon}}
			\leq C F_{2}^2 \left(A^{-\frac16} F_{1}^2   +  A^{\frac43 - 2\epsilon } F_{1}^2 +  A^{- \frac12\epsilon} F_{1}^2 + A^{- \frac56\epsilon}  F_{2}^2\right)\leq C
	\end{aligned}\end{equation}
	provided with $A\geq\mathcal{B}_{5}.$
	
	{\bf Estimate of $\frac{1}{A^{1-\frac56\epsilon}} \|{\rm e}^{a A^{-\frac{1}{3}} t}\partial_{x}n_{\neq}\|_{L^2 L^2}^2$.} When
	$ A\geq\mathcal{B}_{5},$ by Lemma \ref{lem:est of E21},
	one deduces
	\begin{equation}\label{n}
		\frac{1}{A^{1-\frac56\epsilon}}\|{\rm e}^{aA^{-\frac13}t}\partial_{x}n_{\neq}\|_{L^{2}L^{2}}^{2}\leq\frac{C}{A^{\frac23-\frac56\epsilon}}\|\partial_{x}n_{\neq}\|_{Y_{a}}^{2}
		\leq C\left(\|(\partial_{x}n_{\rm in})_{\neq}\|_{L^{2}}^2+1\right).
	\end{equation}
	
	{\bf Estimate of $\frac{1}{A^{1-\frac56\epsilon}} \|{\rm e}^{a A^{-\frac{1}{3}} t}\left(\px, \pz\right)\left(u\cdot\nabla u_2\right)_{\neq}\|_{L^2 L^2}^2$.}
	It holds that
	\begin{equation*}\begin{aligned} 
			&\|{\rm e}^{a A^{-\frac{1}{3}} t}\left(\px, \pz\right)\left(u\cdot\nabla u_2\right)_{\neq}\|_{L^2 L^2}^2 
			\leq  C\big( \|{\rm e}^{a A^{-\frac{1}{3}} t}\left(\px, \pz\right)\left(u_0 \cdot\nabla u_{2, \neq}\right)\|_{L^2 L^2}^2 
			\\&+  \|{\rm e}^{a A^{-\frac{1}{3}} t}\left(\px, \pz\right)\left(u_{\neq }\cdot\nabla u_{2, 0}\right)\|_{L^2 L^2}^2
			+ \|{\rm e}^{a A^{-\frac{1}{3}} t}\left(\px, \pz\right)\left(u_{\neq}\cdot\nabla u_{2, \neq}\right)_{\neq}\|_{L^2 L^2}^2\big).
	\end{aligned}\end{equation*}
	Then using (ii) of Lemma \ref{lem:px u0u1} and $\eqref{eq:velocity estimate 2}_5$, when $A\geq\mathcal{B}_{5}$, we obtain that
	\begin{equation}\begin{aligned} \label{eq: est of non u2}
			\frac{\|{\rm e}^{a A^{-\frac{1}{3}} t}\left(\px, \pz\right)\left(u\cdot\nabla u_2\right)_{\neq}\|_{L^2 L^2}^2}{A^{1-\frac56\epsilon}} 
			\leq C F_{2}^2 \left(A^{-\frac16} F_{1,}^2   +  A^{\frac43 - 2\epsilon } F_{1}^2 +  A^{ - \frac12\epsilon} F_{1}^2 + A^{- \frac56\epsilon}  F_{2}^2\right)\leq C.
	\end{aligned}\end{equation}
	
	{\bf Estimate of $\frac{1}{A^{1-\frac56\epsilon}} \|{\rm e}^{a A^{-\frac13} t} \left(\px, \pz\right)\left(u\cdot \nabla u_3\right)_{\neq}\|_{L^2L^2}^2$.} 
	By (iii) of Lemma \ref{lem:px u0u1} and $\eqref{eq:velocity estimate 2}_{3, 4}$, as long as $A\geq\mathcal{B}_{5}$, we get
	\begin{equation}\begin{aligned} \label{eq: est of non u3}
			& \frac{1}{A^{1-\frac56\epsilon}}\|{\rm e}^{a A^{-\frac{1}{3}} t}\left(\px, \pz\right)\left(u\cdot\nabla u_3\right)_{\neq}\|_{L^2 L^2}^2 
			\leq \frac{C}{A^{1-\frac56\epsilon}}\big(  \|{\rm e}^{a A^{-\frac{1}{3}} t}\left(\px, \pz\right)\left(u_0 \cdot\nabla u_{3, \neq}\right)\|_{L^2 L^2}^2 
			\\&+  \|{\rm e}^{a A^{-\frac{1}{3}} t}\left(\px, \pz\right)\left(u_{\neq }\cdot\nabla u_{3, 0}\right)\|_{L^2 L^2}^2 
			+ \|{\rm e}^{a A^{-\frac{1}{3}} t}\left(\px, \pz\right)\left(u_{\neq}\cdot\nabla u_{3, \neq}\right)_{\neq}\|_{L^2 L^2}^2\big)\\&
			\leq	 C F_{2}^2 \left(A^{-\frac16} F_{1}^2   +  A^{\frac43 - 2\epsilon } F_{1}^2 +  A^{ - \frac12\epsilon} F_{1}^2 + A^{- \frac56\epsilon}  F_{2}^2\right)\leq C.
	\end{aligned}\end{equation}
	Adding  \eqref{eq:est of omega2neq Ya and u2neq Xa}, \eqref{eq: est of non u1}, \eqref{n}, \eqref{eq: est of non u2} and \eqref{eq: est of non u3} together, and using assumption (\ref{eq:smallness of u_in}), there holds
	\begin{equation*}\begin{aligned}
			E_{2,2}(t)&=A^{\frac{5}{12}\epsilon}\left(\|\partial_{x}\omega_{2,  \neq}\|_{Y_{a}}+\|u_{2,\neq}\|_{X_{a}} \right)\leq C\left(A^{\frac{5}{12}\epsilon}\|(u_{\rm in})_{\neq}\|_{H^{2}}
			+\|(\partial_{x}n_{\rm in})_{\neq}\|_{L^{2}}+1\right)\\
			&\leq C\left(\|(\partial_{x}n_{\rm in})_{\neq}\|_{L^{2}}+1\right). 		
	\end{aligned}\end{equation*}
	
	The proof is complete.
\end{proof}

Adding Lemma \ref{lem:est of E21} and Lemma \ref{lem:est of E22} together, we obtain the following estimate for the energy $E_{2}(t)$ of the non-zero modes.

\begin{Cor}\label{Cor:E2 end}
	As long as $A\geq \mathcal{B}_{5},$	it follows from Lemma \ref{lem:est of E21} and Lemma \ref{lem:est of E22} that
	\begin{equation*}
		E_{2}(t)\leq C\left(\|(\partial_{x}n_{\rm in})_{\neq}\|_{L^{2}}+1 \right)=:F_{2}.
	\end{equation*}
\end{Cor}

	\section{Estimate for $L^{\infty}$-norm of the density $E_{3}(t)$: Proof of Proposition \ref{prop:F3}}\label{Sec 6}
\begin{proof}[Proof of Proposition \ref{prop:F3}]
	For $p=2^j~(j \in \mathbb{Z}^+)$, multiplying $\eqref{eq:main1}_1$ by $2 p n^{2 p-1}$, and integrating over $\mathbb{T} \times \mathbb{I} \times \mathbb{T}$, it holds
		\begin{equation}
		\begin{aligned}
			&\frac{d}{dt}\|n^p\|^2_{L^2}+\frac{2(2p-1)}{Ap}\|\nabla n^{p}\|_{L^2}^2
		\\=&\frac{2(2p-1)}{A}\int_{\mathbb{T} \times \mathbb{I} \times \mathbb{T}}n^{p}\nabla c\cdot\nabla n^{p}dxdydz-\frac{2p\mu}{A}\int_{\mathbb{T} \times \mathbb{I} \times \mathbb{T}}n^{2p+1}dxdydz \\
			\leq&\frac{2(2p-1)}{A}\|n^p\nabla c\|_{L^2}\|\nabla n^p\|_{L^2} 
			\leq\frac{2p-1}{Ap}\|\nabla n^p\|_{L^2}^2
			+\frac{(2p-1)p}{A}\|n^p\nabla c\|^2_{L^2},\nonumber
		\end{aligned}
	\end{equation}
	where we use $-\frac{2p\mu}{A}\int_{\mathbb{T}^{3}}n^{2p+1}dxdydz\leq 0$ due to $n_{\rm in}\geq 0$ and $\mu\geq 0.$
	Using H\"{o}lder inequality and the Gagliardo-Nirenberg inequality, we get
	\begin{equation}
		\begin{aligned}
			\|n^p\nabla c\|_{L^2}^2
			\leq \|n^p\|_{L^4}^2\|\nabla c\|_{L^4}^2
			\leq C\|n^p\|_{L^2}^{\frac{1}{2}}\|\nabla n^p\|_{L^2}^\frac{3}{2}\|\nabla c\|_{L^4}^2, \nonumber
		\end{aligned}
	\end{equation}
	which implies that
	\begin{equation}
		\begin{aligned}
			\frac{d}{dt}\|n^p\|^2_{L^2}+\frac{1}{2A}\|\nabla n^{p}\|_{L^2}^2\leq
			\frac{Cp^8}{A}\|n^p\|_{L^2}^2\left(1+\|\nabla c\|_{L^4}^8 \right).
			\label{58}
		\end{aligned}
	\end{equation}
		Thanks to the Gagliardo-Nirenberg inequality 
	\begin{equation}
		\begin{aligned}
			\|n^p\|_{L^2}
			\leq C\|n^p\|_{L^1}^{\frac{2}{5}}\|\nabla n^p\|_{L^2}^{\frac{3}{5}},
			\nonumber
		\end{aligned}
	\end{equation}
	we infer from (\ref{58}) that
	\begin{equation}
		\begin{aligned}
			\frac{d}{dt}\|n^p\|^2_{L^2}
			\leq-\frac{\|n^p\|_{L^2}^{\frac{10}{3}}}{2AC\|n^p\|_{L^1}^{\frac{4}{3}}}
			+\frac{Cp^8}{A}\|n^p\|_{L^2}^2\left(1+\|\nabla c\|_{L^{\infty}L^4}^8 \right).
			\nonumber
		\end{aligned}
	\end{equation}
	Applying Lemma \ref{lem:the zero mode of c and n}, Lemma \ref{lem:est of the non-zero mode of c and n}, Corollary \ref{Cor:E1 end}  and Corollary \ref{Cor:E2 end}, there holds
	\begin{equation*}
		\begin{aligned}
			\|\nabla c\|_{L^{\infty}L^4}\leq& \|\nabla c_{\neq}\|_{L^{\infty}L^4}+\|\nabla c_0\|_{L^{\infty}L^4}
			\\\leq& C\left(\|\partial_xn_{\neq}\|_{L^{\infty}L^{2}}+\|n_{0}\|_{L^{\infty}L^{2}} \right) \leq C\left(F_{1}+F_{2} \right).
		\end{aligned}
	\end{equation*}
	Therefore, we get that
	\begin{equation*}\label{np}
		\begin{aligned}
			\frac{d}{dt}\|n^p\|^2_{L^2}
			\leq-\frac{\|n^p\|^{\frac{10}{3}}_{L^2}}{2CA\|n^p\|_{L^1}^{\frac{4}{3}}}
			+\frac{Cp^8}{A}\|n^p\|^2_{L^2}(1+F_{1}^8+F_{2}^8),
		\end{aligned}
	\end{equation*}
	which indicates that 
	\begin{equation}
		\begin{aligned}
			\sup_{t\geq 0}\|n^p\|_{L^2}^2\leq \max\big\{8C^3(1+F_{1}^{8}+F_{2}^{8})^{\frac32}p^{12}\sup_{t\geq0}\|n^p\|^2_{L^{1}}, 2\|n_{\rm in}^p\|_{L^2}^2\big\}\label{n1_1}.
		\end{aligned}
	\end{equation}
		Next, the Moser-Alikakos iteration is used to determine $F_{3}$. Recalling that $ p=2^{j} $ with $ j\geq 1, $
	    we rewrite (\ref{n1_1}) into
	\begin{equation}\label{n1_2}
		\begin{aligned}
			&\sup_{t\geq 0}\int_{\mathbb{T}\times\mathbb{I}\times\mathbb{T}}|n(t)|^{2^{j+1}}dxdydz 
			\\\leq& \max\Big\{H_{1}p^{12}\Big(\sup_{t\geq0}\int_{\mathbb{T}\times\mathbb{I}\times\mathbb{T}}|n(t)|^{2^{j}}dxdydz\Big)^2, 2\int_{\mathbb{T}\times\mathbb{I}\times\mathbb{T}}|n_{\rm in}|^{2^{j+1}}dxdydz\Big\},
		\end{aligned}
	\end{equation}
	where $H_{1}=8C^3(1+F_{1}^{8}+F_{2}^{8})^{\frac32}.$
	By Corollary \ref{Cor:E1 end}, we note that
	\begin{equation*}
		\|n_{0}\|_{L^{\infty}L^{2}}\leq F_{1}.
	\end{equation*}
	Hence, one obtains that 
	$$\sup_{t\geq0}\|n(t)\|_{L^2}\leq|\mathbb{T}| \|n_{0}\|_{L^\infty L^2}+\|n_{\neq}\|_{L^\infty L^2}\leq |\mathbb{T}|F_1+F_2.$$
	By interpolation inequality, 	for $j\geq1$, we have
	$$\|n_{\rm in}\|_{L^{2^{j+1}}}\leq\|n_{\rm in}\|^{\theta_j}_{L^2}
	\|n_{\rm in}\|^{1-\theta_j}_{L^\infty}
	\leq\|n_{\rm in}\|_{L^2}+\|n_{\rm in}\|_{L^\infty}\leq|\mathbb{T}|F_1+F_2+\|n_{\rm in}\|_{L^\infty},$$
	for some $0<\theta_j<1$,
	which yields
	$$2\int_{\mathbb{T}\times\mathbb{I}\times\mathbb{T}}|n_{\rm in}|^{2^{j+1}}dxdydz
	\leq2\left(|\mathbb{T}|F_1+F_2+\|n_{\rm in}\|_{L^\infty}\right)^{2^{j+1}}\leq K^{2^{j+1}},$$
	where $K:=2(|\mathbb{T}|F_1+F_2+\|n_{\rm in}\|_{L^\infty}).$
	
	Now, we rewrite (\ref{n1_2}) as
	\begin{equation}
		\begin{aligned}
			\sup_{t\geq0}\int_{\mathbb{T}\times\mathbb{I}\times\mathbb{T}}|n(t)|^{2^{j+1}}dxdydz\leq \max\left\{H_1	4096^{j}\left(\sup_{t\geq0}\int_{\mathbb{T}\times\mathbb{I}\times\mathbb{T}}|n(t)|^{2^{j}}dxdydz\right)^2, K^{2^{j+1}} \right\}.\nonumber
		\end{aligned}
	\end{equation}
	For $j=k$, we get
	\begin{equation}
		\begin{aligned}
			\sup_{t\geq0}\int_{\mathbb{T}\times\mathbb{I}\times\mathbb{T}}|n(t)|^{2^{k+1}}dxdydz\leq H_{1}^{a_k}	4096^{b_k}K^{2^{k+1}},\nonumber
		\end{aligned}
	\end{equation}
	where $a_k=1+2a_{k-1}$ and $b_k=k+2b_{k-1}$.
	
	Generally, one can obtain the following formulas
	$$a_k=2^k-1,\ {\rm and}\ \ b_k=2^{k+1}-k-2.$$
	Therefore, we deduce that
	\begin{equation}
		\begin{aligned}
			\sup_{t\geq0}\left(\int_{\mathbb{T}\times\mathbb{I}\times\mathbb{T}}|n(t)|^{2^{k+1}}dxdydz\right)^{\frac{1}{2^{k+1}}}\leq H_1^{\frac{2^k-1}{2^{k+1}}}	4096^{\frac{2^{k+1}-k-2}{2^{k+1}}}K.\nonumber
		\end{aligned}
	\end{equation}
	Letting $k\rightarrow\infty$, it holds that
	$$\sup_{t\geq0}\|n(t)\|_{L^\infty}\leq C(1+F_{1}^{8}+F_{2}^{8})^{\frac{3}{4}}(|\mathbb{T}|F_1+F_2+\|n_{\rm in}\|_{L^\infty})=:F_3,$$
	which completes the proof.
\end{proof}

\appendix 
\section{}

For a given function $f=f(x, y, z)$, by Fourier series it holds
\begin{equation}\label{Sob}
	\begin{aligned}
		\|f_{\neq}\|_{L^2}^2 \leq\|\partial_x^j f_{\neq}\|_{L^2}^2, 
		\quad\|\partial_x f\|_{L^2}^2 \leq\|\partial_x^j f\|_{L^2}^2 \quad \text{and}\quad 
		\|\partial_z f\|_{L^2}^2 \leq\|\partial_z^j f\|_{L^2}^2,
	\end{aligned}
\end{equation}
where $j$ is a positive integer with $j \geq 1$.

The following lemma gives some relationships between origin velocity $u_{\neq}$, new vorticity $\omega_2=\partial_z u_1-\partial_x u_3$ and new velocity $\Delta u_2$, which will be frequently used in calculations.
\begin{Lem}\label{lemma_u}
	For $j\in\{1,3\},$ it holds
	\begin{equation}\begin{aligned} \label{eq:velocity transform}
			& \|\left(\partial_x, \partial_z\right) u_{\neq}\|_{L^2} \leq C\left(\|\omega_{2, \neq}\|_{L^2}+\|\nabla u_{2, \neq}\|_{L^2}\right), \\
			& \|\left(\partial_x, \partial_z\right) \partial_x u_{\neq}\|_{L^2} \leq C\left(\|\partial_x \omega_{2, \neq}\|_{L^2}+\|\partial_{x} \nabla u_{2, \neq}\|_{L^2}\right), \\
			& \|\left(\partial_x, \partial_z\right) \partial_y u_{\neq}\|_{L^2} \leq C\left(\|\partial_y \omega_{2, \neq}\|_{L^2}+\|\Delta u_{2, \neq}\|_{L^2}\right), \\
			& \|\left(\partial_x, \partial_z\right) \partial_j \nabla u_{\neq}\|_{L^2} \leq C\left(\|\partial_j \nabla \omega_{2, \neq}\|_{L^2}+\|\partial_{j} \Delta u_{2, \neq}\|_{L^2}\right), \\
			& \|(\partial_x^2, \partial_z^2) u_{3, \neq}\|_{L^2} \leq C\left(\|\partial_x \omega_{2, \neq}\|_{L^2}
			+\|\partial_z \nabla u_{2, \neq}\|_{L^2}\right).
	\end{aligned}\end{equation}
\end{Lem}
\begin{proof}
	The first three results can be found in Lemma 3.13 in \cite{CWW2025}. 
	Due to $\partial_{y}u_{2,\neq}\big|_{y=\pm 1}=0,$ 
	using integration by parts, we get
	\begin{equation*}
		\|\nabla^{2}\partial_{j}u_{2,\neq}\|_{L^{2}}^{2}\leq C\|\partial_{j}\Delta u_{2,\neq}\|_{L^{2}}^{2}.
	\end{equation*}
	The above inequality along with
	$
	\|(\partial_{x},\partial_{z})(f_{1}, f_{2})\|_{L^{2}}^{2}=\|(\partial_{z}f_{1}-\partial_{x}f_{2}, \partial_{x}f_{1}+\partial_{z}f_{2})\|_{L^{2}}^{2}
	$
	implies that
	\begin{equation*}
		\begin{aligned}
			\|(\partial_{x},\partial_{z})\partial_{j}\nabla u_{\neq}\|_{L^{2}}^{2}\leq&\|(\partial_{x},\partial_{z})\partial_{j}\nabla u_{2,\neq}\|_{L^{2}}^{2}+\|(\partial_{x},\partial_{z})\nabla\partial_{j}(u_{1,\neq}, u_{3,\neq})\|_{L^{2}}^{2}\\\leq&C\left(\|\partial_{j}\nabla\omega_{2,  \neq}\|_{L^{2}}^{2}+\|\partial_{j}\Delta u_{2,\neq}\|_{L^{2}}^{2} \right),
		\end{aligned}
	\end{equation*}
	which gives $(\ref{eq:velocity transform})_{4}.$ Moreover, $(\ref{eq:velocity transform})_{5}$ follows from $\partial_{x}\omega_{2,  \neq}=-\partial_{z}\partial_{y}u_{2,
		\neq}-(\partial_{x}^{2}+\partial_{z}^{2})u_{3,
		\neq}.$ 
\end{proof}

Inspired by \cite{CWW2025} and \cite{CWW20251}, we give some $L^{\infty}$-norm Sobolev embeddings in domain $\mathbb{T}\times\mathbb{I}\times\mathbb{T}$. 
\begin{Lem}[Lemma 3.1 in \cite{CWW2025} and Lemma A.2 in \cite{CWW20251}]
	For a given function $f_0=\frac{1}{|\mathbb{T}|} \int_{\mathbb{T}} f(t, x, y, z) d x$ and a constant $\alpha \in\left(\frac{1}{2}, 1\right]$,
	 we have
	\begin{equation} \label{eq:0_norm_L_infty}
		\begin{aligned}
			& \|f_0\|_{L^{\infty}} \leq C\big(\|\partial_y f_0\|_{L^2}^{\frac{1}{2}}\|f_0\|_{L^2}^{\frac{1}{2}}
			+\|\partial_z\nabla f_0\|_{L^2}^{\frac{1}{2}}\|\partial_z f_0\|_{L^2}^{\alpha-\frac{1}{2}}\|f_0\|_{L^2}^{1-\alpha}+\|f_{0}\|_{L^{2}}\big) ,\\
			& \|f_0\|_{L^{\infty}} \leq C\big(\|\partial_y f_0\|_{L^2}^{\frac{1}{2}}\|f_0\|_{L^2}^{\frac{1}{2}}
			+\|\partial_z\nabla f_0\|_{L^2}^{\alpha-\frac{1}{2}}\|\partial_z f_0\|_{L^2}^{\frac{1}{2}}\|\partial_y f_0\|_{L^2}^{1-\alpha}+\|f_{0}\|_{L^{2}}\big) ,\\
			& \|f_0\|_{L_z^{\infty} L_y^2} \leq C\left(\|f_0\|_{L^2}+\|\partial_z f_0\|_{L^2}^\alpha\|f_0\|_{L^2}^{1-\alpha}\right) ,\\
			& \|f_0\|_{L_y^{\infty} L_z^2} \leq C\left(\|f_{0}\|_{L^{2}}+\|\partial_y f_0\|_{L^2}^{\frac{1}{2}}\|f_0\|_{L^2}^{\frac{1}{2}} \right).
		\end{aligned}
	\end{equation}

		Specially, when  $f_0|_{y=\pm 1} = 0 $, we rewrite $\eqref{eq:0_norm_L_infty}_{1,2,4}$ as
	\begin{equation}\label{Sob f0}
		\begin{aligned}
				& \|f_0\|_{L^{\infty}} \leq C\big(\|\partial_y f_0\|_{L^2}^{\frac{1}{2}}\|f_0\|_{L^2}^{\frac{1}{2}}
			+\|\partial_z\partial_{y} f_0\|_{L^2}^{\frac{1}{2}}\|\partial_z f_0\|_{L^2}^{\alpha-\frac{1}{2}}\|f_0\|_{L^2}^{1-\alpha}\big) ,\\
			& \|f_0\|_{L^{\infty}} \leq C\big(\|\partial_y f_0\|_{L^2}^{\frac{1}{2}}\|f_0\|_{L^2}^{\frac{1}{2}}
			+\|\partial_z\partial_{y} f_0\|_{L^2}^{\alpha-\frac{1}{2}}\|\partial_z f_0\|_{L^2}^{\frac{1}{2}}\|\partial_y f_0\|_{L^2}^{1-\alpha}\big) ,\\
			& \|f_0\|_{L_y^{\infty} L_z^2} \leq C\|\partial_y f_0\|_{L^2}^{\frac{1}{2}}\|f_0\|_{L^2}^{\frac{1}{2}}.
		\end{aligned}
	\end{equation}
\end{Lem}

\begin{Lem}[Lemma 3.2 in \cite{CWW2025} and Lemma A.3 in \cite{CWW20251}] \label{eq:neq infty form}
	For a given function $g(x,y,z)=g_{\neq}(x,y,z)$,
	we have
	\begin{equation}\label{Sob neq}
		\begin{aligned}
			& \|g\|_{L^{\infty}} \leq C\big(\|\partial_z\nabla g\|_{L^2}^{\frac{1}{2}}\|\partial_x \partial_z g\|_{L^2}^{\alpha-\frac{1}{2}}
			\|\partial_x^2 g\|_{L^2}^{\alpha-\frac{1}{2}}\|\partial_x g\|_{L^2}^{\frac{3}{2}-2 \alpha}+\|\partial_x \nabla g\|_{L^2}^{\frac{1}{2}}
			\|\partial_x g\|_{L^2}^{\alpha-\frac{1}{2}}\|g\|_{L^2}^{1-\alpha}\big), \\
			& \|g\|_{L_{y, z}^{\infty} L_x^2} \leq C\big(\|\partial_yg\|^{\frac{1}{2}}_{L^2}
			\|g\|^{\frac{1}{2}}_{L^2}
			+\|\partial_zg\|^{\frac{1}{2}}_{L^2}
			\|\partial_z\nabla g\|^{\alpha-\frac{1}{2}}_{L^2}
			\|\partial_yg\|^{1-\alpha}_{L^2}+\|g\|_{L^{2}}\big), \\
			& \|g\|_{L_{x, y}^{\infty} L_z^2} \leq C\big(\|\partial_x g\|_{L^2}^{\frac{1}{2}}\|\partial_x \partial_y g\|_{L^2}^{\alpha-\frac{1}{2}}
			\|\partial_y g\|_{L^2}^{1-\alpha}+\|\partial_{x}g\|_{L^{2}} \big), \\
			& \|g\|_{L_{x, z}^{\infty} L_y^2} \leq C\big(\|\partial_x g\|_{L^2}^\alpha\|g\|_{L^2}^{1-\alpha}+\|\partial_x \partial_z g\|_{L^2}^\alpha\|g\|_{L^2}^{1-\alpha}\big), \\
			& \|g\|_{L_x^{\infty} L_{y, z}^2} \leq C\|\partial_x g\|_{L^2}^\alpha\|g\|_{L^2}^{1-\alpha}, \\
			& \|g\|_{L_z^{\infty} L_{y, x}^2} \leq C\left(\|g\|_{L^2}+\|\partial_z g\|_{L^2}^\alpha\|g\|_{L^2}^{1-\alpha}\right), \\
			& \|g\|_{L_y^{\infty} L_{x, z}^2} \leq C\big(\|g\|_{L^{2}}+\|\partial_y g\|_{L^2}^{\frac{1}{2}}\|g\|_{L^2}^{\frac{1}{2}} \big),
		\end{aligned}
	\end{equation}
	where $\alpha \in\left(\frac{1}{2}, \frac{3}{4}\right]$ for $\eqref{Sob neq}_1$ and $\alpha \in\left(\frac{1}{2}, 1\right]$ for the others.
	
	Specially, when $g_{\neq}|_{y=\pm 1} = 0$, we rewrite $\eqref{Sob neq}_{1,2,3,7}$ as
		\begin{equation} \label{Sob g neq}
		\begin{aligned}
			& \|g\|_{L^{\infty}} \leq C\big(\|\partial_z\partial_{y} g\|_{L^2}^{\frac{1}{2}}\|\partial_x \partial_z g\|_{L^2}^{\alpha-\frac{1}{2}}
			\|\partial_x^2 g\|_{L^2}^{\alpha-\frac{1}{2}}\|\partial_x g\|_{L^2}^{\frac{3}{2}-2 \alpha}+\|\partial_x \partial_{y} g\|_{L^2}^{\frac{1}{2}}
			\|\partial_x g\|_{L^2}^{\alpha-\frac{1}{2}}\|g\|_{L^2}^{1-\alpha}\big), \\
			& \|g\|_{L_{y, z}^{\infty} L_x^2} \leq C\big(\|\partial_yg\|^{\frac{1}{2}}_{L^2}
			\|g\|^{\frac{1}{2}}_{L^2}
			+\|\partial_zg\|^{\frac{1}{2}}_{L^2}
			\|\partial_z\partial_{y} g\|^{\alpha-\frac{1}{2}}_{L^2}
			\|\partial_yg\|^{1-\alpha}_{L^2}\big),\\
			& \|g\|_{L_{x, y}^{\infty} L_z^2} \leq C\|\partial_x g\|_{L^2}^{\frac{1}{2}}\|\partial_x \partial_y g\|_{L^2}^{\alpha-\frac{1}{2}}
			\|\partial_y g\|_{L^2}^{1-\alpha}, \\
			& \|g\|_{L_y^{\infty} L_{x, z}^2} \leq C\|\partial_y g\|_{L^2}^{\frac{1}{2}}\|g\|_{L^2}^{\frac{1}{2}},
		\end{aligned}
	\end{equation}
	where $\alpha \in\left(\frac{1}{2}, \frac{3}{4}\right]$ for $\eqref{Sob g neq}_1$ and $\alpha \in\left(\frac{1}{2}, 1\right]$ for the others.
\end{Lem}

\begin{Lem}[Lemma A.2 in \cite{CWW1}] \label{lem:est of 0neq L^infty}
	Let $f(y, z)$ be a function such that $f_{(0, \neq)} \in H^2(\mathbb{I} \times \mathbb{T})$ , then there holds
	$$
	\|f_{(0, \neq)}\|_{L^{\infty}} \leq C\|\nabla f_{(0, \neq)}\|_{L^2}^{1-\tau}\|\partial_z \nabla f_{(0, \neq)}\|_{L^2}^\tau,
	$$
	where $\tau \in(0,1]$.
\end{Lem}

\begin{Lem}[Lemma 3.6 in \cite{CWW2025}] \label{lem:est of 00}
	It holds that
	\begin{equation*}\begin{aligned} \label{eq:est of 00_1}
			\|(f g)_{(0,0)}\|_{L^2} \leq C\|f\|_{L^2}\|g\|_{L^2}^{\frac{1}{2}}\|\partial_y g\|_{L^2}^{\frac{1}{2}}
	\end{aligned}\end{equation*}
	for $\lt. g\rt|_{y=\pm1}  =  0$ or
	\begin{equation*}\begin{aligned} \label{eq:est of 00_2}
			\|(f g)_{(0,0)}\|_{L^2} \leq C\|f\|_{L^2}\|g\|_{L^2}^{\frac{1}{2}}\|\partial_y g\|_{L^2}^{\frac{1}{2}} +  C\|f\|_{L^2}\|g\|_{L^2}.
	\end{aligned}\end{equation*}
\end{Lem}
\begin{Lem}[Lemma A.7 in \cite{CWZ1}] \label{lem:transform for p}
	If $f$ is a function in $[-1,1] \times \mathbb{T}$, i.e., $f=f(y, z)$, then we have 
	$$
	\|\partial_y^2 f\|_{L^2}+\|\partial_z^2 f\|_{L^2} \leq C\left(\|\partial_z \partial_y f\|_{L^2}+\|\Delta f\|_{L^2}\right).
	$$
\end{Lem}

\subsection{Elliptic Estimates}
\begin{Lem}[Lemma 3.1 in \cite{CWW1}] \label{lem:the zero mode of c and n}
	Let $c_0$ and $n_0$ be the zero mode of $c$ and $n$, respectively, satisfying
	$$
	-\Delta c_0+c_0=n_0,\left.\quad c_0\right|_{y= \pm 1}=0.
	$$
	Then there hold
	$$\|\Delta c_0(t)\|_{L^2}^2   +  2\|\nabla c_0(t)\|_{L^2}^2  + \|c_0(t)\|_{L^2}^2 = \|n_0(t)\|_{L^2}^2$$
	and
	$$\|\nabla c_0(t)\|_{L^4} \leq C\|n_0(t)\|_{L^2},$$
	for any $t \geq 0$.
\end{Lem}
\begin{Lem}[Lemma 3.2 in \cite{CWW1} or Lemma 3.9 in \cite{CWW2025}] \label{lem:est of the non-zero mode of c and n}
	Let $c_{\neq}$ and $n_{\neq}$ be the non-zero mode of $c$ and $n$, respectively, satisfying
	$$-\Delta c_{\neq}+c_{\neq}=n_{\neq},\left.\quad c_{\neq}\right|_{y= \pm 1}=0.$$
	Then there holds
	\begin{equation}\label{ellip c}
		\begin{aligned}
			\|\partial_{x}^{j} \pz^i \Delta c_{\neq}(t)\|_{L^2}+\|\partial_{x}^{j} \pz^i \nabla c_{\neq}(t)\|_{L^2} &\leq C\|\partial_{x}^{j} \pz^i n_{\neq}(t)\|_{L^2},\\
			\|\partial_{x}^{j}\pz^i \nabla c_{\neq}(t)\|_{L^4} &\leq C\|\partial_{x}^{j} \pz^i n_{\neq}(t)\|_{L^2}
		\end{aligned}
	\end{equation}
	for any $t \geq 0$, where $i, j\in\{0, 1, 2\}$.
\end{Lem}

\begin{Lem} \label{lem:c00}
	It holds that
	$$
	\|\partial_{y}^2 c_{(0,0)}(t)\|_{L^2}^2+2\|\partial_y c_{(0,0)}(t)\|_{L^2}^2+\|c_{(0,0)}(t)\|_{L^2}^2=\|n_{(0,0)}(t)\|_{L^2}^2,
	$$
	$$
	\|\Delta c_{(0, \neq)}(t)\|_{L^2}^2+2\|\nabla c_{(0, \neq)}(t)\|_{L^2}^2  
	+\|c_{(0, \neq)}(t)\|_{L^2}^2=\|n_{(0, \neq)}(t)\|_{L^2}^2
	$$
	and
	$$
	\|\partial_y c_{(0,0)}(t)\|_{L^{\infty}} \leq C\|n_{(0,0)}(t)\|_{L^2}.
	$$
\end{Lem}
\begin{proof}
	Direct calculations yield that
	\begin{equation*}
		\begin{aligned}
			&\|\partial_{y}^2 c_{(0,0)}(t)\|_{L^2}^2+2\|\partial_y c_{(0,0)}(t)\|_{L^2}^2+\|c_{(0,0)}(t)\|_{L^2}^2=\|n_{(0,0)}(t)\|_{L^2}^2,\\
			&\|\Delta c_{(0, \neq)}(t)\|_{L^2}^2+2\|\nabla c_{(0, \neq)}(t)\|_{L^2}^2  
			+\|c_{(0, \neq)}(t)\|_{L^2}^2=\|n_{(0, \neq)}(t)\|_{L^2}^2.
		\end{aligned}
	\end{equation*}
By the Gagliardo-Nirenberg inequality, one obtains
	$$
	\|\partial_y c_{(0,0)}(t)\|_{L^{\infty}} 
	\leq C\big(\|\partial_y c_{(0,0)}(t)\|_{L^2}^\frac12 \|\partial_{y}^2 c_{(0,0)}(t)\|_{L^2}^\frac12  + \|\partial_y c_{(0,0)}(t)\|_{L^2}  \big)
	\leq C\|n_{(0,0)}(t)\|_{L^2}.
	$$
\end{proof}

\subsection{Time-space estimates}
\begin{Prop}[Theorem 1.1 in \cite{CWZ1}]\label{prop:key prop for omega2neq}
	Let $ f $ be a solution of 
	\begin{equation*}
		\left\{
		\begin{array}{lr}
			\partial_tf-\frac{1}{A}\left(\partial_{y}^{2}-\eta^{2} \right)f+ik_{1}yf=-ik_{1}f_{1}-\partial_{y}f_{2}-ik_{3}f_{3}-f_{4}, \\
			f|_{y=\pm 1}=0,~~f|_{t=0}=f_{\rm in}, 
		\end{array}
		\right.
	\end{equation*}
	with $ f_{4}(t, \pm 1)=0 $ and $ f_{\rm in}(\pm 1)=0. $ Then it holds
	\begin{equation*}
		\begin{aligned}
			&\|{\rm e}^{aA^{-\frac13}t}f\|_{L^{\infty}L^{2}}^{2}+\frac{1}{A}\|{\rm e}^{aA^{-\frac13}t}\partial_{y}f\|_{L^{2}L^{2}}^{2}+\big(A^{-1}\eta^{2}+(A^{-1}k_{1}^{2})^{\frac13} \big)\|{\rm e}^{aA^{-\frac13}t}f\|_{L^{2}L^{2}}^{2}\\\leq&C\big(\|f_{\rm in}\|_{L^{2}}^{2}+A\|{\rm e}^{aA^{-\frac13}t}f_{2}\|_{L^{2}L^{2}}^{2}+(\eta|k_{1}|)^{-1}\|{\rm e}^{aA^{-\frac13}t}\partial_{y}f_{4}\|_{L^{2}L^{2}}^{2}+\eta|k_{1}|^{-1}\|{\rm e}^{aA^{-\frac13}t}f_{4}\|_{L^{2}L^{2}}^{2}\\&+\min\{(A^{-1}\eta^{2})^{-1}, (A^{-1}k_{1}^{2})^{-\frac13} \}\|{\rm e}^{aA^{-\frac13}t}(k_{1}f_{1}+k_{3}f_{3} )\|_{L^{2}L^{2}}^{2} \big),
		\end{aligned}
	\end{equation*}	
	where ``a'' is a non-negative constant, and $  f_{1}, f_{2}, f_{3}, f_{4} $ are given functions.	
\end{Prop}
\begin{Prop}[Theorem 1.2 in \cite{CWZ1}]\label{prop:key prop for delta u2neq}
	Let $f$	be a solution of
	\begin{equation*}
		\left\{
		\begin{array}{lr}
			\partial_tf-\frac{1}{A}\left(\partial_{y}^{2}-\eta^{2} \right)f+ik_{1}yf=ik_{1}f_{1}+\partial_{y}f_{2}+ik_{3}f_{3}, \\
			\left(\partial_{y}^{2}-\eta^{2} \right)\varphi=f,\quad \partial_{y}\varphi|_{y=\pm 1}=\varphi|_{y=\pm 1}=0, \\
			f|_{t=0}=f_{\rm in}, 
		\end{array}
		\right.
	\end{equation*}
	with $\partial_{y}\varphi_{\rm in}|_{y=\pm 1}=0$. Then it holds
	\begin{equation*}
		\begin{aligned}
			&|k_{1}\eta|^{\frac12}\|{\rm e}^{aA^{-\frac13}t}(\partial_{y}, \eta)\varphi\|_{L^{2}L^{2}}+A^{-\frac34}\|{\rm e}^{aA^{-\frac13}t}\partial_{y}f\|_{L^{2}L^{2}}+A^{-\frac12}\eta\|{\rm e}^{aA^{-\frac13}t}f\|_{L^{2}L^{2}}\\&+\eta\|{\rm e}^{aA^{-\frac13}t}(\partial_{y},\eta)\varphi\|_{L^{\infty}L^{2}}+A^{-\frac14}\|{\rm e}^{aA^{-\frac13}t}f\|_{L^{\infty}L^{2}}\\\leq&CA^{\frac12}\|{\rm e}^{aA^{-\frac13}t}(f_{1}, f_{2}, f_{3})\|_{L^{2}L^{2}}+C\left(\eta^{-1}\|\partial_{y}f_{\rm in}\|_{L^{2}}+\|f_{\rm in}\|_{L^{2}} \right),	
		\end{aligned}
	\end{equation*}
	where ``a'' is a non-negative constant, and $  f_{1}, f_{2}, f_{3} $ are given functions.
\end{Prop}

\section*{Acknowledgement}
W. Wang was supported by National Key R\&D Program of China (No.2023YFA1009200) and NSFC under grant 12471219 and 12071054.  The research of J. Wei is partially supported by GRF from RGC of Hong Kong
entitled ``New frontiers in singularity formations in nonlinear partial differential equations".

\end{document}